\documentclass[final,a4paper, 10pt]{article}

\usepackage[a4paper,bottom=20mm,scale=0.768]{geometry}

\usepackage[english]{babel}

\usepackage[mathscr]{euscript}
\usepackage{color}
\usepackage{fancyhdr}
\usepackage{bm}
\usepackage{hyperref}
\usepackage{cite}
\usepackage{showkeys}
\usepackage{subfigure}
\usepackage{float}
\usepackage{enumitem}

\usepackage{amsfonts}
\usepackage{amsmath}
\usepackage{amssymb}
\usepackage{amscd}
\usepackage{amsthm}
\usepackage{tikz}
\usepackage{tikz-3dplot}
\usetikzlibrary{shapes.geometric, arrows,3d,positioning}

\usepackage{accents}
\usepackage{graphicx}
\usepackage{array}

\numberwithin{equation}{section}

\newtheorem{theorem}{Theorem}[section]
\newtheorem{definition}[theorem]{Definition}
\newtheorem{lemma}[theorem]{Lemma}
\newtheorem{corollary}[theorem]{Corollary}

\newtheorem{remark}[theorem]{Remark}

\newcommand*{\N}{\ensuremath{\mathbb{N}}}
\newcommand*{\Z}{\ensuremath{\mathbb{Z}}}
\newcommand*{\R}{\ensuremath{\mathbb{R}}}

\renewcommand{\i}{\mathrm{i}}
\renewcommand{\epsilon}{{\varepsilon}}

\renewcommand{\d}[1]{\,\mathrm{d}#1 \,}

\newcommand{\F}{\mathcal{F}} % Fourier transform
\newcommand{\J}{\mathcal{J}} % Bloch transform
\newcommand{\p}{{\mathrm{per}}}
\newcommand{\X}{{\mathcal{X}}}

\renewcommand{\S}{\mathcal{S}}

\renewcommand{\tilde}{\widetilde}

\newlength{\dhatheight}

\usepackage{color}
\newcommand{\high}[1] {{\color{red}{#1}}}

\begin{document}

\sloppy

\title{The radiation condition for Helmholtz equations above (locally perturbed) periodic surfaces}
\author{Ruming Zhang\thanks{Institute for Mathematics, Technische Universit{\"a}t Berlin, Berlin, Germany; \texttt{ruming.zhang@tu-berlin.de}}
}
\date{}
\maketitle

\abstract{
The radiation condition is the key question in the mathematical modelling for scattering problems in unbounded domains. Mathematically, it plays the role as the "boundary condition" at the infinity, which guarantees the well-posedness of the mathematical problem; physically, it describes the far-field asymptotic behaviour of the physical waves. In this paper, we focus on the radiation conditions for scattering problems above (locally perturbed) periodic surfaces. According to Hu et al. (2021), the radiating solution satisfies the Sommerfeld radiation condition:
\[
\frac{\partial u}{\partial r}-\i k u=o\left(r^{-1/2}\right),\quad r\rightarrow\infty.
\]
Although there are literature which have studied this problem, there is no specific method for dealing with periodic structures. Due to this reason, the important properties for the periodic structures may be ignored. Moreover, the existing method is not extendable to bi-periodic structures in three dimensional spaces. 

In this paper, we study the radiation condition for the time-harmonic scattering problem with periodic surfaces, which is modelled by the Helmholtz equation. We introduce a novel method based on the Floquet-Bloch transform, which, to the best of the author’s knowledge, is the first method that works particularly for periodic media. With this method, we improve the Sommerfeld radiation condition for the scattered field from periodic media to:
\[
\frac{\partial u}{\partial r}-\i k u=O\left(r^{-3/2}\right),\quad r\rightarrow\infty.
\]
 More importantly, the prospect of extending this method to 3D cases is optimistic.
}

\section{Introduction}

For the mathematical modelling and analysis of the scattering problems with unbounded domains, the radiation condition is always the key topic. From the mathematical point of view, we need a "boundary condition" at the infinity, to formulate a well-posed mathemtaical problem to describe the physical process. From the physical point of view, a proper radiation condition describes the physical solution explicitly, such that the mathematical model coincides with the physical phenomenan in the real world. For the easiest case, i.e., time-harmonic acoustic waves scattered by a bounded obstacle, the scattered field satisfies the well-known Sommerfeld radiation condition (see \cite{Somme1957}). Similarly, for electromagnetic waves the radiation condition is know as the Silver-M{\"u}ller condition. For details we refer to \cite{Colto2013}.

In this paper, we consider a much more complex case, i.e., time-harmonic acoustic waves scattered by (locally perturbed) periodic surfaces. The scattering problems in periodic structures have a large area of applications in, for example, photonic crystals. We refer to \cite{Petit1980} for details. The problem is modelled by the Helmholtz equation, where the domain is a periodic half-space.  Compared to the bounded obstacles, this problem is more difficult due to the unbounded domains and complex singularities. If we ignore the periodicity of the structure but simply treat it as a infinite layer or interface, then the theory of rough surface/layer scattering problems can be adopted. Roughly speaking, the theory is based on the integral equation representation for the solutions, and detailed analysis on the integrals in infinite curves is carried out to study the properties of the radiation conditions.  We refer to \cite{Arens2005a,Chand1996,Chand1996a,Chand1998,Chand1999a,Chand1999} for a series of earlier work  related to this topic.  Here we also mention that, for the special step-like surfaces, a shaper condition is obtained, see \cite{Lu2021}. Based on the method, Hu, Lu and Rathsfeld studied the radiation condition for the Dirichlet periodic surface scattering problems.   They proved that the scattered field satisfies the half-plane Sommerfeld radiation condition (HPSRC) in a half space above the periodic surface:
\[
\sup_{x_2> h}r^{1/2}\left|\frac{\partial u}{\partial r}(x)-\i k u(x)\right|\rightarrow 0,\quad \sup_{x_2> h}r^{1/2}\left|u(x)\right|<\infty,\quad |x|=r\rightarrow\infty.
\]
For details we refer to  \cite{Hu2021}.
  The general case, i.e., scattering problems with periodic inhomogeneous layers, it is more complex due to the existence of guided waves (see \cite{Bonne1994}). Based on his previous work as well as the work in \cite{Hu2021}, Kirsch proved the radiation condition for the radiating part in \cite{Kirsc2022}. At the end of that paper, he also gave the asymptotic  behaviour of the radiating part. However, due to technical difficulties, the extension of this kind of problems to bi-periodic structures in three dimensional spaces is not realistic.
 
We can clearly observe that, in all the mentioned work, the analysis is based on the rough surface/layer scattering problem, instead of the periodic surface/layer scattering problems. It implies that some major properties of the periodic structures may be ignored in the analysis, and the results are expected to be improved if specific methods are developed. Motivated by this, we intend to introduce a novel method, which is, to the best of the author's knowledge, the first one to treat particularly the radiation conditions for the periodic structures. We will utilize the properties for the periodic problems and obtain a shaper radiation condition compared to HPSRC introduced by Hu, Lu and Rathsfeld in \cite{Hu2021}:
\[
\sup_{x_2> h}r^{3/2}\left|\frac{\partial u}{\partial r}(x)-\i k u(x)\right|<\infty,\quad \sup_{x_2> h}r^{1/2}\left|u(x)\right|<\infty,\quad |x|=r\rightarrow\infty.
\]
 More importantly, this approach has a great potential to be extended to bi-periodic structured embedded in three-dimensional spaces, which is impossible for the previous methods.

An important tool to treat periodic structures is the so-called Floquet-Bloch transform. We refer to \cite{Hoang2011,Fliss2015,Kirsc2017a,Zhang2019a} for the analysis of the radiation conditions for scattering problems in closed periodic waveguides. For scattering problems with (locally perturbed) periodic surfaces, we refer to \cite{Coatl2012,Lechl2015e,Hadda2016,Lechl2016,Lechl2016a,Lechl2017,Zhang2017e} for the applications in numerical simulations. For more general open waveguide problems, i.e., with the existence of guided waves, we refer to \cite{Kirsc2017} for detailed discussions.  In this paper, we will also apply the Floquet-Bloch transform to the  scattering problem. The original problem is then written into a equivalent family of quasi-periodic problems, and the solution has the form of  the integral of the quasi-periodic functions with respect to the quasi-periodicity parameter. With the integral representation, the scattered field is decomposed into an evanescent wave and a Herglotz wave function with particular density function. Thus we only need to investigate the radiation conditions for both of the functions. Technically, the asymptotic behaviours of some parameter dependent integrals will be required for the investigation.

The rest of the paper will be organized as follow. In Section \ref{sec:model}, the mathematical model for the problem with purely periodic surface will be given and we will also give a brief review of some important notations and results. The Floquet-Bloch transform and its applications will be described in Section \ref{sec:fb}. For the convenience of the readers, we will present {\bf the main theorem (Theorem \ref{th:SRC}) and also the sketch of the proof in Section \ref{th:skt}}. The proof of Part (A) in Theorem \ref{th:SRC} we refer to Section \ref{sec:decay_uj}. The proof of Part (B) is much more challenging. We prove three sub-topics in   Sections \ref{sec:rc_0}, \ref{sec:rc_1} and \ref{sec:rc_2}, and then summarize the proof in Section \ref{sec:proof_main}. Finally, we extend the proof for locally perturbed periodic surfaces in Section \ref{sec:loc_sur}.

\section{The Dirichlet scattering problems}
\label{sec:model}

Assume $\zeta$ is a bounded and $2\pi$-periodic function, then define $\Gamma$ as the graph of $\zeta$:
\[
\Gamma:=\left\{(x_1,\zeta(x_1)):\,x_1\in\R\right\}\subset\R^2.
\]
The domain $\Omega$ is the half space above $\Gamma$. Without loss of generality, we assume that
\[
0<\inf_{x_1\in\R}\{\zeta(x_1)\}\leq \sup_{x_1\in\R}\{\zeta(x_1)\}<H.
\]
Let $\Gamma^h:=\R\times\{h\}$ for any $h\in\R$, then $\Gamma^H$ is the straight line which lies above $\Gamma$. Let $\Omega^H$ be defined as the domain between $\Gamma$ and $\Gamma^H$.  Thus $\Omega^H$ is an infinite strip which is $2\pi$-periodic in $x_1$-direction. Let $U_h$ be the half space above $\Gamma^h$, i.e., $U_h:=\R\times(h,\infty)$. For the visulization we refer to Figure \ref{fig:sample}.

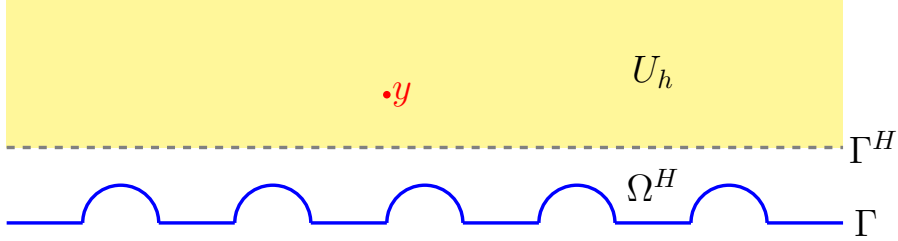
\begin{figure}[h]
\centering
\begin{tikzpicture}[scale=1, transform shape]
\draw[blue,line width=0.45mm](-5.5,0) -- (-4.5,0);
\draw[blue,line width=0.45mm](-3.5,0) arc (0:180:.5);
\draw[blue,line width=0.45mm](-3.5,0) -- (-2.5,0);
\draw[blue,line width=0.45mm](-1.5,0) arc (0:180:.5);
\draw[blue,line width=0.45mm](-1.5,0) -- (-0.5,0);
\draw[blue,line width=0.45mm](0.5,0) arc (0:180:.5);
\draw[blue,line width=0.45mm](0.5,0) -- (1.5,0);
\draw[blue,line width=0.45mm](2.5,0) arc (0:180:.5);
\draw[blue,line width=0.45mm](2.5,0) -- (3.5,0);
\draw[blue,line width=0.45mm](4.5,0) arc (0:180:.5);
\draw[blue,line width=0.45mm](4.5,0) -- (5.5,0);
\fill[yellow!50](-5.5,1) rectangle (5.5,3);
\draw[black!50, style=dashed, line width=0.45mm] (-5.5,1) -- (5.5,1);
\draw(3,0.5) node {\Large{$\Omega^H$}};
\draw(3,2.0) node {\Large{$U_h$}};
\draw(5.8,0) node {\Large{$\Gamma$}};
\draw(5.9,1) node {\Large{$\Gamma^H$}};
\fill[red] (-0.5,1.7) circle(0.05);
\draw(-0.3,1.7) node {\Large{\high{$y$}}};
\end{tikzpicture}
\label{fig:sample}
\caption{Definition of domains, surfaces and points.}
\end{figure}

Consider the following Dirichlet scattering problem above  $\Gamma$:
\begin{equation}
\label{eq:sca}
\Delta u+k^2 u=0\text{ in }\Omega;\quad u=f:=-G(x,y)\text{ on }\Gamma,
\end{equation}
where  $y=(y_1,y_2)$ is the source located above $\Gamma$ (i.e., $y_2>\zeta(y_1)$); $G(x,y)$ is the half-space Green's function given by
\[
G(x,y)=\Phi(x,y)-\Phi(x,y'):=\frac{\i}{4}H_0^{(1)}(k|x-y|)-\frac{\i}{4}H_0^{(1)}(k|x-y'|)
\]
here $H_0^{(1)}(\cdot)$ is the Hankel function of the first kind, $y'=(y_1,-y_2)$ is the reflection of $y$ with respect to $x_2=0$.  Moreover, it is know that for the radiating solution $u$ satisfies the so-called upward propagationg radiation condition (UPRC, see \cite{Chand1998}):
\begin{equation}
\label{eq:uprc}
u(x)=2\int_{\Gamma^H}\frac{\partial\Phi(x,y)}{\partial y_2}u(y)\d s(y),\quad x_2>H.
\end{equation}
The UPRC is equivalent to the {\em angular spectrum representation} (see \cite{Arens2005a}),
\begin{equation}
\label{eq:asr}
u(x_1,x_2)=\int_{\Gamma^H}e^{\i\xi x_1 +\i\sqrt{k^2-\xi^2}(x_2-H)}\hat{u}(\xi,H)\d\xi,
\end{equation}
which is equivalent to the DtN map in terms of the Fourier transform (see \cite{Chand2005}): 
\begin{equation}
\label{eq:rad}
\frac{\partial u}{\partial x_2}=\i\int_{\Gamma^H}\sqrt{k^2-\xi^2}e^{\i\xi x_1 +\i\sqrt{k^2-\xi^2}(x_2-H)}\hat{u}(\xi,H)\d\xi
\end{equation}
where $\hat{u}(\xi,H)$ is the Fourier transform of $u(\cdot,H)$.  The well-posedness of this problem is already guaranteed by the following theorem.

\begin{theorem}[Theorem 4.1, \cite{Chand2010}]
\label{th:wep}
When $\Gamma$ is a graph of a bounded function, then given $f\in H_r^{1/2}(\Gamma)$ for any fixed $|r|<1$, the problem \eqref{eq:sca},\eqref{eq:rad} has a unique solution $u\in H^1_r(\Omega^H)$.
\end{theorem}
Note that here $H^s_r(D)$  ($s,r\in\R$) is the weighted sobolev space defined in the unbounded domain $D\subset\R^2$ by
\[
H_s^r(D):=\left\{\phi:\,(1+|x|^2)^{r/2}\phi(x)\in H^s(D)\right\},
\]
where $H^s(D)$ is the classic Sobolev space.  Note that this theorem works not only for periodic surfaces, but also for more general rough surfaces with Dirichlet boundary conditions.

\begin{remark}
For the problem we are discussing, when $f$ is given by $-G(x,y)$ for a fixed $y$ above $\Gamma$, it is already known that $f$ decays at the rate of $O\left(|x_1|^{-3/2}\right)$ (see \cite{Chand1996}). Thus it is easily checked that $f\in H_r^{1/2}$ for any $r<1$. Thus $u\in H_r^1(\Omega^H)$ is the unique solution. 

The Dirichlet boundary condition here is not changable. If we choose another boundary condition, the uniqueness result in Theorem \ref{th:wep} may fail.  Moreover, we always need to restrict $\Gamma$ to be a graph of a function. 
\end{remark}

\section{The Floquet-Bloch transform and its applications}
\label{sec:fb}

\subsection{The Floquet-Bloch transform, and definitions of spaces}

An important tool for the periodic problems is the Floquet-Bloch (FB) transform. For the details of the definitions an propreties we refer to \cite{Lechl2016}. Given any function $\phi\in C_0^\infty(\R)$ (the space of smooth and compactly supported functions), we can define the one-dimensional transform as
\begin{equation}
\label{eq:def_fb}
(\J\phi)(\alpha,x)=\sum_{j\in\Z}\phi(x+2\pi j)e^{\i 2\pi \alpha j},\quad \alpha\in\Lambda:=[-1/2,1/2],\,x\in[-\pi,\pi].
\end{equation}
The transform is well-defined for any $\phi\in C_0^\infty(\R)$ since there are only finite number of non-vanishing terms in the series. But the operator  can also be extended  to more general function spaces. We give two examples of the norm of the space $H_0^r\left(\Lambda;H^s_\alpha[-\pi,\pi]\right)$ when $s=r=0$ (in this case, it is also denoted by $L^2\left(\Lambda;L^2[-\pi,\pi]\right)$) and $s\in\R$ and $r=n\in\N$:
\begin{align*}
\|\psi\|_{L^2\left(\Lambda;L^2[-\pi,\pi]\right)}&=\left[\int_\Lambda\int_{-\pi}^\pi|\psi(\alpha,x)|^2\d x\d\alpha\right]^{1/2};\\
\|\psi\|_{H_0^n\left(\Lambda;H^s[-\pi,\pi]\right)}&=\left[\sum_{\ell=0}^n\int_\Lambda\left\|\frac{\partial^\ell \psi}{\partial\alpha^\ell}(\alpha,\cdot)\right\|^2_{H^s_\alpha(-\pi,\pi)}\d\alpha\right]^{1/2}.
\end{align*}
 For the explicit definition of the space $H_0^r\left(\Lambda;H^s_\alpha[-\pi,\pi]\right)$ for general $r,s\in\R$, we also refer to \cite{Lechl2016}.  Then we have the following properties of the operator $\J$.

\begin{theorem}
\label{th:fb}
The Floquet-Bloch transform is an isomorphism between $H^s_r(\R)$ and $H_0^r\left(\Lambda;H^s_\alpha[-\pi,\pi]\right)$. Given any $\phi\in H^s_r(\R)$, the function $\J\phi(\cdot,x)$ is 1-periodic with respect to $\alpha$ and {\em $\alpha$-quasi-periodic} with respect to $x\in[-\pi,\pi]$, i.e.,
\[
(\J\phi)(\alpha,x+(2\pi,0)^\top)=e^{\i\alpha 2\pi}(\J\phi)(\alpha,x),\quad\forall(\alpha,x)\in\Lambda\times[-\pi,\pi].
\] The inverse operator is given by
\begin{equation}
\label{eq:def_ifb}
(\J^{-1}\psi)(x+(2\pi,0)^\top)=\int_{-1/2}^{1/2}\psi(\alpha,x)e^{\i 2\pi j\alpha}\d\alpha,\quad x\in[-\pi,\pi],\,j\in\Z.
\end{equation}
\end{theorem}

\begin{figure}[h]
\centering
\begin{tikzpicture}[scale=.8, transform shape]
\draw[blue,line width=0.45mm](-3.5,0) -- (-2.5,0);
\draw[blue,line width=0.45mm](-1.5,0) arc (0:180:.5);
\draw[blue,line width=0.45mm](-1.5,0) -- (-0.5,0);
\draw[blue,line width=0.45mm](0.5,0) arc (0:180:.5);
\draw[blue,line width=0.45mm](0.5,0) -- (1.5,0);
\draw[blue,line width=0.45mm](2.5,0) arc (0:180:.5);
\draw[blue,line width=0.45mm](2.5,0) -- (3.5,0);
\draw[blue,line width=0.45mm](4.5,0) arc (0:180:.5);
\draw[blue,line width=0.45mm](4.5,0) -- (5.5,0);
\draw[blue,line width=0.45mm] (-3.5,2) -- (5.5,2);
\draw[blue, style = dashed, line width =0.45mm] (-3,0) -- (-3,2);
\draw[blue, style = dashed, line width =0.45mm] (-1,0) -- (-1,2);
\draw[blue, style = dashed, line width =0.45mm] (1,0) -- (1,2);
\draw[blue, style = dashed, line width =0.45mm] (3,0) -- (3,2);
\draw[blue, style = dashed, line width =0.45mm] (5,0) -- (5,2);
\draw(2.5,1) node {\Large{$\Omega^H$}};
\draw(5.8,0) node {\Large{$\Gamma$}};
\draw(5.9,2) node {\Large{$\Gamma^H$}};
\draw[blue,line width=0.45mm](8.5,1) -- (9,1);
\draw[blue,line width=0.45mm](10,1) arc (0:180:.5);
\draw[blue,line width=0.45mm](10,1) -- (10.5,1);
\draw[blue,line width=0.45mm](8.5,1) -- (8.5,3);
\draw[blue,line width=0.45mm](10.5,1) -- (10.5,3);
\draw[blue,line width=0.45mm](8.5,3) -- (10.5,3);
\draw[red,line width=0.45mm](8.5,-1) --(10.5,-1);
\draw[red,line width=0.45mm](8.5,-1) -- (8.5,-.9);
\draw[red,line width=0.45mm](10.5,-1) -- (10.5,-.9);
\draw[black,line width=0.2mm](9.3,-0.2) -- (9.7,0.2);
\draw[black,line width=0.2mm](9.3,0.2) -- (9.7,-0.2);
\draw[black, line width=0.4mm] (6,1.2) -- (8,1.2);
\draw[black, line width=0.4mm] (6,0.6) -- (8,0.6);
\draw[black, line width=0.4mm] (7.8,1.4) -- (8,1.2);
\draw[black, line width=0.4mm] (6,0.6) -- (6.2,0.4);
\draw(7,1.5) node {\Large{$\J$}};
\draw(7.2,0.3) node {\Large{${\J}^{-1}$}};
\draw(9.5,2.3) node {\Large{$x\in\Omega^H_0$}};
\draw(10.8,1) node {\Large{$\Gamma_0$}};
\draw(10.9,3) node {\Large{$\Gamma^H_0$}};
\draw(9.5,-1.3) node {\Large{$\alpha\in\Lambda$}};
\end{tikzpicture}
\label{fig:fb}
\caption{The Floquet-Bloch transform.}
\end{figure}
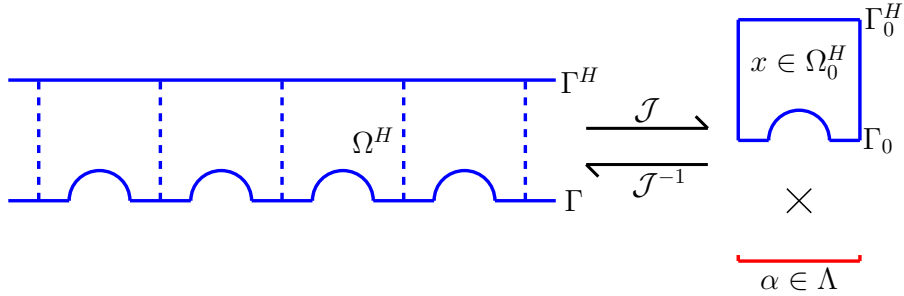

Note that the definition is easily extended to the two-dimensional domains, which are periodic in the $x_1$-dimension. For convenience, we first define the restrictions of the domains and curves in one periodicity cell:
\[
\Omega_0:=\Omega\cap[-\pi,\pi]\times\R;\quad \Omega_0^H:=\Omega^H\cap[-\pi,\pi]\times\R
\]
and
\[
\Gamma_0:=\Omega\cap[-\pi,\pi]\times\R;\quad \Omega_0^h:=\Gamma^h\cap[-\pi,\pi]\times\R.
\]
We also define the translated periodicity domains by
\[
\Omega_j:=\{x+(2\pi j,0)^\top:\,x\in\Omega_0\};\quad \Omega_j^H:=\{x+(2\pi j,0)^\top:\,x\in\Omega_0^H\}
\]
and
\[
\Gamma_j:=\{x+(2\pi j,0)^\top:\,x\in\Gamma_0\};\quad \Gamma_j^H:=\{x+(2\pi j,0)^\top:\,x\in\Gamma_0^H\}.
\]
The properties of the two dimensional Floquet-Bloch transform is concluded as follows. For a visulization please refer to Figure \ref{fig:fb}.

\begin{theorem}
\label{th:fb_2d}
The Floquet-Bloch transform is an isomorphism between $H^s_r(\Omega^H)$ and $H_0^r\left(\Lambda;H^s_\alpha(\Omega^H_0)\right)$. Given any $\phi\in H^s_r(\Omega^H)$, the function $\J\phi(\cdot,x)$ is 1-periodic with respect to $\alpha$ and $\alpha$-quasi-periodic with respect to $x_1$. The inverse operator is given by
\begin{equation}
\label{eq:def_ifb_2d}
(\J^{-1}\psi)(x+(2\pi j,0)^\top)=\int_{-1/2}^{1/2}\psi(\alpha,x)e^{\i 2\pi j\alpha}\d\alpha,\quad x\in\Omega^H_0,\,j\in\Z.
\end{equation}
\end{theorem}

\begin{remark}
Note that since $\psi(\alpha,x)$ depends periodically on $\alpha$, we can easily translate the integral interval  in \eqref{eq:def_ifb_2d} as
\begin{equation}
\label{eq:def_ifb2d_v2}
(\J^{-1}\psi)(x+2\pi j)=\int_{\gamma}^{1-\gamma}\psi(\alpha,x)e^{\i 2\pi j\alpha}\d\alpha,\quad x\in\Omega^H_0,\,j\in\Z
\end{equation}
for any $\gamma\in\R$. Thus we will automatically change the interval when necessary.
\end{remark}

At the end of this subsection, we will introduce the following spaces as in \cite{Zhang2017e}. Let $K\subset\R$ be a bounded interval,  $W\subset\R^2$ be a bounded domain and $S(W)$ be a Sobolev space defined in $W$. Note that here $S(W)$ can depend on $\alpha$, for example, $H^1_\alpha(\Omega^H_0)$ in Theorem \ref{th:fb}. For any $n\in\N$ define the function space
\begin{align*}
C^n(K;S(W)):=&\left\{f\in\mathcal{D}'(K\times W):\, \text{ for all }j=0,1,\dots,n,\,\frac{\partial^j f(\alpha,\cdot)}{\partial\alpha^j}\in S(W),
\right.\\&\qquad\left.\text{moreover, }\frac{\partial^j f(\alpha,\cdot)}{\partial\alpha^j}\text{ depends continuously on $\alpha\in K$}\right\};
\end{align*}
and for $n=\infty$,
\begin{align*}
C^\infty(K;S(W)):=\bigcap_{n=1}^\infty C^n(K;S(W)).
\end{align*}
We also define the spaces $C^n_0(K;S(W))$ and  $C^\infty_0(K;S(W))$ as subspaces of $C^n(K;S(W))$ and $C^\infty(K;S(W))$, where the functions are compactly supported in $K$ with respect to $\alpha$; and $C^n_\p(K;S(W))$ and  $C^\infty_\p(K;S(W))$ as subspaces of $C^n(K;S(W))$ and $C^\infty(K;S(W))$, where the functions are periodic with respect to $\alpha$.

\subsection{Application of the Floquet-Bloch transform and results}
\label{sec:ap_fb}

In this section, we apply the Floquet-Bloch transform to the unique solution $u$ to the problem \eqref{eq:sca}. From Theorem \ref{th:wep}, the solution $u\in H_r^1(\Omega)$. Then with Theorem \ref{th:fb_2d}, the Floquet-Bloch tranform is well defined for $u$. Let $w(\alpha,x):=\J u$, then $w\in H_0^r\left(\Lambda;H^1_\alpha(\Omega^H_0)\right)$. For each fixed $\alpha\in\Lambda$, $w(\alpha,\cdot)$ is $\alpha$-quasi-periodic and satisfies
\begin{equation}
\label{eq:quasi}
\Delta w(\alpha,\cdot)+k^2 w(\alpha,\cdot)=0\text{ in }\Omega_0;\quad w(\alpha,\cdot)=(\J f)(\alpha,\cdot)\text{ on }\Gamma_0
\end{equation}
with the radiation condition, i.e., the Rayleigh expansion (see \cite{Milla1973}):
\begin{equation}
\label{eq:rayleigh}
w(\alpha,x)=\sum_{j\in\Z}\hat{w}_j(\alpha) e^{\i(\alpha+j)x_1+\i\sqrt{k^2-(\alpha+j)^2}(x_2-H)},\quad x_2\geq H.
\end{equation}
From \cite{Kirsc1993}, the quasi-periodic problems are always well-posed.

\begin{corollary}
\label{cr:w_reg}
Suppose the surface $\Gamma$ satisfies the conditions in Theorem \ref{th:wep}. Then for the solution $u|_{\Omega^h}$ (for any fixed $h\geq H$) to the  problem \eqref{eq:sca}, the Floquet-Bloch transform $w:=\J u$ is well defined and $w\in H_0^r(\Lambda;H^1_\alpha(\Omega^h_0))$. Moreover, for each fixed $\alpha$, $w(\alpha,\cdot)$ satisfies \eqref{eq:quasi}-\eqref{eq:rayleigh}.
\end{corollary}

To prove the main results, the existence and regularity of the transformed field $w$ as described in Corollary \ref{cr:w_reg} is not sufficient. We need more detailed analysis for the field $w$. We refer to \cite{Zhang2017e} for more explicit descriptions for the regularity with respect to $\alpha$. The results are summarized as the following theorem.\\
%\begin{equation}
%\label{eq:fb_per1}
%w(\alpha,x)=w_0(\alpha,x)+\sqrt{\alpha-\alpha_0}w_1(\alpha,x)
%\end{equation}
%for $2k\notin\N$ and $\alpha_0$ is the Rayleigh singularity; or
%\begin{equation}
%\label{eq:fb_per2}
%w(\alpha,x)=w_0(\alpha,x)+\sqrt{\alpha-\alpha_0}w_1(\alpha,x)+|\alpha-\alpha_0|w_2(\alpha,x)
%\end{equation}
%where $2k\in\N$. Thus we conclude the solution in the following cases.\\

\begin{theorem}\label{th:reg_sca}
Suppose $\Gamma\subset\R^2$ is the graph of a periodic and bounded function $\zeta$, $u$ is the solution to \eqref{eq:sca}-\eqref{eq:uprc}. Let $w$ be the Floquet-Bloch transform of $u$, then the regularity of $w$ depending on $\alpha$ are discussed with two different cases.\\

\noindent
{\bf Case I: $2k\notin\N$.} Let $k=\kappa+j$ where $\kappa\in(-1/2,1/2)\setminus\{0\}$ and $j\in\N$. Let $\alpha_0$ be either $\kappa$ or $-\kappa$, then 
\begin{equation}
\label{eq:sing_c1}
w(\alpha,x)=w_0(\alpha,x)+\sqrt{\alpha-\alpha_0}\,w_1(\alpha,x)+\sqrt{\alpha+\alpha_0}\,w_2(\alpha,x),
\end{equation}
where $w_0,\,w_1,\,w_2\in C^\infty\left(\Lambda;H^1_\alpha(\Omega^H_0)\right)$. Moreover, $w_0\in C^\infty_\p\left(\Lambda;H^1_\alpha(\Omega^H_0)\right)$ and $w_1,\,w_2\in C^\infty_0\left((\pm\alpha_0-\delta,\pm\alpha_0+\delta);H^1_\alpha(\Omega^H_0)\right)$. Moreover, the parameter $\delta>0$ is adjusted such that $[-\alpha_0-\delta,-\alpha_0+\delta]$ and $[\delta_0-\delta,\alpha_0+\delta]$ are both in the interior of $\overline{\Lambda}$. In this case,
\begin{equation}
\label{eq:u_sing_c1}
u(x)=\int_\Lambda w_0(\alpha,x)\d\alpha+\int_{\alpha_0-\delta}^{\alpha_0+\delta}\sqrt{\alpha-\alpha_0}\,w_1(\alpha,x)\d\alpha+\int_{-\alpha_0-\delta}^{-\alpha_0+\delta}\sqrt{\alpha+\alpha_0}\,w_2(\alpha,x)\d\alpha,\quad x\in\Omega^H,
\end{equation}
where $w_0$, $w_1$ and $w_2$ are extended into $\alpha$-quasi-periodic functions for $x\in\Omega^H$.
\\

\noindent
{\bf Case II: $2k\in\N$.} First when $k$ is an integer, let $\alpha_0=0$ and $\Lambda=(-1/2,1/2]$; when $k-1/2$ is an integer, let $\alpha_0=1/2$ and $\Lambda=(0,1]$. Then
\begin{equation}
\label{eq:sing_c2}
w(\alpha,x)=w_0(\alpha,x)+\sqrt{\alpha-\alpha_0}w_1(\alpha,x)+|\alpha-\alpha_0|w_2(\alpha,x),
\end{equation}
where $w_0,\,w_1,\,w_2\in C^\infty\left(\Lambda;H^1_\alpha(\Omega^H_0)\right)$. Moreover, $w_0\in C^\infty_\p\left(\Lambda;H^1_\alpha(\Omega^H_0)\right)$ and $w_1,\,w_2\in C^\infty_0\left((\alpha_0-\delta,\pm\alpha_0+\delta);H^1_\alpha(\Omega^H_0)\right)$. Also, the parameter $\delta>0$ is chosen such that $[\alpha_0-\delta,\alpha_0+\delta]$ lies in the interior of $\overline{\Lambda}$. In this case,
\begin{equation}
\label{eq:u_sing_c2}
u(x)=\int_\Lambda w_0(\alpha,x)\d\alpha+\int_{\alpha_0-\delta}^{\alpha_0+\delta}\sqrt{\alpha-\alpha_0}\,w_1(\alpha,x)\d\alpha+\int_{\alpha_0-\delta}^{\alpha_0+\delta}|\alpha-\alpha_0|\,w_2(\alpha,x)\d\alpha,\quad x\in\Omega^H,
\end{equation}
where again, $w_0$, $w_1$ and $w_2$ are extended into $\alpha$-quasi-periodic functions for $x\in\Omega^H$.
\end{theorem}

From now on, define the set $\S$ and interval $\Lambda$ for different $k$:
\[
\S:=\begin{cases}
\{-\kappa,\kappa\},\quad 2k\not\in\N;\\
\{0\},\quad k\in\N;\\
\{1/2\},\quad k\in\left\{n+\frac{1}{2}:\,n\in\N\right\}.
\end{cases}\quad \Lambda=\begin{cases}
(-1/2,1/2],\quad k\not\in\left\{n+\frac{1}{2}:\,n\in\N\right\};\\
(0,1],\quad\text{ otherwise}.
\end{cases}
\]
We also define the related set $J_-(\alpha_0)$ and $J_+(\alpha_0)$:
\[
J_-(\alpha_0):=\{j\in\Z:\,|\alpha_0+j|<k\},\quad J_+(\alpha_0):=\{j\in\Z:\,|\alpha_0+j|>k\},\quad J_0(\alpha_0):=\{j\in\Z:\,|\alpha_0+j|=k\}
\]
for any $\alpha_0\in \S$. We can easily check that $J_-(\alpha_0)$ and $J_0(\alpha_0)$ are finite sets and $J_+(\alpha_0)$ is infinite.

\section{The main theorems and the sketch of proof}
\label{th:skt}

To make the paper more readable, we here present the main theorem as well as the sketch of proof in this section.  First define the radiation condition as follows.

\begin{definition}[Radiation Condition]
\label{def:rc}
A function $u$ defined in $\Omega$ satisfies the radiation condtion, if
for a fixed $h>0$,
\begin{enumerate}[label=(\Alph*)]
    \item the restriction $u\big|_{\Omega^{H+h}}$ decays in the weak form:
    \begin{equation}
    \label{eq:decay_weak}
    \|u\|_{H^1\left(\Omega^{H+h}_j\right)}\leq C(1+|j|)^{-3/2};
    \end{equation}
    \item the restriction $u\big|_{U_{H+h}}$ satisfies the half-space Sommerfeld radiation condition, i.e.,
\begin{equation}
\label{eq:SRC}
\frac{\partial u}{\partial r}-\i k u=O\left(r^{-3/2}\right),\quad u=O\left(r^{-1/2}\right),\quad r\rightarrow\infty
\end{equation}
where
\[
r=\left|x-(0,H)^\top\right|=\sqrt{x_1^2+(x_2-H)^2}
\]
holds for any $x=(x_1,x_2)\in U_{H+h}$ for any $h>0$.
\end{enumerate}
For the notations and domains we refer to Figure \ref{fig:src}.
\end{definition}

\begin{theorem}
\label{th:SRC}Let $\Gamma$ be the graph of a periodic bounded function. 
The unique solution $u\in H^1_{r}(\Omega)$ ($|r|<1$) to \eqref{eq:sca} - \eqref{eq:uprc} also satisfies the radiation condition defined in Definition \ref{def:rc}.
\end{theorem}

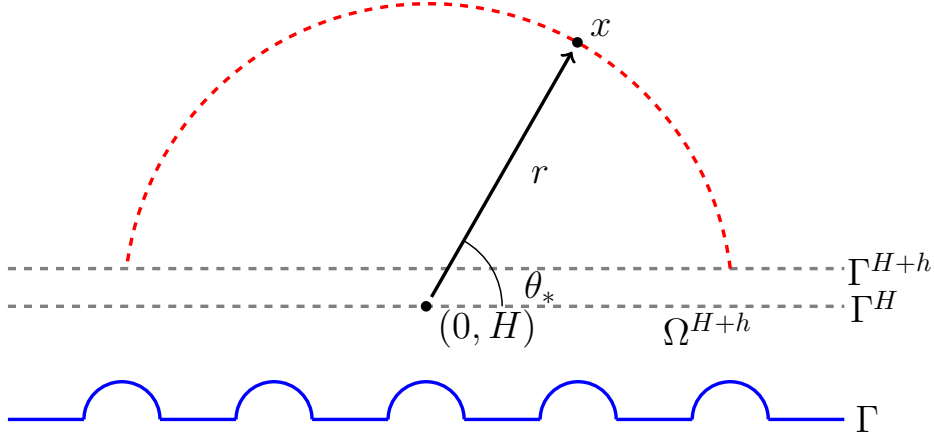
\begin{figure}[h]
\centering
\begin{tikzpicture}[scale=1, transform shape]
\draw[blue,line width=0.45mm](-5.5,0) -- (-4.5,0);
\draw[blue,line width=0.45mm](-3.5,0) arc (0:180:.5);
\draw[blue,line width=0.45mm](-3.5,0) -- (-2.5,0);
\draw[blue,line width=0.45mm](-1.5,0) arc (0:180:.5);
\draw[blue,line width=0.45mm](-1.5,0) -- (-0.5,0);
\draw[blue,line width=0.45mm](0.5,0) arc (0:180:.5);
\draw[blue,line width=0.45mm](0.5,0) -- (1.5,0);
\draw[blue,line width=0.45mm](2.5,0) arc (0:180:.5);
\draw[blue,line width=0.45mm](2.5,0) -- (3.5,0);
\draw[blue,line width=0.45mm](4.5,0) arc (0:180:.5);
\draw[blue,line width=0.45mm](4.5,0) -- (5.5,0);
\draw[black!50, style=dashed, line width=0.45mm] (-5.5,1.5) -- (5.5,1.5);
\draw[black!50, style=dashed, line width=0.45mm] (-5.5,2) -- (5.5,2);
\draw(3.7,1.2) node {\Large{$\Omega^{H+h}$}};
\draw(5.8,0) node {\Large{$\Gamma$}};
\draw(5.9,1.5) node {\Large{$\Gamma^H$}};
\draw(6.1,2) node {\Large{$\Gamma^{H+h}$}};
\draw[red, style=dashed,,line width=0.45mm](4,2) arc (7:173:4);
\fill[black] (0,1.5) circle(0.07);
\draw(0.8,1.2) node {\Large{{$(0,H)$}}};
\fill[black] (1.99,5) circle(0.07);
\draw(2.3,5.2) node {\Large{{$x$}}};
\node (A) at (0, 1.5) {};
\node (B) at (1.99, 5) {};
\draw [->, line width=0.45mm] (A) -- (B);
\draw[black,line width=0.2mm] (1,1.5) arc (0:60:1);
\draw(1.5,1.75) node {\Large{$\theta_*$}};
\draw(1.5,3.25) node {\Large{$r$}};
\end{tikzpicture}
\label{fig:src}
\caption{Notations and definitions for the structure.}
\end{figure}

To prove Theorem \ref{th:SRC}, we begin with the representation \eqref{eq:u_sing_c1} and \eqref{eq:u_sing_c2} with different $k$'s. With the $\alpha$-quasi-periodicities of the functions, the solution $u$ is decomposed into integrals in the following three forms:
\begin{align}
\label{eq:def_u0}
&u_0(x+(2\pi j,0)^\top):=\int_\Lambda w_0(\alpha,x)e^{\i 2\pi\alpha}\d\alpha,\quad w_0\in C^\infty_\p\left(\Lambda;H^1_\alpha(\Omega^{H+h}_0)\right);\\
\label{eq:def_u1}
&u_1(x+(2\pi j,0)^\top):=\int_{\alpha_0-\delta}^{\alpha_0+\delta} \sqrt{\alpha-\alpha_0}w_1(\alpha,x)e^{\i 2\pi\alpha}\d\alpha,\quad w_1\in C^\infty_0\left((\alpha_0-\delta,\alpha_0+\delta);H^1_\alpha(\Omega^{H+h}_0)\right);\\
\label{eq:def_u2}
&u_2(x+(2\pi j,0)^\top):=\int_{\alpha_0-\delta}^{\alpha_0+\delta} |\alpha-\alpha_0|w_2(\alpha,x)e^{\i 2\pi\alpha}\d\alpha,\quad w_2\in C^\infty_0\left((\alpha_0-\delta,\alpha_0+\delta);H^1_\alpha(\Omega^{H+h}_0)\right).
\end{align}

We also note that \eqref{eq:def_u0}, \eqref{eq:def_u1} and \eqref{eq:def_u2}  only describe the solutions restricted in $\Omega^{H+h}$. For the extension of the solution  into the half space $U_H$, we need the apply the UPRC \eqref{eq:u_sing_c1}
\[
u(x):=2\int_{\Gamma^H}\frac{\partial\Phi(x,y)}{\partial y_2}u(y)\d s(y),\quad x_2>H.
\]
Since $u$ is decomposed into the functions in terms of $u_0$, $u_1$ and $u_2$, we need to study the radiation condition for
 \begin{equation}
\label{eq:def_v}
v_j(x):=2\int_{\Gamma^H}\frac{\partial\Phi(x,y)}{\partial y_2}u_j(y)\d s(y),\quad x_2>H
\end{equation}
where $j=0,1,2$. We will need to study the decay rate of these three functions, as well as the radiation conditions for these functions, when $r=|x-(0,H)^\top|\rightarrow\infty$. \\

%\noindent
%To prove Theorem \ref{th:SRC}, we need to analyze the asymptotic behaviours of a family of integrals (see Table \ref{tb:integral}). This is {\bf a preparation step for the main proof}, and will be given in Section \ref{sec:fourier}.\\

\noindent
For the {\bf proof for Part (A)}, we simply study the asymptotic behaviour of  the terms \eqref{eq:def_u0}, \eqref{eq:def_u1} and \eqref{eq:def_u2} in Section \ref{sec:decay_uj}, based on the analysis of some integral in Section \ref{sec:fourier}. We will prove that when $r\rightarrow\infty$, in the weak sense,
\begin{itemize}
\item the function $|u_0(x)|$ decays super algebraically, see Section \ref{sec:decay_0};\\
\item the function $|u_1(x)|$ decays at the rate of $r^{-3/2}$, see Section \ref{sec:decay_1};\\
\item the function $|u_0(x)|$ decays at the rate of $r^{-2}$, see Section \ref{sec:decay_2}.
\end{itemize}
Moreover, we also obtain the decay of $|u_0(x)|$ restricted on $\Gamma^H$ with above rates in the strong sense, which is also useful in the proof for Part (B).
\\

\noindent
The {\bf proof for Part (B)} is the more challenging part, where we prove the radiation conditions for $v_0$, $v_1$ and $v_2$. {\bf The proof for $v_0$} simply follows the proof in \cite{Kirsc2022}, using the decay rate of $u_0$ given by Section \ref{sec:decay_0}. We refer to Section \ref{sec:rc_0} for the details. {\bf The proofs for $v_1$ and $v_2$ }are more complex. Since the decay rates of $u_1$ and $u_2$ are too slow, the method for $v_0$ no longer works thus new approaches will be introduced.

 Take $v_1$ as an example. Recall that $v_1$ is defined by \eqref{eq:def_v} with $u_1$ defined by \eqref{eq:def_u1}. Assume that $w_1(\alpha,x)\big|_{\Gamma^H_0}\in C_0^\infty\left((\alpha_0-\delta,\alpha_0+\delta);C^2(\Gamma^H_0)\right)$  has the Fourier series
\[
w_1(\alpha,x)\big|_{\Gamma^H_0}=\sum_{j\in\Z}g_j(\alpha)e^{\i (\alpha+j)x_1},
\]
where $g_j(\alpha)$ defined by
\[
g_j(\alpha)=\frac{1}{2\pi}\int_{\alpha_0-\delta}^{\alpha_0+\delta}\left[w_1(\alpha,x)\big|_{\Gamma^H_0}\right]e^{-\i(\alpha+j)x_1}\d x_1\in C^\infty_0(\alpha_0-\delta,\alpha_0+\delta)
\]
and decays super-algebraically with respect to $j\in\Z$. Then
\[
u_1(y_1,H)=\int_{\alpha_0-\delta}^{\alpha_0+\delta} \sqrt{\alpha-\alpha_0} w_1(\alpha,y)\d\alpha=\sum_{j\in\Z}\int_{\alpha_0-\delta}^{\alpha_0+\delta}\sqrt{\alpha-\alpha_0}g_j(\alpha)e^{\i(\alpha+j)y_1}\d\alpha,
\]
where the interchanging of the summation and integration is guaranteed by the decay of $g_j(\alpha)$. Define
\[
h(\alpha):=\begin{cases}\displaystyle
\sqrt{\alpha-j-\alpha_0}g_j(\alpha-j)\quad\text{ when }\alpha\in(\alpha_0+j-\delta,\alpha+j+\delta),\,\forall\,j\in\Z;\\
\displaystyle
0,\quad\text{ when }\alpha\not\in\cup_{j\in\Z}(\alpha_0+j-\delta,\alpha+j+\delta).
\end{cases}
\]
and it equals to $0$ for other cases. Then
\[
u_1(y_1,H)=\int_{\R}h(\alpha)e^{\i\alpha y_1}\d\alpha\quad\Rightarrow\quad \F[\phi](\xi)=h(\xi)
\]
where $\F$ is the Fourier transform. On the other hand, when we fix $x_2\geq H+h$,
\[
\left.\frac{\partial \Phi(x,y)}{\partial y_2}\right|_{\Gamma^H}=\int_\R e^{\i\xi (x_1-y_1)+\i\sqrt{k^2-\xi^2}(x_2-H)}\d\xi:=\tilde{\Phi}(x_1-y_1)
\]
which implies that
\[
\F\left[\tilde{\Phi}\right](\xi)=e^{\i\sqrt{k^2-\xi^2}(x_2-H)}.
\]
Since
\[
v_1(x)=\int_\R \tilde{\Phi}(x_1-y_1)u_1(y_1,H)\d y_1=\left(\tilde{\Phi}*\phi\right)(x_1)
\]
where $*$ is the convolution operator. From the convolution theorem, the Fourier transform of $u$ has the form of
\[
\F[v_1](\xi)=\F\left[\tilde{\Phi}\right](\xi)\F[u_1(\cdot,H)](\xi)=h(\xi)e^{\i\sqrt{
k^2-\xi^2}(x_2-H)}.
\]
Thus $u$ has the form of the inverse Fourier transform
\begin{align}
v_1(x)&=\F^{-1}\F[v_1](x)=\int_\R h(\xi)e^{\i\xi x_1+\i\sqrt{
k^2-\xi^2}(x_2-H)}\d\xi\nonumber\\
\label{eq:decom_u}
&=\sum_{j\in\Z}\int_{\alpha_0-\delta}^{\alpha_0+\delta}\sqrt{\alpha-\alpha_0}g_j(\alpha)e^{\i(\alpha+j)x_1+\i\sqrt{k^2-(\alpha+j)^2}(x_2-H)}\d\alpha.
\end{align}
From \eqref{eq:decom_u}, $v_1$ is decomposed the functions in terms of
\[
v_1^{(j)}(x):=\int_{\alpha_0-\delta}^{\alpha_0+\delta}\sqrt{\alpha-\alpha_0}g_j(\alpha)e^{\i(\alpha+j)x_1+\i\sqrt{k^2-(\alpha+j)^2}(x_2-H)}\d\alpha,\quad j\in\Z.
\]
Similarly for $v_2$, it is decomposed into:
\[
v_2^{(j)}(x):=\int_{\alpha_0-\delta}^{\alpha_0+\delta}|\alpha-\alpha_0|g_j(\alpha)e^{\i(\alpha+j)x_1+\i\sqrt{k^2-(\alpha+j)^2}(x_2-H)}\d\alpha,\quad j\in\Z.
\]
Thus we need to study the radiation conditions for $v_1^{(j)}$ and $v_2^{(j)}$, respectively. For the investigation for $v_1^{(j)}$ we refer to Section \ref{sec:rc_1} and for $v_2^{(j)}$ we refer to \ref{sec:rc_2}. Note that with different choices of $j$, the signs of $k^2-(\alpha+j)^2$  are also different. Thus we study the cases
\begin{itemize}
\item when $j\in J_-(\alpha_0)$ in Section \ref{sec:v1} for $v_1^{(j)}$ and in Section \ref{sec:v1_2} for $v_2^{(j)}$; 
\item when $j\in J_+(\alpha_0)$ in Sectin \ref{sec:v2} for $v_1^{(j)}$ and in Section \ref{sec:v2_2} for $v_2^{(j)}$; 
\item when $j\in J_0(\alpha_0)$ in Sectin \ref{sec:v3} for $v_1^{(j)}$ and in Section \ref{sec:v3_2} for $v_2^{(j)}$.
\end{itemize}
\vspace{0.2cm}

\noindent
In Section \ref{sec:proof_main}, we  {\bf finish the proof for Part (B)} by summarizing the results from Section \ref{sec:rc_0}, \ref{sec:rc_1} and \ref{sec:rc_2}. For the extension to locally perturbed periodic surfaces, we refer to Section \ref{sec:loc_sur}.\\

\begin{remark}
From \eqref{eq:decom_u}, the function $v_1$ (also $v_2$) is decomposed into a finite number of Herglotz wave function and a function which decays exponentially with respect to $x_2$.  Thus part of the paper can be treated to the study of the radiation condition of the Herglotz wave functions with centain density functions in certain directions.
\end{remark}

\section{Decay rate of functions $u_0,\,u_1,\,u_2$ in $\Omega^{H+h}$}
\label{sec:decay_uj}

In this section, we prove Part (A) in Theorem \ref{th:SRC}, i.e., to study the decay of the functions $u_0,u_1,u_2$ when $|x_1|\rightarrow\infty$, in the domain $\Omega^{H+h}:=\Omega\cap\R\times(-\infty,H+h]$. Here $h>0$ is a fixed small value. Thus $0<x_2\leq H+h$ is uniformly bounded.

%Note that the surface $\Gamma$ is assumed to be defined by a bounded periodic function $\zeta$, which is not sufficient to prove a strong decay rate of the solution. To this end, we introduce a function $\tilde{\zeta}$, which is $C^2$-continous and $2\pi$-periodic. 

%For simplicity, let $\tilde{\zeta}$ be a periodic and $C^2$-smooth function 
%such that $\tilde{\zeta}(x_1)>\zeta(x_1)$ for all $x_1\in\R$. Define
%\[
%\tilde{\Gamma}:=\left\{\left(x_1,\tilde{\zeta}(x_1)\right):\,x_1\in\R\right\}
%\]
%and  $\tilde{\Omega}^{H+h}$ be the domain between $\tilde{\Gamma}$ and $\Gamma^{H+h}$. From the interior regularity for elliptic equations, $u\in H^2\left(\tilde{\Omega}^{H+h}\right)$ thus is of course uniformly bounded. Then for all $x\in \tilde{\Omega}^{H+h}$, there is a constant $C>0$ such that $|u(x)|\leq C$.

\subsection{Decay of $u_0$}
\label{sec:decay_0}

Recall the definition of $u_0$ in \eqref{eq:def_u0}. Fix $x_0\in\Omega^{H+h}_0$ and let $x=x_0+(2\pi j,0)^\top$ for any $j\in\Z\setminus\{0\}$, then
\[
u_0(x)=\int_\Lambda w(\alpha,x_0)e^{\i 2\pi\alpha j}\d\alpha.
\]
Since $w\in C^\infty_\p(\Lambda;H^1_\alpha(\Omega^{H+h}_0))$, with integration by parts,
\[
u_0(x)=\int_{\Lambda}w(\alpha,x_0) e^{\i 2\pi\alpha j}\d\alpha=(-\i 2\pi j)^{-m}\int_{\Lambda}\frac{\partial^m}{\partial\alpha^m}w(\alpha,x_0) e^{\i 2\pi\alpha j}\d\alpha.
\]
This implies that $u_0$ decays super-algebraically when $|x|\rightarrow\infty$. 

\begin{theorem}
\label{th:decay_0}
Suppose $u_0$ is defined by  \eqref{eq:def_u0} and $w\in C^\infty_\p(\Lambda;H^1_\alpha(\Omega^{H+h}_0))$. Then $u_0$  decays super-algebraically, i.e., for any $m\in\N$, there is a constant $C>0$ such that
\[
|u_0(x)|\leq C(1+|x|)^{-m}
\]
where $C$ depends on the unique point $x_0\in \Omega^{H+h}_0$ such that $x=x_0+(2\pi j,0)^\top$ for some $j\in\Z$. We can also easily estimate that
\[
\|u_0\|_{H^1(\Omega_j^{H+h})}\leq C(1+|j|)^{-m}
\]
holds uniformly for any $j\in\Z$, and $C$ only depends on $m$.
\end{theorem}

Since we will need the decay rate of $u_0\big|_{\Gamma^H}$ in the strong form later, with the interior regularity property for elliptic equations (see \cite{ Evans1998}), we have the following corollary.

\begin{corollary}
\label{cr:decay_0}
Suppose $u_0$ is defined by  \eqref{eq:def_u0} and $w\in C^\infty_\p(\Lambda;H^1_\alpha(\Omega^{H+h}_0))$. Then $u_0\big|_{\Gamma^H}$  decays super-algebraically, i.e., for any $m\in\N$, there is a constant $C>0$ which does not depend on $x$ such that
\[
|u_0(x_1,H)|\leq C(1+|x_1|)^{-m}
\]
holds uniformly for $x_1\in H$ with fixed $m\in\N$.
\end{corollary}

\subsection{Decay of $u_1$}
\label{sec:decay_1}

We consider the case that $u_1$ defined by \eqref{eq:def_u1}. Similar as in Section \ref{sec:decay_2}, let $x=x_0+(2\pi j,0)^\top$ for $j\in\Z$,
\[
u_1(x)=\int_{\alpha_0-\delta}^{\alpha_0+\delta}\sqrt{\alpha-\alpha_0}w(\alpha,x)e^{\i 2\pi\alpha j}\d\alpha
\]
where $\phi$ is compactly supported and smooth.

\begin{theorem}
\label{th:decay_1} For any $u_1$ defined by \eqref{eq:def_u1}, then there is a constant $C>0$ which depends on $x_0\in \Omega^{H+h}_0$ such that
\[
|u_1(x)|\leq C\|w(\cdot,x_0)\|_{C^2(\alpha_0-\delta,\alpha_0+\delta)}(1+|x_1|)^{-3/2}.
\]
And we also have 
\[
\|u_1\|_{H^1(\Omega_j^{H+h})}\leq C(1+|j|)^{-3/2}
\]
holds uniformly for any $j\in\Z$, and $C$ only depends on $m$.
\end{theorem}

\subsection{Decay of $u_2$}
\label{sec:decay_2}

In this section, we consider the decay of $u_{2}(x)$ when $|x_1|\rightarrow\infty$ when $w\in C^\infty_0(\alpha_0-\delta,\alpha_0+\delta)$.

\begin{theorem}
\label{th:decay22}
Suppose $\X$ be a smooth cutoff function which vanishes when $|t|>\delta$ and equals to $1$ when $|t|<\delta/2$; the function $\phi\in C^2(\R)$ with compact support. Then the function $u_2(x)$ defined by \eqref{eq:def_u2} satisfies
\[
|u_{2}(x_1)|\leq C\|w(\cdot,x_0)\|_{C^2(\alpha_0-\delta,\alpha_0+\delta)}|x_1|^{-2}
\]
where $C$ depends on $x_0$. And we also have 
\[
\|u_2\|_{H^1(\Omega_j^{H+h})}\leq C(1+|j|)^{-2}
\]
holds uniformly for any $j\in\Z$, and $C$ only depends on $m$.
\end{theorem}

\begin{proof}
Directly from the definition \eqref{eq:def_u2},
\begin{align*}
u_{2}(x)&=\int_{\alpha_0-\delta}^{\alpha_0+\delta}|\alpha-\alpha_0|w(\alpha,x_0)e^{\i 2\pi j \alpha}\d\alpha\\
&=e^{\i 2\pi j\alpha_0 }\int_0^\delta  s w(s+\alpha_0,x_0)e^{\i 2\pi j s}\d s
+e^{\i 2\pi j \alpha_0}\int_0^\delta s w(-s+\alpha_0,x_0)e^{-\i 2\pi j s}\d s\\
&:=e^{\i 2\pi j\alpha_0 }(I)+e^{\i 2\pi j\alpha_0 }(II).
\end{align*}
From Theorem \ref{th:decay_inta3}, there is a constant $C>0$ such that
\[
|(I)|\leq C|x_1|^{-2}.
\]
The argument also holds for (II). The result is concluded in the following theorem.
\end{proof}

\section{Radiation condition for $v_0$}
\label{sec:rc_0}

In this section, we will prove  Part (B) in Theorem \ref{th:SRC} for $v_0$, which is defined by \eqref{eq:def_v}.  We  consider the radiation condition for the following layer potential (for simplicity, denote $u=v_0$ and $\phi=2 u_0\big|_{\Gamma^H}$ in the definition):
\begin{equation}
\label{eq:layer1}
u(x):=\int_{\Gamma^H}\frac{\partial\Phi(x,y)}{\partial y_2}\phi(y)\d s(y),\quad x_2\geq H+h,
\end{equation}
where $\phi$ is defined by \eqref{eq:def_u0}. From Corollary \ref{cr:decay_0} , $u_0\big|_{\Gamma^H}$ decays super-algebraically, i.e., for any given integer $m\geq 1$, there is a constant $C>0$ such that
\[
|\phi(y)|\leq C(1+|y|^2)^{-m}
\]
holds uniformly for all $y\in\Gamma^H$, where $C$ does not depend on $y$ but depends on the integer $m$. 

The aim of this section is to consider the asymptotic behaviour of
\begin{equation}
\label{eq:potential_old}
 u:=\int_{\Gamma_H}S(x,y)\phi(y)\d s(y),
\end{equation}
where $S(x,y)=\frac{\partial\Phi(x,y)}{\partial y_2}$, and
\begin{equation}
\label{eq:potential_new}
\frac{\partial u}{\partial r}-\i k u:=\int_{\Gamma_H}K(x,y)\phi(y)\d s(y),
\end{equation}
 where $K(x,y)$ is the kernel for the integral equation, when $r=|x|\rightarrow\infty$. 

\begin{remark}
The proof mainly follows the proofs in \cite{Hu2021,Kirsc2022}. The only difference is the decay rate of the density function $\phi$, which is super-algebraic in this paper, but the density function in these references decay at the rate of $(1+|y|^2)^{-3/4}$.
\end{remark}

First we need to study the asymptotic behaviours of the kernels $S$ and $K$. 
From the definition of the fundamental solution:
\[
\Phi(x,y)=\frac{\i}{4}H_0^{(1)}(k|x-y|),
\]
where $H_0^{(1)}(t)$ is the Hankel function of the first kind. With \eqref{eq:hankel_diff1},
\[
S(x,y)=\frac{\partial\Phi(x,y)}{\partial y_2}=\frac{\i k}{4}\frac{x_2-H}{|x-y|}H_1^{(1)}(k|x-y|)
\]
Then with \eqref{eq:hankel_diff2}, and for simplicity let $\tilde{x}:=x-(0,H)^\top$,
\begin{align*}
\frac{\partial^2\Phi(x,y)}{\partial r\partial y_2}&=\frac{\tilde{x}}{\left|\tilde{x}\right|}\cdot\nabla_x\left(\frac{\partial\Phi(x,y)}{\partial y_2}\right)\\&=\frac{\i k}{4}H_1^{(1)}(k|x-y|)\left[\frac{x_2-H}{\left|\tilde{x}\right||x-y|}-\frac{\left<\tilde{x},x-y\right>(x_2-H)}{\left|\tilde{x}\right||x-y|^3}\right]\\
&+\frac{\i k^2}{4}\frac{\left<\tilde{x},x-y\right>(x_2-H)}{\left|\tilde{x}\right||x-y|^2}\left[H_0^{(1)}(k|x-y|)-\frac{1}{k|x-y|}H_1^{(1)}(k|x-y|)\right]\\
&=\frac{\i k}{4}H_1^{(1)}(k|x-y|)\left[\frac{x_2-H}{\left|\tilde{x}\right||x-y|}-\frac{2\left<\tilde{x},x-y\right>(x_2-H)}{\left|\tilde{x}\right||x-y|^3}\right]\\
&+\frac{\i k^2}{4}\frac{\left<\tilde{x},x-y\right>(x_2-H)}{\left|\tilde{x}\right||x-y|^2}H_0^{(1)}(k|x-y|).
\end{align*}
Thus we have:
\begin{align*}
K(x,y)&:=\frac{\partial^2\Phi(x,y)}{\partial r\partial y_2}-\i k S(x,y)\\&=
\frac{\i k}{4}H_1^{(1)}(k|x-y|)\left[\frac{x_2-H}{\left|\tilde{x}\right||x-y|}-\frac{2\left<\tilde{x},x-y\right>(x_2-H)}{\left|\tilde{x}\right||x-y|^3}\right]\\
&+\frac{\i k^2}{4}\frac{\left<\tilde{x},x-y\right>(x_2-H)}{\left|\tilde{x}\right||x-y|^2}H_0^{(1)}(k|x-y|)+\frac{k^2}{4}\frac{x_2-H}{|x-y|}H_1^{(1)}(k|x-y|).
\end{align*}

We first use the asymptotic behaviour \eqref{eq:asym} for $n=0,1$ and $\ell=1$:
\begin{eqnarray}
\label{eq:hankel01}
H_0^{(1)}(k|x-y|)&=&\frac{\gamma}{\sqrt{|x-y|}}\exp\left(\i k |x-y|\right)+O\left(|x-y|^{-3/2}\right);\\
\label{eq:hankel11}
H_1^{(1)}(k|x-y|)&=&-\i\frac{\gamma}{\sqrt{|x-y|}}\exp\left(\i k |x-y|\right)+O\left(|x-y|^{-3/2}\right).
\end{eqnarray}
where $\gamma=\sqrt{\frac{2}{k\pi}}\exp\left(-\i\frac{\pi}{4}\right)$ is a constant. Thus 
\[
S(x,y),\,K(x,y)=O\left(|x-y|^{-1/2}\right)
\]
in general when $|x-y|\rightarrow\infty$.

For the kernel $K(x,y)$, with more detailed studies,
\begin{align*}
K(x,y)&=O\left(|x-y|^{-3/2}\right)+\frac{\i k^2\gamma}{4\sqrt{|x-y|}}\frac{x_2-H}{|x-y|}\left[\frac{\left<\tilde{x},x-y\right>}{\left|\tilde{x}\right||x-y|}-1\right]\exp(\i k |x-y|)\\
&=O\left(|x-y|^{-3/2}\right)+O\left(|x-y|^{-1/2}\right)\left[\frac{\left<\tilde{x},x-y\right>}{\left|\tilde{x}\right||x-y|}-1\right].
\end{align*}
Now we will focus on the term
\[
\frac{\left<\tilde{x},x-y\right>}{\left|\tilde{x}\right||x-y|}-1
\]
particularly. The property of this term is concluded in the following lemma.

\begin{lemma}
\label{lm:hankel}
Assume that $|x|>>H$, $y=(y_1,y_2)=(y_1,H)$ satisfying $|y_1|<\sqrt{|x|}$. Then 
\begin{equation}
\label{eq:lm:hankel}
0\leq 1-\frac{\left<\tilde{x},x-y\right>}{\left|\tilde{x}\right||x-y|}<\frac{4}{r}.
\end{equation}
\end{lemma}

\begin{proof}Let $\tilde{y}:=y-(0,H)^\top$, then
\[
x-y=\tilde{x}-\tilde{y}\quad\Rightarrow\quad \frac{\left<\tilde{x},x-y\right>}{|\tilde{x}||x-y|}-1=
\frac{\left<\tilde{x},\tilde{x}-\tilde{y}\right>}{|\tilde{x}||\tilde{x}-\tilde{y}|}-1.
\]
When $|y_1|<\sqrt{|\tilde{x}|}$, then $|\tilde{y}|<2\sqrt{|\tilde{x}|}= 2\sqrt{r}$. In this case, since $|x|>>H$,
 \[
|\tilde{x}-\tilde{y}|\geq |\tilde{x}|-|\tilde{y}|> |\tilde{x}|-2\sqrt{|\tilde{x}|}\geq\frac{r}{2}.
\]
 First, note that
\[
|\tilde{x}-\tilde{y}|=\sqrt{|\tilde{x}|^2+|\tilde{y}|^2-2\left<\tilde{x},\tilde{y}\right>}=|\tilde{x}|\sqrt{1+\frac{|\tilde{y}|^2}{|\tilde{x}|^2}-\frac{2\left<\tilde{x},\tilde{y}\right>}{|\tilde{x}|^2}}:=|\tilde{x}|\sqrt{1+t},
\]
where
\[
t=\frac{|\tilde{y}|^2}{|\tilde{x}|^2}-\frac{2\left<\tilde{x},\tilde{y}\right>}{|\tilde{x}|^2}.
\]
Since $|\tilde{x}|>>1$,
\[
|t|\leq \frac{4|\tilde{x}|}{|\tilde{x}|^2}+\frac{2|\tilde{y}|}{|\tilde{x}|}<\frac{4}{|\tilde{x}|}+\frac{4}{\sqrt{|\tilde{x}|}}<1.
\]
Then
\[
|\tilde{x}||\tilde{x}-\tilde{y}|=|\tilde{x}|^2\sqrt{1+t}\leq |\tilde{x}|^2\left(1+\frac{1}{2}t\right)=|\tilde{x}|^2+\frac{|\tilde{y}|^2}{2}-\left<\tilde{x},\tilde{y}\right>,
\]
where we use the fact that
\[
\sqrt{1+t}\leq 1+\frac{1}{2}t,\quad |t|<1.
\]
Use this result,
\begin{align*}
1-\frac{\left<\tilde{x},\tilde{x}-\tilde{y}\right>}{\left|\tilde{x}\right||\tilde{x}-\tilde{y}|}&=\frac{\left|\tilde{x}\right||\tilde{x}-\tilde{y}|-|\tilde{x}|^2+\left<\tilde{x},\tilde{y}\right>}{\left|\tilde{x}\right||\tilde{x}-\tilde{y}|}\\
&\leq\frac{\left|\tilde{x}\right|^2+\frac{|\tilde{y}|^2}{2}-\left<\tilde{x},\tilde{y}\right>-|\tilde{x}|^2+\left<\tilde{x},\tilde{y}\right>}{|\tilde{x}||\tilde{x}-\tilde{y}|}\\
&=\frac{|\tilde{y}|^2}{2|\tilde{x}||\tilde{x}-\tilde{y}|}\leq\frac{4r}{2r(r/2)}=\frac{4}{r}.
\end{align*}

\end{proof}

The above lemma  implies that when $|y_1|<\sqrt{|x|}$, since $|x-y|\sim|x|$,
\[
K(x,y)=O\left(|x-y|^{-3/2}\right).
\]
Now we conclude that in general,
\begin{equation}
\label{eq:kernel_general}
S(x,y),\,K(x,y)=O\left(|x-y|^{-1/2}\right)
\end{equation}
and when $|y_1|<\sqrt{|x|}$,
\begin{equation}
\label{eq:kernel_far}
K(x,y)=O\left(|x-y|^{-3/2}\right).
\end{equation}
We get back to the layer potential defined by \eqref{eq:potential_old}-\eqref{eq:potential_new} and estimate the decay when $|x|\rightarrow\infty$.

%\begin{theorem}
%\label{th:rc_layer}
%Let $u$ be defined as in \eqref{eq:layer} where $\phi$ is decaying algebriacally, i.e., 
%\end{theorem}

\begin{theorem}
\label{th:decay_layer}
Let $u$ be defined as the layer potential \eqref{eq:potential_old} and $S$ satisfies the asymptotic behaviour  \eqref{eq:kernel_general}. Then when $|x|=r\rightarrow\infty$, there is a constant $C>0$ which does not depend on $x$ such that
\begin{equation}
\label{eq:decay_layer}
|u(x)|\leq Cr^{-1/2}
\end{equation}
holds uniformly for all $x_2\geq H+h$.
\end{theorem}

\begin{proof}
First split the integral \eqref{eq:potential_old} into two parts:
\begin{align*}
|u(x)|=\int_{|y_1|<\sqrt{r}}S(x,y)\phi(y)\d s(y)+\int_{|y_1|\geq\sqrt{r}}S(x,y)\phi(y)\d s(y):=(I)+(II).
\end{align*}
For term (I), since $|x-y|\geq |x|-|y|\geq\frac{r}{2}$, use the asymptotic behaviour of $S(x,y)$ and $\phi$,
\begin{align*}
|(I)|&\leq \int_{|y_1|<\sqrt{r}}|S(x,y)||\phi(y)|\d s(y)\\
&\leq Cr^{-1/2}\int_{-\sqrt{r}}^{\sqrt{r}}(1+|y_1|^2)^{-m}\d y_1\\
&\leq Cr^{-1/2}\int_{-\infty}^{\infty}(1+|y_1|^2)^{-m}\d y_1<Cr^{-1/2}.
\end{align*}
Here we use the fact that the later integral is bounded for $m\geq 1$. For term (II),
\begin{align*}
|(II)|\leq C\int_{|y_1|\geq\sqrt{r}}(1+|y_1|^2)^{-m}\d s(y)\leq C|x|^{-m+1/2}=Cr^{-m+1/2}.
\end{align*}
Take $m=1$, then $|(II)|\leq Cr^{-1/2}$. The proof is finished.
\end{proof}

We can prove the following theorem with similar technique.

\begin{theorem}
\label{th:rc_layer}
Let $u$ be defined as the layer potential \eqref{eq:potential_new} and $K$ satisfies \eqref{eq:kernel_general} and \eqref{eq:kernel_far}. Then when $|x|=r\rightarrow\infty$, there is a constant $C>0$ which does not depend on $x$ such that
\begin{equation}
\label{eq:rc_layer}
\left|\frac{\partial u}{\partial r}-\i k u\right|\leq Cr^{-3/2}
\end{equation}
holds uniformly for all $x_2\geq H+h$.
\end{theorem}

\begin{proof}

Similar as the proof of Theorem \ref{th:decay_layer}, we split the integral into two parts:
\[
\frac{\partial u}{\partial r}-\i k u=\int_{|y_1|<\sqrt{r}}K(x,y)\phi(y)\d s(y)+\int_{|y_1|\geq\sqrt{r}}K(x,y)\phi(y)\d s(y):=(I)+(II).
\]
For term (I), since  $|x-y|>\frac{r}{2}$, with \eqref{eq:kernel_far},
\begin{align*}
\left|(I)\right|&\leq C\int_{|y_1|\leq \sqrt{r}}|x-y|^{-3/2}(1+|y_1|^2)^{-m}\d s(y)\\
&\leq Cr^{-3/2}\int_{-\infty}^{\infty}(1+|y_1|^2)^{-m}\d s(y)
= Cr^{-3/2}.
\end{align*}

For term $(II)$, since $K$ is uniformly bounded,
\begin{align*}
|(II)|\leq C\int_{|y_1|>\sqrt{r}}(1+|y_1|^2)^{-m}\d s(y)=C r^{-m+1/2}.
\end{align*}
In particular, let $m=2$, then $|(II)|\leq Cr^{-3/2}$.
With the estimation of (I), the proof is then finished.
\end{proof}

\section{Radiation condition for $v_1$}
\label{sec:rc_1}
In this section,  we study the radiation condition for the function
\begin{equation}
\label{eq:def_vj}
v(x):=\int_{\alpha_0-\delta}^{\alpha_0+\delta}\sqrt{\alpha-\alpha_0}\phi(\alpha) e^{\i(\alpha+j)x_1+\i\sqrt{k^2-(\alpha+j)^2}(x_2-H)}\d\alpha.
\end{equation}
where we denote any smooth and compactly supported function $g_j$ by $\phi\in C_0^\infty(\alpha_0-\delta,\alpha_0+\delta)$. In the following sections, we will consider different $j$'s. 
\begin{itemize}
\item In Section \ref{sec:v1}, we consider Case I: $j\in J_-$, i.e., $|\alpha_0+j|< k$.
\item In Section \ref{sec:v2}, we consider Case II: $j\in J_+$, i.e., $|\alpha_0+j|> k$. 
\item In Section \ref{sec:v3}, we consider Case III: $j\in J_0$, i.e., $|\alpha_0+j|=k$. 
\end{itemize}
By choosing a proper small parameter $\delta>0$, we can have the further conditions:
\begin{itemize}
\item In Case I,  $|\alpha+j|<k-\delta$ for all $\alpha\in  [\alpha_0-\delta,\alpha_0+\delta]\subset\Lambda$.
\item In Case II, $|\alpha+j|>k+\delta$ for all $\alpha\in  [\alpha_0-\delta,\alpha_0+\delta]\subset\Lambda$.
\end{itemize}

In the following, we always assume that $x=r(\cos\theta_*,\sin\theta_*)+(0,H)$ (see Figure \ref{fig:src}), where $\theta_*\in (0,\pi)$. Moreover, $x_2\geq H+h$. Then $|x|\approx r$ when $r\rightarrow\infty$. We will consider the asymptotic behaviour of $v$ and $\frac{\partial v}{\partial r}-\i k v$ when $r\rightarrow\infty$ for all of the three cases.

\subsection{Case I: $|\alpha_0+j|< k$} 
\label{sec:v1}

As described in the beginning of this section, for any fixed $\alpha_0\in S$, $|\alpha_0+j|<k$ for all $j\in J_-(\alpha_0)$, which is a finite set. Define
\[
\alpha_0+j=k\cos\theta_0,\quad\alpha+j=k\cos\theta,\quad\sqrt{k^2-(\alpha+j)^2}=k\sin\theta
\]
then $\theta_0,\,\theta\in \left[\arccos\frac{\alpha_0+\delta+j}{k},\arccos\frac{\alpha_0-\delta+j}{k}\right]:=[\beta_1,\beta_2]\subset(0,\pi)$. Since for any $j\in J_-(\alpha_0)$, $|\alpha+j|<k-\delta$ for all $\alpha\in[\alpha_0-\delta,\alpha_0+\delta]$, we can have the following lower and upper boundaries of $\beta_1$ and $\beta_2$:
\[
\beta_1>\arccos(1-\delta/k)>\delta_0,\beta_2<\arccos(-1+\delta/k)<\pi-\delta_0.
\]
Since $0<\theta_*<\pi$, we also have the estimation
\[
\delta_0-\pi<\theta-\theta_*<\pi-\delta_0.
\]

 With the polar coordinate of $x$,
\begin{equation}
\label{eq:v_11_def}
v(x)=\int_{\beta_1}^{\beta_2}\sqrt{\theta-\theta_0}\,\psi(\theta)e^{\i k r \cos(\theta-\theta_*)}\d\theta,
\end{equation}
where $$\psi(\theta)=-k^{3/2}\sin\theta\,\phi(k\cos\theta-j)\sqrt{\frac{\cos\theta-\cos\theta_0}{\theta-\theta_0}}\in C_0^\infty(\beta_1,\beta_2).$$ The decay of $v$ with respect to $r$ is considered in the following theorem.  
\begin{theorem}
\label{th:v_11_decay}
Given any $\psi\in C_0^\infty(\beta_1,\beta_2)$ and let $v$ be defined by \eqref{eq:v_11_def}. There is a constant $C>0$ which only depens on $\psi$ such that
\begin{equation}
\label{eq:v_11_decay}
|v(x)|\leq Cr^{-1/2}
\end{equation}
holds uniformly for all $x=r(\cos\theta_*,\sin\theta_*)+(0,H)$ with $x_2\geq H+h$.
\end{theorem}

\begin{proof}
We use the formula
\[
\cos(\theta-\theta_*)=1-2\sin^2\left(\frac{\theta-\theta_*}{2}\right).
\]
Since $\theta-\theta_*\in(\delta_0-\pi,\pi-\delta_0)$, the map $\theta\mapsto \sin\left(\frac{\theta-\theta_*}{2}\right)$ is an injection and anaytic. Let $t:=\sin\left(\frac{\theta-\theta_*}{2}\right)\in[\gamma_1,\gamma_2]\subset(-1,1)$, then
\[
\theta=2\arcsin t+\theta_*,\quad\d\theta=\frac{2}{\sqrt{1-t^2}}\d t.
\]
Replace $\theta$ by $t$ in \eqref{eq:v_11_def} we get
\[
v(x)=e^{\i k r}\int_{\gamma_1}^{\gamma_2}\sqrt{t-a}\zeta(t) e^{-2\i k r t^2}\d t,
\]
where $a=\sin((\theta_0-\theta_*)/2)$
\[
\zeta(t)=\psi\left(2\arcsin t+\theta_*\right)\frac{2}{\sqrt{1-t^2}}\sqrt{\frac{2\arcsin t-2\arcsin a}{t-a}}\in C_0^\infty(\gamma_1,\gamma_2).
\]
Thus $v$ has exactly the same form as the integral $I_2^C(r)$ defined by \eqref{eq:intc2}. From Corollary \eqref{eq:decay_intc2}, there is a constant $C>0$ such that
\[
|v(x)|\leq Cr^{-1/2}
\]
holds uniformly for all $\theta_*\in(0,pi)$ when $r\rightarrow\infty$. 
The proof is finished.
\end{proof}

Now we move on to the radiation condition. From direct computation,
\begin{equation}
\label{eq:v_11_rc}
\left(\frac{\partial v}{\partial r}-\i k v\right)=\i k \int_{\beta_1}^{\beta_2}\sqrt{\theta-\theta_0}\left(\cos(\theta-\theta_*)-1\right){\psi}(\theta)e^{\i k r\cos(\theta-\theta_*)}\d\theta,
\end{equation}
We need to consider two different situations, i.e.,  $\theta_*\notin[\beta_1,\beta_2]$ and $\theta_*\in[\beta_1,\beta_2]$, separately.
The radiation condtion when $\theta_*\notin[\beta_1,\beta_2]$ is concluded in the following theorem.

\begin{theorem}
\label{th:rc_v111}
For any $\psi\in C_0^\infty(\beta_1,\beta_2)$, and $\theta_*\notin[\beta_1,\beta_2]$,   there is a constant $C>0$ only depends on $\psi$ such that
\begin{equation}
\label{eq:rc_v111}
\left|\frac{\partial v}{\partial r}-\i k v\right|\leq Cr^{-3/2}.
\end{equation} 
\end{theorem}

\begin{proof}

Since $\theta_*\in(0,\pi)$ and $\theta\in[\beta_1,\beta_2]\subset(\delta_0,\pi-\delta_0)$, $\frac{\theta-\theta_*}{2}\neq \frac{\pi}{2}$. From direct computation, 
\[
-\frac{\cos(\theta-\theta_*)-1}{\sin(\theta-\theta_*)}=\frac{\sin^2\left(\frac{\theta-\theta_*}{2}\right)}{\sin\left(\frac{\theta-\theta_*}{2}\right)\cos\left(\frac{\theta-\theta_*}{2}\right)}=\tan\left(\frac{\theta-\theta_*}{2}\right).
\]
This function is also extended continuously for $\theta=\theta_*$. Thus with integration by parts,
\begin{align*}
\frac{\partial v}{\partial r}-\i k v&=\frac{1}{ r}\int_{\beta_1}^{\beta_2}\sqrt{\theta-\theta_0}\tan\left(\frac{\theta-\theta_*}{2}\right){\psi}(\theta)\d e^{\i k r\cos(\theta-\theta_*)}\\
&=-\frac{1}{r}\int_{\beta_1}^{\beta_2}\left[\sqrt{\theta-\theta_0}\tan\left(\frac{\theta-\theta_*}{2}\right){\psi}(\theta)\right]' e^{\i k r\cos(\theta-\theta_*)}\d\theta\\
&=\frac{1}{r}\int_{\beta_1}^{\beta_2}\frac{1}{\sqrt{\theta-\theta_0}}\zeta(\theta)e^{\i k r \cos(\theta-\theta_*)}\d\theta:=\frac{1}{\i k r}(I),
\end{align*}
where
\[
\zeta(\theta)=-\frac{1}{2}\tan\left(\frac{\theta-\theta_*}{2}\right)\psi(\theta)-\frac{\theta-\theta_0}{\cos(\theta-\theta_*)+1}\psi(\theta)-(\theta-\theta_0)\tan\left(\frac{\theta-\theta_*}{2}\right)\psi'(\theta)\in C_0^\infty(\beta_1,\beta_2).
\]

Since $\theta_*\notin[\beta_1,\beta_2]$, then $\theta\mapsto \cos(\theta-\theta_*)$ is an injection. Let $t:=\cos(\theta-\theta_*)$, then it lies in an interval $[\gamma_1,\gamma_2]\subset (-1,1)$ and $\theta=\arccos(t)+\theta_*$. Thus let $a:=\cos(\theta-\theta_*)$
\[
I(r)=\int_{\gamma_1}^{\gamma_2}\frac{1}{\sqrt{t-a}}\rho(t)e^{\i r t}\d t,
\]
where
\[
\rho(t)=-\frac{1}{\sqrt{1-t^2}}\sqrt{\frac{t-a}{\arccos(t)-\arccos(a)}}\zeta(\arccos(t)+\theta_*)\in C_0^\infty(\gamma_1,\gamma_2).
\]
Since
\[
I(r)=e^{\i r a}\int_{\gamma_1-a}^{\gamma_2-a}\frac{1}{\sqrt{t}}\rho(t+a)e^{\i r t}\d t,
\]
which has exactly the same as $I_1(r)$ defined in \eqref{eq:inta1}. From Theorem \ref{th:decay_inta1}, $|I_1(r)|\leq Cr^{-1/2}$.
From the previous computation, we get that
\[
\left|\frac{\partial v}{\partial r}-\i k v\right|\leq Cr^{-3/2}
\]
holds uniformly for all $\theta_*\in (0,\pi)\setminus[\beta_1,\beta_2]$ when $r\rightarrow\infty$.
\end{proof}

In the next theorem, we will focus on the case that $\theta_*\in[\beta_1,\beta_2]$.
\begin{theorem}
\label{th:rc_v112}
For any $\psi\in C_0^\infty(\beta_1,\beta_2)$, and $\theta_*\in[\beta_1,\beta_2]$,   there is a constant $C>0$ only depends on $\psi$ such that
\begin{equation}
\label{eq:rc_v112}
\left|\frac{\partial v}{\partial r}-\i k v\right|\leq Cr^{-3/2}.
\end{equation} 
\end{theorem}

\begin{proof}
Note that in this case, $\theta\mapsto\cos(\theta-\theta_*)$ is no longer and injection so the previous method fails. Thus we need to look for a different as in the proof of Theorem \ref{th:rc_v111}.

Let $t:=\sin((\theta-\theta_*)/2)$ and $a:=\sin((\theta_0-\theta_*)/2)$. From the proof of Theorem \ref{th:v_11_decay}, the change of variable is well defined. Then
\begin{align*}
\frac{\partial v}{\partial r}-\i k v&=\int_{\gamma_1}^{\gamma_2}{\sqrt{t-a}}t^2\zeta(t)e^{-2\i k r t^2}\d t:=(I),
\end{align*}
where
\[
\zeta(t)=-\i k^{5/2}\frac{4}{\sqrt{1-t^2}}\sqrt{\frac{2\arcsin t-2\arcsin a}{t-a}}\psi\left(\arcsin(2t)+\theta_*\right)\in C_0^\infty(\gamma_1,\gamma_2)
\]
 Although it has exactly the same form as $I_2^C(r)$ as defined in \eqref{eq:intc2}, we still need a further study since the decay rate is too low compared to our expectation.  With integration by parts,
 \begin{align*}
(I)&=-\frac{1}{4\i k r}\int_{\gamma_1}^{\gamma_2}t\sqrt{t-a}\zeta(t)\d e^{-2\i k rt^2} \\
%&=\frac{1}{4\i k r}\int_{\gamma_1}^{\gamma_2}\left[\frac{t}{2\sqrt{t-a}}\zeta(t)+\sqrt{t-a}\zeta(t)+t\sqrt{t-a}\zeta'(t)\right]e^{-2\i k rt^2}\d t\\
&:=\frac{a}{\i k r}\int_{\gamma_1}^{\gamma_2}\frac{1}{\sqrt{t-a}}\zeta_1(t)e^{-2\i k r t^2}\d t+\frac{1}{\i k r}\int_{\gamma_1}^{\gamma_2}{\sqrt{t-a}}\zeta_2(t)e^{-2\i k r t^2}\d t\\
&:=\frac{a}{\i k r}(i)+\frac{1}{\i k r}(ii)
\end{align*}  
where
\[
\zeta_1(t)=\frac{1}{8}\zeta(t),\quad
\zeta_2(t)=\frac{3}{8}\zeta(t)+\frac{1}{4}t\zeta'(t)\in C_0^\infty(\gamma_1,\gamma_2).
\]
 Note that we have used the fact that
 \[
\frac{t}{\sqrt{t-a}}=\frac{t-a}{\sqrt{t-a}}+\frac{a}{\sqrt{t-a}}=\sqrt{t-a}+\frac{a}{\sqrt{t-a}}. 
 \]
Oberserve that (i) has exactly the form of $I_1^C(r)$ defined in \eqref{eq:intc1} and (ii) has the form of $I_2^C(r)$ defined in \eqref{eq:intc2}. From Theorem \ref{th:intc1} and Corollary \ref{cr:decay_intc2}, there is a constant $C>0$ such that
\[
|a(i)|,|(ii)|\leq Cr^{-1/2}
\]
holds uniformly for  $|a|<1$. Thus we finally conclude that
\[
\left(\frac{\partial v}{\partial r}-\i k v\right)=O\left(r^{-3/2}\right).
\]
\end{proof}

As a conclusion, when $|\alpha_0+j|<k$ where $\alpha_0\in\S$ and $j\in J_-(\alpha_0)$, 
\[
v(x)=O\left(r^{-1/2}\right),\quad\frac{\partial v}{\partial r}-\i k v=O\left(r^{-3/2}\right).
\]
We will focus on the other two cases in the following subsections.

\subsection{Case II: $|\alpha_0+j|>k$}
\label{sec:v2}

In this case, since $|\alpha+j|>k$, $\i\sqrt{k^2-(\alpha+j)^2}=-\sqrt{(\alpha+j)^2-k^2}$ and there is a $\delta_1>0$ such that
\[
\sqrt{(\alpha+j)^2-k^2}>2\delta_1
\]
holds uniformly for all $j\in J_+(\alpha_0)$ and $\alpha\in[\alpha_0-\delta,\alpha_0+\delta]$. Moreover, note that when $|j|\rightarrow\infty$,
\[
\sqrt{(\alpha+j)^2-k^2}\sim |j|.
\]
 Also recall that $x_2\geq H+h$.  
We start from the definition of $v$:
\begin{equation}
\label{eq:v_12}
v=\int_{\alpha_0-\delta}^{\alpha_0+\delta}\sqrt{\alpha-\alpha_0}\phi(\alpha)e^{\i(\alpha+j) x_1-\sqrt{(\alpha+j)^2-k^2}(x_2-H)}\d \alpha\
\end{equation}
From direct computation,
\begin{align}
\frac{\partial v}{\partial r}-\i k v&=\int_{\alpha_0-\delta}^{\alpha_0+\delta}\left[\i(\alpha+j)\cos\theta_*-\sqrt{(\alpha+j)^2-k^2}\sin\theta_*-\i k \right]\nonumber\\&\cdot\sqrt{\alpha-\alpha_0}\phi(\alpha)e^{\i(\alpha+j) x_1-\sqrt{(\alpha+j)^2-k^2}(x_2-H)}\d \alpha\nonumber\\
\label{eq:rc_v12}
&:=\int_{\alpha_0-\delta}^{\alpha_0+\delta}\sqrt{\alpha-\alpha_0}{\psi}(\alpha)e^{\i(\alpha+j) x_1-\sqrt{(\alpha+j)^2-k^2}(x_2-H)}\d \alpha,
\end{align}
where
\[
\psi(\alpha)=\left[\i(\alpha+j)\cos\theta_*-\sqrt{(\alpha+j)^2-k^2}\sin\theta_*-\i k \right]\phi(\alpha)\in C_0^\infty(\alpha_0-\delta,\alpha_0+\delta).
\]

Since the exponential decay of the term $e^{-\sqrt{(\alpha+j)^2-k^2}(x_2-H)}$ when $x_2$ increases, the case that $x_2\rightarrow\infty$ is relatively easier. This is studied in the following theorem.
 
\begin{theorem}
\label{th:decay_rc_v121}Given $\phi\in C_0^\infty(\alpha_0-\delta,\alpha_0+\delta)$. 
Suppose $x_2-H>\frac{r}{2}$, then there are two  constants $C,\epsilon_0>0$ such that
\begin{equation}
\label{eq:decay_v121}
|v(x)|,\,\left|\frac{\partial v}{\partial r}-\i k v\right|\leq C e^{-\epsilon_0 |j|r}
\end{equation}
holds uniformly for all $j\in J_+(\alpha_0)$. 
\end{theorem}

\begin{proof}
Since $x_2-H>\frac{r}{2}$ and $\sqrt{(\alpha+j)^2-k^2}>2\epsilon_0|j|$, we can immediately get 
\[
|v|\leq \int_{\alpha_0-\delta}^{\alpha_0+\delta}\sqrt{|\alpha-\alpha_0|}|\phi(\alpha)|e^{-\sqrt{(\alpha+j)^2-k^2}(x_2-H)}\d\alpha\leq Ce^{-\epsilon_0|j| r}.
\]
The estimation holds similarly for $\frac{\partial v}{\partial r}-\i k v$. The only difference is that $|\psi|\sim|j|$, which can be merged into the exponential term. The proof is finished.
\end{proof}

The case that $x_2-H$ is smaller than $\frac{r}{2}$ is relatively more challenging, which will be discussed in the following theorem. 

\begin{theorem}
\label{th:decay_rc_v122}Given $\phi\in C_0^\infty(\alpha_0-\delta,\alpha_0+\delta)$. 
Suppose $x_2-H\leq \frac{r}{2}$, then there are two  constants $C,\epsilon_0>0$ such that
\begin{equation}
\label{eq:decay_v122}
|v(x)|,\,\left|\frac{\partial v}{\partial r}-\i k v\right|\leq C e^{-\epsilon_0 |j|h}r^{-3/2}
\end{equation}
holds uniformly for all $j\in J_+$.
\end{theorem}

\begin{proof}
Since $x_2-H\leq \frac{r}{2}$, then $|x_1|\geq \frac{\sqrt{3}}{2}r>\frac{r}{2}$. Note that both \eqref{eq:v_12} and \eqref{eq:rc_v12} have exactly the same form as $I_2^A(r)$ defined by \eqref{eq:inta2}, then from Corollary \ref{cr:decay_inta2},
\[
|v|,\,\left|\frac{\partial v}{\partial r}-\i k v\right|\leq C\left|x_1\right|^{-3/2}
=Cr^{-3/2}
\]
where $C$ is determined by $\left\|\phi(\alpha)e^{-\sqrt{(\alpha+j)^2-k^2}(x_2-H)}\right\|_{C^3(\R)}$, where the norm of $C^3(\R)$ is computed with respect to $\alpha$ with fixed $x$. 

Note that the fact that $\left\|\phi(\alpha)e^{-\sqrt{(\alpha+j)^2-k^2}(x_2-H)}\right\|_{C^3(\R)}$ also depends on $x_2$, which is not bounded when $r\rightarrow\infty$. Note that the leading term with respect to $x_2-H$ has the form of
\[
(x_2-H)^3 e^{-\sqrt{(\alpha+j)^2-k^2}(x_2-H)},
\]
where the coefficients are ignored. Define
\[
f(t):=t^3 e^{-\epsilon_0|j| t},\quad t\geq h.
\]
Since the minimum value of $f$ is achieved at the point $t=\frac{3}{\epsilon_0|j|}$, when $|j|\rightarrow\infty$, $t<h$. Thus 
\[
f(t)\leq f(h)=h^3 e^{-\epsilon_0|j|h},
\]
which finishes the proof.

\end{proof}

\subsection{Case III: $|\alpha_0+j|=k$}
\label{sec:v3}

In this section we consider the particular case that $|\alpha_0+j|=k$ where $j\in\Z$. For $\alpha\in(\alpha_0-\delta,\alpha_0+\delta)$, either $|\alpha+j|\leq k$ or $|\alpha+j|\geq k$. Since both cases are complex and the proofs are long, we will discuss them in two subsections. For simplicity, we assume that $\alpha_0+j=k$, then  (i) when $\alpha\in(\alpha_0-\delta,\alpha_0]$, $\alpha+j\leq k$ (see Section \ref{sec:v13_1}),  (ii) when $\alpha\in[\alpha_0,\alpha_0+\delta)$, $\alpha+j\geq k$ (see Section \ref{sec:v13_2}). The case $\alpha_0+j=-k$
 is very similar thus we will not discuss in this paper.

\subsubsection{Case (i): $\alpha\in(\alpha_0-\delta,\alpha_0]$}
\label{sec:v13_1}

Similar to \eqref{eq:def_vj}, with the polar coordinate and a slightly different definition of $\phi$, 
\begin{align}
v(x)&=\int_0^\beta \sqrt{1-\cos\theta}\sin\theta\phi(k\cos\theta-j)e^{\i k r\cos(\theta-\theta_*)}\d\theta\nonumber\\\label{eq:v_13i}
&:=\int_0^\beta \sin^2\theta\psi(\theta)e^{\i k r\cos(\theta-\theta_*)}\d\theta
\end{align}
where $\beta=\arccos(1-\delta/k)>0$ is a small value, and
\[
\psi(\theta):=\i \sqrt{2}\phi(k\cos\theta-j)\frac{\sin^2(\theta/2)}{\sin^2\theta}\in C^\infty[0,\beta)
\]
which vanishes around $\theta=\beta$. We study the decay of $v$ in the following theorem.

\begin{theorem}
\label{th:decay_v_13i}
Let $\psi\in C^\infty[0,\beta)$ be a function that vanishes around $\beta$. Then there is a constant $C>0$ which only depends on $\psi$ such that
\begin{equation}
\label{eq:decay_v_13i}
|v(x)|\leq Cr^{-1/2}
\end{equation}
holds uniformly for all $x=r(\cos\theta_*,\sin\theta_*)+(0,H)$ with $x_2\geq H+h$.
\end{theorem}

\begin{proof}Due to different types of singularities, we consider the following cases for the location of $\theta_*$  separately: i) $\theta_*\in(2\beta,\pi-2\beta)$; ii) $\theta_*\in[0,2\beta]$; iii) $\theta_*\in[\pi-2\beta,\pi]$.\\

\noindent
i) First consider the case that $\theta_*\in(2\beta,\pi-2\beta)$, which implies that $\left|\sin(\theta-\theta_*)\right|\geq\sin(\beta)>0$ holds uniformly for $\theta\in[0,\beta]$. In this case, with integration by parts,
\begin{align*}
v(x)&=-\frac{1}{\i k r}\int_0^\beta\frac{\sin^2\theta}{\sin(\theta-\theta_*)}\psi(\theta)\d e^{\i k r\cos(\theta-\theta_*)}\\&=\frac{1}{\i k r}\int_0^\beta\left[\frac{\sin^2\theta}{\sin(\theta-\theta_*)}\psi(\theta)\right]'e^{\i k r\cos(\theta-\theta_*)}\d\theta,
\end{align*}
where the integrand is smooth and uniformly bounded. Thus there is a constant $C>0$ such that
\[
|v(x)|\leq Cr^{-1}
\]
holds uniformly for all $\theta\in(2\beta,\pi-2\beta)$. We can even reach a higher order convergence rate by integration by parts again. Note that
\[
v(x)=\frac{1}{r}\int_0^\beta \sin\theta\zeta(\theta)e^{\i k r\cos(\theta-\theta_*)}\d\theta
\]
where
\[
\zeta(\theta)=\frac{1}{\i k}\left[\frac{2\cos\theta}{\sin(\theta-\theta_*)}\psi(\theta)-\sin\theta\frac{\cos(\theta-\theta_*)}{\sin^2(\theta-\theta_*)}\psi(\theta)+\frac{\sin\theta}{\sin(\theta-\theta_*)}\psi'(\theta)\right]\in C^\infty[0,\beta)
\]
which vanishes around $\beta$. Since the boundary values vanish, we can again get
\begin{align*}
v(x)&=-\frac{1}{\i k r^2}\int_0^\beta\frac{\sin\theta}{\sin(\theta-\theta_*)}\zeta(\theta)\d e^{\i k r\cos(\theta-\theta_*)}\\
&=\frac{1}{\i k r}\int_0^\beta\left[\frac{\sin\theta}{\sin(\theta-\theta_*)}\zeta(\theta)\right]' e^{\i k r\cos(\theta-\theta_*)}\d\theta
\end{align*}
which implies that there is a constant $C>0$ such that
\[
|v(x)|\leq Cr^{-2}
\]
holds uniformly for all $\theta\in(2\beta,\pi-2\beta)$.\\

\noindent
ii) Second consider the case that $\theta_*\in[0,2\beta]$.  The technique of integration by parts now fails since $\sin(\theta-\theta_*)$ is close/equal to zero since $|\theta-\theta_*|$ can be very small. In this case, let 
\[
t:=\sin\left(\frac{\theta-\theta_*}{2}\right),\quad a=\sin\left(\frac{\theta_*}{2}\right)
\]
then $\theta=2\arcsin t+\theta_*$ and
\[
v(x)=e^{\i k r}\int_{-a}^{\gamma}(t+a)^2\zeta(t)e^{-2\i k r t^2}\d t
\]
where  $\gamma=\sin((\beta-\theta_*)/2)$ which is close to $0$; and
\[
\zeta(t)=\frac{2}{\sqrt{1-t^2}}\psi(2\arcsin t+\theta_*)\left(\frac{\sin\left(2\arcsin t+\theta_*\right)}{t+a}\right)^2\in C^\infty[-a,\gamma)
\]
which vanishes at $\gamma$. Here we have used the trigonometric dentity
\[
t+a=\sin\left(\frac{\theta-\theta_*}{2}\right)+\sin\left(\frac{\theta_*}{2}\right)=2\sin\left(\frac{\theta}{4}\right)\cos\left(\frac{\theta-2\theta_*}{4}\right),
\]
which implies that
\[
\frac{\sin(\theta)}{t+a}=\frac{\sin\left(\theta\right)}{\sin\left(\frac{\theta-\theta_*}{2}\right)+\sin\left(\frac{\theta_*}{2}\right)}=\frac{\sin\left(\theta\right)}{\sin\left(\theta/4\right)}\frac{1}{2\cos((\theta-2\theta_*)/4)}\in C^\infty[0,2\beta].
\]
Replacing $\theta$ by $2\arcsin t+\theta_*$, then the function is also in $[-a,\gamma]$.

Note that now $v$ has exactly the same form of $I_B^2(r)$ defined by \eqref{eq:intb2}. Using Theorem \ref{th:decay_intb2}, there is a constant $C>0$ such that
\[
|v(x)|\leq Cr^{-1/2}
\]
holds uniformly for $\theta_*\in[0,2\beta]$.\\

\noindent
iii) At the end we consider the case that $\theta_*\in[\pi-2\beta,\pi]$. Let $\theta_*:=\pi-\tau_*$ then $\tau_*\in[0,2\beta]$ then
\[
\cos(\theta-\theta_*)=\cos(\theta+\tau_*-\pi)=-\cos(\theta+\tau_*).
\]
Then
\[
v(x)=\int_0^\beta\sin^2\theta\psi(\theta) e^{-\i k r\cos(\theta+\tau_*)}\d\theta.
\]
Let
\[
t:=\sin\left(\frac{\theta+\tau_*}{2}\right),\quad a=\sin\left(\frac{\tau_*}{2}\right)
\]
then with similar process to Case ii),
\[
v(x)=e^{-\i k r}\int_{a}^{\gamma}(t-a)^2\zeta(t)e^{2\i k r t^2}\d t
\]
where  $\gamma=\sin((\beta+\tau_*)/2)>0$ is close to $0$; and
\[
\zeta(t)=\frac{2}{\sqrt{1-t^2}}\psi(2\arcsin t-\tau_*)\left(\frac{\sin\left(2\arcsin t-\tau_*\right)}{t-a}\right)^2\in C^\infty[a,\gamma)
\]
which vanishes at $\gamma$, since
\[
t-a=\sin\left(\frac{\theta+\tau_*}{2}\right)-\sin\left(\frac{\tau_*}{2}\right)=2\sin\left(\frac{\theta}{4}\right)\cos\left(\frac{\theta-2\tau_*}{4}\right).
\]
Again the form is the same as $I_B^2(r)$ thus we can use Theorem \ref{th:decay_intb2}. Thus we have a constant $C>0$ such that
\[
|v(x)|\leq Cr^{-1/2}.
\]
The proof is then finished.
\end{proof}

Now we move on to the radiation condition. With similar process, we have
\begin{equation}
\label{eq:v13i_rc}
\frac{\partial v}{\partial r}-\i k v=\i k\int_0^\beta\sin^2\theta\left(\cos(\theta-\theta_*)-1\right){\psi}(\theta)e^{\i k r\cos(\theta-\theta_*)}\d\theta.
\end{equation}
We will investigate the decay of \eqref{eq:rc_v13i} in the following theorem.

\begin{theorem}
\label{th:rc_v13i}
Suppose $\psi\in C^\infty[0,\beta)$ which vanishes near $\beta$. Then there is a constant $C>0$ depends on $\psi$ such that
\begin{equation}
\label{eq:rc_v13i}
\left|\frac{\partial v}{\partial r}-\i k v\right|\leq Cr^{-3/2}.
\end{equation}
\end{theorem}

\begin{proof}
Similar to the proof of Theorem \ref{th:decay_v_13i}, we still discuss the problems for the following three cases, i.e, i) $\theta_*\in(2\beta,\pi-2\beta)$; ii) $\theta_*\in[0,2\beta]$; iii) $\theta_*\in[\pi-2\beta,\pi]$, separately.\\ 

\noindent 
i) First consider $\theta_*\in(2\beta,\pi-2\beta)$,  where the decay rate $O(r^{-2})$ comes directly from integration by parts, following the proof of i)  Theorem \ref{th:decay_v_13i}.\\

\noindent
ii) Then we consider $\theta_*\in[0,2\beta]$. In this case, both $\theta$ and $\theta_*$ are close to $0$. With integration by parts, following the proof of Theorem  \ref{th:rc_v111},
\begin{align*}
\frac{\partial v}{\partial r}-\i k v&=\frac{1}{r}\int_0^\beta \sin^2\theta\tan\left(\frac{\theta-\theta_*}{2}\right)\psi(\theta) \d e^{\i k r\cos(\theta-\theta_*)}\\
&=\frac{1}{r}\int_0^\beta\left[\sin(\theta) \sin\left({\theta-\theta_*}\right)\zeta_1(\theta)+\sin^2\theta\zeta_2(\theta)\right]
e^{\i k r\cos(\theta-\theta_*)}\d\theta\\
&:=\frac{1}{r}(I)+\frac{1}{r}(II)
\end{align*}
where
\begin{align*}
\zeta_1(\theta)&=-2\cos\theta\frac{\tan((\theta-\theta_*)/2)}{\sin(\theta-\theta_*)}\psi(\theta)-\sin\theta \frac{\tan((\theta-\theta_*)/2)}{\sin(\theta-\theta_*)}\psi'(\theta);\\
\zeta_2(\theta)&=-\frac{1}{\cos(\theta-\theta_*)+1}\psi(\theta).
\end{align*}
Both functions are $C^\infty[0,\beta)$ and vanish around $\beta$. 
where $\zeta_1$ and $\zeta_2$ are smooth functions and any order derivative equals to $0$ at $\beta$. 
We need to consider the above integrals separately. For (I), again with integration by parts,
\begin{align*}
|(I)|
=&\left|-\frac{1}{\i k r}\int_0^\beta\sin\theta\zeta_1(\theta)\d e^{\i k r \cos(\theta-\theta_*)}\right|\\
\leq&\frac{1}{\i k r}\int_0^\beta\left|\cos\theta\zeta_1(\theta)+\sin\theta\zeta'_1(\theta)\right|\d\theta\leq Cr^{-1}.
\end{align*}
For (II), following the proof of Theorem \ref{th:decay_v_13i} to get
\[
|(II)|\leq Cr^{-1/2}
\]
holds uniformly. Thus we conclude that for $\theta_*\in[0,2\beta]$,
\[
\frac{\partial v}{\partial r}-\i k v=O\left(r^{-3/2}\right).
\]
The proof for Case ii) is then finished.\\

\noindent
iii) When $\theta_*\in[\pi-2\beta,\pi]$, the above strategy fails since $\cos(\theta-\theta_*)\approx-1$ then $\zeta_2(\theta)$ becomes singular. In this case, following again Case iii) of the proof of Theorem \ref{th:decay_v_13i}, let $\theta_*:=\pi-\tau_*$ then $\tau_*\in[0,2\beta]$. Then with integration by parts,
\begin{align*}
\frac{\partial v}{\partial r}-\i k v&=-\i k\int_0^\beta\sin^2\theta\left(\cos(\theta+\tau_*)+1\right){\psi}(\theta)e^{-\i k r\cos(\theta+\tau_*)}\d\theta\\
&=-\frac{1}{r}\int_0^\beta\sin^2\theta\frac{\cos(\theta+\tau_*)+1}{\sin(\theta+\tau_*)}{\psi}(\theta)\d e^{-\i k r\cos(\theta+\tau_*)}\\
&=\frac{1}{r}\int_0^\beta\left[\eta(\theta)\zeta_1(\theta)+\eta^2(\theta)\zeta_2(\theta)\right]e^{-\i k r\cos(\theta+\tau_*)}\d\theta\\
&:=\frac{1}{r}(I)+\frac{1}{r}(II),
%&=-\frac{\sqrt{2}}{kr}\int_0^\beta\left[\left(2\sin\theta\cos(\theta)\frac{\cos(\theta+%%\tau_*)+1}{\sin(\theta+\tau_*)}+\frac{\sin^2\theta}{\cos(\theta+\tau_*)-1}
%\right){\psi}(\theta)\right.\\&\qquad+\left.\sin^2\theta\frac{\cos(\theta+\tau_*)+1}%{\sin(\theta+\tau_*)}{\psi}'(\theta)\right] e^{-\i k r\cos(\theta+\tau_*)}\d\theta
\end{align*}
where
\[
\eta(\theta)=\frac{\sin\theta}{\sin(\theta+\tau_*)}\in C^\infty[0,\beta),\quad |\eta(\theta)|\leq 1,\,\tau_*\in[0,2\beta],\,\theta\in[0,\beta]
\]
and is strictly increasing for $\theta\in[0,\beta)$; and
\begin{align*}
\zeta_1(\theta)&=2\cos\theta[\cos(\theta+\tau_*)+1]\psi(\theta)+\sin\theta[\cos(\theta+\tau_*)+1]\psi'(\theta);\\
\zeta_2(\theta)&=-\frac{\sin^2(\theta+\tau_*)}{2\sin^2((\theta+\tau_*)/2)}\psi(\theta)
\end{align*}
where $\zeta_1,\,\zeta_2\in C^\infty[0,\beta)$ and vanish around $\beta$. 

For the term (I), to avoid potential singularity at $0$, we first define
\[
(I)=\int_{0}^{r^{-1/2}}+\int_{r^{-1/2}}^\beta\eta(\theta)\zeta_1(\theta)e^{-\i k r\cos(\theta+\tau_*)}\d\theta:=(I_1)+(I_2).
\]
From direct computation,
\[
|(I_1)|\leq \int_0^{r^{-1/2}}|\eta(\theta)||\zeta_1(\theta)|\d\theta\leq Cr^{-1/2}.
\]
For the term $(I_2)$, we need to change the variable. Let  $t:=\sin((\theta+\tau_*)/2)$,
\begin{align*}
(I_2)=e^{-\i k r}\int_{\gamma_1}^{\gamma_2}\lambda(t)\rho_1(t)e^{2\i k r t^2}\d t
\end{align*}
where $\gamma_1=\sin((r^{-1/2}+\tau_*)/2)$ and $\gamma_2=\sin((\beta+\tau_*)/2)$, and
\[
\lambda(t)=\eta\left(2\arcsin t-\tau_*\right)\in C^\infty[\lambda_1,\lambda_2),\,\rho_1(t)=\frac{2}{\sqrt{1-t^2}}\zeta_1\left(2\arcsin t-\tau_*\right)\in C^\infty[\lambda_1,\lambda_2)
\]
and $\rho_1$ vanishes around $\gamma_2$. With integration by parts,
\begin{align*}
(I_2)&=\frac{e^{-\i k r}}{4\i k r}\int_{\gamma_1}^{\gamma_2}\frac{1}{t}\lambda(t)\rho_1(t)\d e^{2\i k r t^2}\\
&=\frac{e^{-\i k r}}{4\i k r}\left.\frac{1}{t}\lambda(t)\rho_1(t) e^{2\i k r t^2}\right|^{\gamma_2}_{\gamma_1}-\frac{e^{-\i k r}}{4\i k r}\int_{\gamma_1}^{\gamma_2}\left[\frac{\lambda(t)\rho'_1(t)+\lambda'(t)\rho_1(t)}{t}-\frac{\lambda(t)\rho_1(t)}{t^2}\right]e^{2\i k r t^2}\d t.
\end{align*}
Since $\gamma_1\geq \frac{1}{2}r^{-1/2}$ and $\gamma_2>\gamma_1$, we can estimate immediately that
\[
\left|\frac{e^{-\i k r}}{4\i k r}\left.\frac{1}{t}\lambda(t)\rho_1(t) e^{2\i k r t^2}\right|^{\gamma_2}_{\gamma_1}\right|\leq Cr^{-1/2}.
\]
Since $\lambda$, $\rho_1$ and $\rho'_1$ are uniformly bounded in $[0,\beta]$,
\[
\left|\int_{\gamma_1}^{\gamma_2}\left[\frac{\lambda(t)\rho'_1(t)}{t}-\frac{\lambda(t)\rho_1(t)}{t^2}\right]e^{2\i k r t^2}\d t\right|\leq C\int_{\gamma_1}^{\gamma_2}\frac{1}{t^2}\d t=\left.-\frac{C}{t}\right|^{\gamma_2}_{\gamma_1}\leq Cr^{1/2}.
\]
Now we move on to the term in $(I_2)$ that involves $\lambda'(t)$. From direct computation,
\[
\eta'(\theta)=\frac{2\sin\tau_*}{1-\cos(2\theta+2\tau_*)}=\frac{\sin\tau_*}{\sin^2(\theta+\tau_*)}\rightarrow\frac{1}{\sin\tau_*},\quad \theta\rightarrow 0.
\]
It means that when $\tau_*$ is close to $0$, the value will become very large thus $\lambda'$ is not uniformly bounded. But note that $\eta'(\theta)>0$ for any $\theta\geq 0$, it implies that $\lambda$ is strictly increasing. Thus
\begin{align*}
\left|\int_{\gamma_1}^{\gamma_2}\frac{\lambda'(t)\rho_1(t)}{t}e^{2\i k r t^2}\d t\right|&\leq C\int_{\gamma_1}^{\gamma_2}\frac{\lambda'(t)}{t}\d t\\&\leq C\sqrt{r}\int_{\gamma_1}^{\gamma_2}\lambda'(t)\d t=C\sqrt{r}\left[\lambda(\gamma_2)-\lambda(\gamma_2)\right]\leq C{r}^{1/2}.
\end{align*}
Thus we finally get that 
\[
|(I_2)|\leq Cr^{-1/2}.
\]
Together with the estimation of $(I_1)$, we conclude that
\[
|(I)|\leq Cr^{-1/2}.
\]

Now we move on to the term $(II)$. We can simply treat $\eta^2$ as $\eta$ as in $(I)$ to reach the same result, since $\eta^2$ is again a strictly increasing function. So finally we can still prove that
\[
\left|\frac{\partial v}{\partial r}-\i k v\right|\leq Cr^{-3/2}.
\]

\end{proof}

\subsubsection{Case (ii): $\alpha\in[\alpha_0,\alpha_0+\delta)$}
\label{sec:v13_2}
Then we consider the case that $\alpha\in[\alpha_0,\alpha_0+\delta)$. In this case, $\i\sqrt{k^2-(\alpha+j)^2}=-\sqrt{(\alpha+j)^2-k^2}$. Let $t:=\alpha+j-k$ then
\begin{equation}
\label{eq:v_13ii}
v(x)=\int_{0}^\delta\sqrt{t}\phi(t)e^{\i(t+k)x_1-\sqrt{t^2+2kt}(x_2-H)}\d t.
\end{equation}
With $x=r(\cos\theta_*,\sin\theta_*)$,
\begin{equation}
\label{eq:rc_v_13ii}
\frac{\partial v}{\partial r}-\i k v
=e^{\i k x_1}\int_0^\delta \left[\i (t+k)\cos\theta_*-\sqrt{t^2+2kt}\sin\theta_*-\i k \right] \sqrt{t}\phi(t)e^{\i t x_1-\sqrt{t^2+2kt}(x_2-H)}\d t.
\end{equation}

Similar as in Section \ref{sec:v2}, the case that $x_2$ is larger is easiler. Thus for the estimation of \eqref{eq:v_13ii} and \eqref{eq:rc_v_13ii}, we go to the following theorem.

\begin{theorem}
\label{th:decay_v13ii1}
Suppose $x_2-H\geq\sqrt{r}$ and $\phi\in C^\infty[0,\delta)$ and vanishes at $\delta$. There is a constant $C>0$ only depends on $\phi$ such that
\begin{equation}
\label{eq:decay_v13ii1}
|v(x)|,\,\left|\frac{\partial v}{\partial r}-\i k v\right|\leq Cr^{-3/2}.
\end{equation}
\end{theorem}

\begin{proof}
Since $x_2-H\geq \sqrt{r}$ and $\sqrt{t^2+2kt}\geq\sqrt{2kt}$, we have
\[
\exp\left(-\sqrt{t^2+2kt}(x_2-H)\right)\leq \exp\left(-\sqrt{2kt r}\right).
\]
From direct computation
\begin{align*}
\left|v(x)\right|&\leq C\int_0^\delta\sqrt{t}e^{-\sqrt{2krt}}\d t=\frac{4}{(2kr)^{3/2}}
\left[1-\left(kr\delta+\sqrt{2kr\delta}+1\right)e^{-\sqrt{2kr\delta}}\right]
\leq Cr^{-3/2}.
\end{align*}
Here we use the formula
\[
\int \sqrt{t}e^{-\gamma\sqrt{t}}\d t=-\frac{2(\gamma^2 t+2\gamma\sqrt{t}+2)}{\gamma^3}e^{-\gamma\sqrt{t}}+C.
\]
This result holds also for $\frac{\partial v}{\partial r}-\i k v$, thus the proof is finished.
\end{proof}

In the following theorem, we will deal with the more challenging part, i.e., $x_2-H<\sqrt{r}$.

\begin{theorem}
\label{th:decay_v13ii2}
Suppose $x_2-H<\sqrt{r}$ and $\phi\in C^\infty[0,\delta)$ and vanishes at $\delta$. Then there is a constant $C>0$ such that
\begin{equation}
\label{eq:decay_v13ii2}
|v(x)|,\,\left|\frac{\partial v}{\partial r}-\i k v\right|\leq Cr^{-3/2}.
\end{equation}
\end{theorem}

\begin{proof}

When $x_2-H<\sqrt{r}$, then $|x_1|\geq r/2$, we need to split the integral into two parts:
\begin{align*}
\frac{\partial v}{\partial r}-\i k v
&=e^{\i k x_1}\int_0^\delta \sqrt{t}\psi_1(t)e^{\i t x_1-\sqrt{t^2+2kt}(x_2-H)}\d t+e^{\i k x_1}\int_0^\delta {t}\psi_2(t)e^{\i t x_1-\sqrt{t^2+2kt}(x_2-H)}\d t\\&:=e^{\i k x_1}(I)+e^{\i k x_1}(II),
\end{align*}
where $\psi_1$ and $\psi_2$ only depend on $\theta_*$ and $k$. Both functions are $C^\infty[0,\delta)$ and vanishes around $\delta$. Note that $v$ defined by \eqref{eq:v_13ii} has exactly the same form as $(I)$, thus the discussion is included in the proof of $\frac{\partial v}{\partial r}-\i k v$.

For simplicity, let $\xi:=x_2-H$ thus $h<\xi<\sqrt{r}$. We begin with part $(I)$. From integration by parts,
\begin{align*}
(I)&=\frac{1}{\i x_1}\int_0^\delta \sqrt{t}\psi_1(t) e^{-\xi\sqrt{t^2+2kt}}\d e^{\i t x_1}\\
%&=\frac{1}{\i x_1}\int_0^\delta\left[ \left(\frac{1}{2\sqrt{t}}-\xi\frac{t+k}{\sqrt{t+2k}}\right)\psi_1(t)+\sqrt{t}\psi'_1(t)\right]e^{\i t x_1-\xi\sqrt{t^2+2kt}} \d t\\
&:=\frac{1}{\i x_1}\int_0^\delta \left[\frac{1}{\sqrt{t}}\zeta_1(t)+\xi\zeta_2(t)\right]e^{\i t x_1-\xi\sqrt{t^2+2kt}} \d t\\&:=\frac{1}{\i x_1}(I_1)+\frac{\xi}{\i x_1}(I_2)
\end{align*}
where
\begin{align*}
\zeta_1(t)=-\frac{1}{2}\psi_1(t)-t\psi'_1(t),\quad
\zeta_2(t)=\frac{t+k}{\sqrt{t+2k}}\psi_1(t)\quad\in C^\infty[0,\delta)
\end{align*}
and both functions vanish around $\delta$. For $(I_1)$, we first split it into two parts due to the singularity:
\begin{align*}
(I_1)=\int_0^{1/r}+\int_{1/r}^\delta \frac{1}{\sqrt{t}}\psi_1(t)e^{\i t x_1-\xi\sqrt{t^2+2kt}} \d t:=(I_{11})+(I_{12}).
\end{align*}
For the term $(I_{11})$, from direct computation,
\[
\left|(I_{11})\right|\leq C\int_0^{1/r}\frac{1}{\sqrt{t}}\d t=Cr^{-1/2}.
\]
For the term $(I_{12})$, with integration by parts again,
\begin{align*}
(I_{12})=&\frac{1}{\i x_1}\int_{1/r}^\delta \frac{1}{\sqrt{t}}\psi_1(t)e^{-\xi\sqrt{t^2+2kt}} \d e^{\i t x_1}\\=&-\frac{\sqrt{r}}{\i x_1}\psi_1(0)+\frac{1}{\i x_1}\int_{1/r}^\delta\left[\frac{1}{2\sqrt{t}^3}\psi_1(t)-\frac{1}{\sqrt{t}}\psi'_1(t)+\frac{\xi(t+k)}{\sqrt{t}\sqrt{t^2+2kt}}\psi_1(t)\right]e^{\i t x_1-\xi\sqrt{t^2+2kt}}\d t.
\end{align*}
Consider the later integral separately. First estimate
\begin{align*}
\left|\int_{1/r}^\delta\left[-\frac{1}{2\sqrt{t}^3}\psi_1(t)+\frac{1}{\sqrt{t}}\psi'_1(t)\right]e^{\i t x_1-\xi\sqrt{t^2+2kt}}\d t\right|
\leq C\int_{1/r}^\delta t^{-3/2}\d t\leq C \sqrt{r}.
\end{align*}
Since $\sqrt{t^2+2kt}\geq \sqrt{2kt}$, with $t\geq 1/r$,
\begin{align*}
\left|\int_{1/r}^\delta\frac{\xi(t+k)}{\sqrt{t}\sqrt{t^2+2kt}}\psi_1(t)e^{\i t x_1-\xi\sqrt{t^2+2kt}}\d t\right|
\leq C\sqrt{r}\int_{1/r}^\delta \frac{\xi}{\sqrt{2kt}}e^{-\xi\sqrt{2kt}}\d t\leq C\sqrt{r}
\end{align*}
where $C$ does not depend on the parameter $c>0$. Here we use the fact that
\[
\int \frac{1}{\sqrt{\lambda t}}e^{-\xi\sqrt{\lambda t}}\d t=-\frac{2 e^{-\xi\sqrt{\lambda t}}}{\lambda\xi}+C.
\]
Thus there is a constant $C>0$ such that
\[
\left|(I_{12})\right|\leq Cr^{-1/2}\quad\Rightarrow\quad |(I_1)|\leq Cr^{-1/2}.
\]

Now we move on to $(I_2)$. With integration by parts,
\begin{align*}
(I_2)=&\frac{1}{\i x_1}\int_0^{\delta} \psi_2(t) e^{-\xi\sqrt{t^2+2kt}}\d e^{\i t x_1}\\=&-\frac{1}{\i x_1}\psi_2(0)-\frac{1}{\i x_1}\int_0^\delta \left(\psi'_2(t)-\frac{\xi(t+k)}{\sqrt{t^2+2kt}}\psi_2(t)\right)e^{\i t x_1-\xi\sqrt{t^2+2kt}}\d t.
\end{align*}
With similar arguments,
\begin{align*}
\left|\int_0^\delta \left(\psi'_2(t)-\frac{\xi(t+k)}{\sqrt{t^2+2kt}}\psi_2(t)\right)e^{\i t x_1-c\sqrt{t^2+2kt}}\d t\right|\leq C+C\int_0^\delta \frac{\xi}{\sqrt{2kt}}e^{-\xi\sqrt{2kt}}\d t\leq C,
\end{align*}
we finally have
\[
|(I_2)|\leq C\frac{1}{|x_1|}\quad\Rightarrow\quad |(I)|\leq Cr^{-3/2}
\]
since $|\xi|\leq\sqrt{r}$.

At the end we move on to $(II)$. From integration by parts,
\begin{align*}
(II)&=\frac{1}{\i x_1}\int_0^\delta {t}\psi_2(t)e^{-\xi\sqrt{t^2+2kt}}\d e^{\i t x_1}\\
&:=\frac{1}{\i x_1}\int_0^\delta \left[\zeta_3(t)+\xi\sqrt{t}\zeta_4(t)\right]e^{\i t x_1-\xi\sqrt{t^2+2kt}}\d t\\
&:=\frac{1}{\i x_1}(II_1)+\frac{\xi}{\i x_1}(II_2),
\end{align*}
where
\begin{align*}
\zeta_3(t)=-\psi_2(t)-t\psi'_2(t),\quad \zeta_4(t)=\frac{t+k}{\sqrt{t+2k}}\psi_2(t)\quad\in C^\infty[0,\delta)
\end{align*}
which vanish around $\delta$. For the term $(II_1)$, since it has exactly the same form of $(I_2)$, we can easily estimate that
\[
|(II_1)|\leq Cr^{-1}.
\]
For $(II_2)$, with integration by parts and the inequality $\sqrt{k^2+2kt}\geq\sqrt{2kt}$,
\begin{align*}
|(II_2)|&=\left|\frac{1}{\i x_1}\int_0^\delta \sqrt{t}\zeta_4(t)e^{-\xi\sqrt{t^2+2kt}}\d e^{\i t x_1}\right|\\
&=\left|\frac{1}{\i x_1}\int_0^\delta\left[\frac{1}{2\sqrt{t}}\zeta_4(t)+\sqrt{t}\zeta'_4(t)-\frac{\xi(t+k)}{\sqrt{t+2k}}\right]e^{\i t x_1-\xi\sqrt{t^2+2kt}}\d t\right|\\
&\leq Cr^{-1}\int_0^\delta \frac{1}{\sqrt{t}}\d t+Cr^{-1}\int_0^\delta \frac{\xi}{\sqrt{2kt}} e^{-\xi\sqrt{2kt}}\d t\leq Cr^{-1}
\end{align*}
Thus finally we have the esstimation that
\begin{align*}
|(II)|\leq Cr^{-2}.
\end{align*}
Combine the results for $(I)$, the proof is finished.

\end{proof}

\section{Radiation condition for $u_2$}
\label{sec:rc_2}

 In this section, we will consider $v_2$ be defined by \eqref{eq:def_v} where $u$ be given by \eqref{eq:def_u2}. Following the arguments at the beginning of Section \ref{sec:rc_1}, we only need to consider the radiation conditions for functions defined
\begin{equation}
\label{eq:def_v2}
v(x):=\int_{\alpha_0-\delta}^{\alpha_0+\delta}|\alpha-\alpha_0|\phi(\alpha) e^{\i(\alpha+j)x_1+\i\sqrt{k^2-(\alpha+j)^2}(x_2-H)}\d\alpha,
\end{equation}
where  $\phi\in C_0^\infty(\alpha_0-\delta,\alpha_0+\delta)$. Similar to Section \ref{sec:rc_1}, we will consider different $j$'s in the following three cases.. 
\begin{itemize}
\item In Section \ref{sec:v1_2}, we consider Case I: $|\alpha_0+j|< k$. These $j$'s define a set $J_-$.
\item In Section \ref{sec:v2_2}, we consider Case II: $|\alpha_0+j|> k$.. These $j$'s define a set $J_+$.
\item In Section \ref{sec:v3_2}, we consider Case III: $|\alpha_0+j|=k$. 
\end{itemize}
Since the formulas in this section is very similar to those in Section \ref{sec:rc_1}, we will only focus on the differences.

\subsection{Case I: $|\alpha_0+j|<k$}
\label{sec:v1_2}

Similar to that in Section \ref{sec:v1}, we need to estimate the decay of the function
\begin{equation}
\label{eq:v1_2}
v(x)=\int_{\beta_1}^{\beta_2}|\theta-\theta_0|\psi(\theta)e^{\i k r \cos(\theta-\theta_*)}\d\theta
\end{equation}
where 
\[
\psi(\theta)=-k^2\sin\theta\phi(k\cos\theta-j)\left|\frac{\cos\theta-\cos\theta_0}{\theta-\theta_0}\right|\in C_0^\infty(\beta_1,\beta_2).
\]
The decay of the function is described in the following theorem.

\begin{theorem}
\label{th:v1_2_decay}
Given $\psi\in C_0^\infty(\beta_1,\beta_2)$, there is a constant $C>0$ that only depends on $\psi$ such that
\begin{equation}
\label{eq:v1_2_decay}
|v(x)|\leq Cr^{-1/2}
\end{equation}
holds uniformly  for all $x\in U_{H+h}$.
\end{theorem}

\begin{proof}
Use the formula $\cos(\theta-\theta_*)=1-2\sin^2((\theta-\theta_*)/2)$ and let $t:=\sin((\theta-\theta_*)/2)\in[\gamma_1,\gamma_2]\subset(-1,1)$, then $\theta=2\arcsin t+\theta_*$ and
\[
v(x)=e^{\i k r}\int_{\gamma_1}^{\gamma_2}|t-a|\zeta(t)e^{-2\i kr t^2}\d t
\]
where $a=\sin((\theta_0-\theta_*)/2)$ and
\begin{align*}
\zeta(t)&=\frac{2}{\sqrt{1-t^2}}\left|\frac{2\arcsin t-2\arcsin a}{t-a}\right|\psi(2\arcsin t+\theta_*)\\&=\frac{2}{\sqrt{1-t^2}}\frac{2\arcsin t-2\arcsin 2}{t-a}\psi(2\arcsin t+\theta_*)\in C_0^\infty(\gamma_1,\gamma_2).
\end{align*}
Since the formula has exactly the same form as $I^D_1(r)$ defined by \eqref{eq:intd1}, from Theorem \ref{th:intd1} we have
\[
\left|I^D_1(r)\right|\leq Cr^{-1/2},
\]
which finishes the proof of this theorem. 
\end{proof}

Now we move on to the radiation condition of $v$, which is given by
\begin{equation}
\label{eq:v1_2_rc}
\frac{\partial v}{\partial r}-\i k v=\i k\int_{\beta_1}^{\beta_2}|\theta-\theta_0|\left(\cos(\theta-\theta_*)-1\right)\psi(\theta)e^{\i k r \cos(\theta-\theta_*)}\d\theta.
\end{equation}

\begin{theorem}
\label{th:v1_2_rc}
Given $\psi\in C_0^\infty(\beta_1,\beta_2)$, there is a constant $C>0$ which only depends on $\psi$ such that
\begin{equation}
\label{eq:rc_v1_2}
\left|\frac{\partial v}{\partial r}-\i k v\right|\leq Cr^{-3/2}
\end{equation}
holds uniformly for $x\in U_{H+h}$.
\end{theorem}

\begin{proof}
Following Theorem \ref{th:rc_v111}, with integration by parts,
\begin{align*}
\frac{\partial v}{\partial r}-\i k v&=\frac{1}{r}\int_{\beta_1}^{\beta_2}|\theta-\theta_0|\tan\left(\frac{\theta-\theta_*}{2}\right)\psi(\theta)\d e^{\i k r \cos(\theta-\theta_*)}\\&=-\frac{1}{r}\int_{\beta_1}^{\beta_2}\left[|\theta-\theta_0|\tan\left(\frac{\theta-\theta_*}{2}\right)\psi(\theta)\right]'e^{\i k r \cos(\theta-\theta_*)}\d\theta\\
&:=\frac{1}{r}\int_{\beta_1}^{\beta_2}|\theta-\theta_0|\zeta_1(\theta)e^{\i k r\cos(\theta-\theta_*)}\d\theta +\frac{1}{ r}\int_{\beta_1}^{\beta_2}{\rm sgn}(\theta-\theta_0)\zeta_2(\theta)e^{\i k r\cos(\theta-\theta_*)}\d\theta\\
&:=\frac{1}{r}(I)+\frac{1}{r}(II),
\end{align*}
where
\begin{align*}
\zeta_1(\theta)&=-\tan\left(\frac{\theta-\theta_*}{2}\right)\psi'(\theta)-\frac{1}{\cos(\theta-\theta_*)+1}\psi(\theta);\\
\zeta_2(\theta)&=-\tan\left(\frac{\theta-\theta_*}{2}\right)\psi(\theta);
\end{align*}
and both functions belong to $C_0^\infty(\gamma_1,\gamma_2)$.

For both integrals, let $t:=\sin((\theta-\theta_*)/2)$ and $a:=\sin((\theta_0-\theta_*)/2)$, then
\begin{align*}
(I)&=e^{\i r t}\int_{\gamma_1}^{\gamma_2}|t-a|\rho_1(t)e^{-2\i k r t^2}\d t;\\
(II)&=e^{\i r t}\int_{\gamma_1}^{\gamma_2}{\rm sgn}(t-a)\rho_2(t)e^{-2\i k r t^2}\d t;
\end{align*}
where
\begin{align*}
\rho_1(t)&=\frac{2}{\sqrt{1-t^2}}\frac{2\arcsin t-2\arcsin a}{t-a}\zeta_1(2\arcsin t+\theta_*)\in C_0^\infty(\gamma_1,\gamma_2);\\
\rho_2(t)&=\frac{2}{\sqrt{1-t^2}}\zeta_2(2\arcsin t+\theta_*)\in C_0^\infty(\gamma_1,\gamma_2).
\end{align*}
 Since $(I)$ has the form of $I^D_1(r)$ defined by \eqref{eq:intd1} and $(II)$ has the form of $I^D_2(r)$ define by \eqref{eq:intd2}, from the results in Theorem \ref{th:intd1} and Corollary \ref{cr:decay_intd2}, 
\[
|(I)|,\,|(II)|\leq Cr^{-1/2}.
\]
The proof is then finished.

\end{proof}

\subsection{Case II: $|\alpha_0+j|>k$}
\label{sec:v2_2}

This section is very similar to Section \ref{sec:v2} thus the details are  omitted here. We only list the results as follows.

\begin{theorem}
\label{th:decay_rc_v221}Given $\phi\in C_0^\infty(\alpha_0-\delta,\alpha_0+\delta)$. 
Suppose $x_2-H>\frac{r}{2}$, then there are two  constants $C,\epsilon_0>0$ such that
\begin{equation}
\label{eq:decay_v221}
|v(x)|,\,\left|\frac{\partial v}{\partial r}-\i k v\right|\leq C e^{-\epsilon_0 |j|r}
\end{equation}
holds uniformly for all $j\in J_+(\alpha_0)$ and $x\in U_{H+h}$. 
\end{theorem}

\begin{theorem}
\label{th:decay_rc_v222}Given $\phi\in C_0^\infty(\alpha_0-\delta,\alpha_0+\delta)$. 
Suppose $x_2-H\leq \frac{r}{2}$, then there are two  constants $C,\epsilon_0>0$ such that
\begin{equation}
\label{eq:decay_v222}
|v(x)|,\,\left|\frac{\partial v}{\partial r}-\i k v\right|\leq C e^{-\epsilon_0 |j|h}r^{-2}
\end{equation}
holds uniformly for all $j\in J_+$ and $x\in U_{H+h}$.
\end{theorem}

\subsection{Case III: $|\alpha_0+j|=k$}
\label{sec:v3_2}

Similar to Section \ref{sec:v3}, we still assume that $\alpha_0+j=k$ for some $j\in\Z$, and consider the two cases, i.e., $\alpha\in(\alpha_0-\delta,\alpha_0]$ and $\alpha\in[\alpha_0,\alpha_0+\delta)$ separately.

When $\alpha\in(\alpha_0-\delta,\alpha_0]$, $\alpha+j\leq k$. In this case, similar to \eqref{eq:v_13i}
\begin{align}
v(x)&=\int_0^\beta|\cos\theta-1|\sin\theta \phi(k\cos\theta-j)e^{\i k r\cos(\theta-\theta_*)}\d\theta\nonumber\\
\label{eq:v_13i_2}
&=\int_0^\beta \sin^3\theta\psi(\theta)e^{\i k r \cos(\theta-\theta_*)}\d\theta 
\end{align}
where
\[
\psi(\theta)=2k^2\phi(k\cos\theta-j)\frac{\sin^2(\theta/2)}{\sin^2\theta}\in C^\infty[0,\beta)
\]
which vanishes around $\beta$. Note that the form is exactly as \eqref{eq:v_13i}, thus 
Theorem \ref{th:decay_v_13i} also holds. 

With similar technique, we can also get the form of the radiation condition
\begin{equation}
\label{eq:v_13i_2_rc}
\frac{\partial v}{\partial r}-\i k v=\i k\int_0^\beta \sin^2\theta \sin\theta\left(\cos(\theta-\theta_*)-1\right)\psi(\theta)e^{\i k r\cos(\theta-\theta_*)}\d\theta.
\end{equation}
which again has exactly the same form of \eqref{eq:v13i_rc}. Thus Theorem \ref{th:rc_v13i} can again be applied to get the estimation. We conclude as the results in the following corollary.

\begin{corollary}
\label{cr:decay_v_13i_2}
Let $\psi\in C^\infty[0,\beta)$ which vanishes round $\beta$. Then there is a constant $C>0$ which only depends on $\psi$ such that
\begin{equation}
\label{eq:decay_v_13i_2}
|v(x)|\leq Cr^{-1/2},\quad \left|\frac{\partial v}{\partial r}-\i k v\right|\leq Cr^{-3/2}
\end{equation}
holds  uniformly for $x\in U_{H+h}$
\end{corollary}

Then we move on to the case that $\alpha+j>k$ and $t:=\alpha+j-k$, $\xi:=x_2-H$ we have
\begin{equation}
\label{eq:v_13ii_2}
v(x)=\int_0^\delta t\psi(t)e^{\i(t+k)x_1-\xi\sqrt{t^2+2kt}}\d t.
\end{equation}
Then we get immediately
\begin{equation}
\label{eq:v_13ii_rc}
\frac{\partial v}{\partial r}-\i k v=\int_0^\delta \left[\i(t+k)\cos\theta_*-\sqrt{t^2+2kt}\sin\theta_*-\i k \right]t\psi(t)e^{\i t x_1-c\sqrt{t^2+2kt}}\d t.
\end{equation}

The case that $x_2-H$ is larger that $\sqrt{r}$ is easier. Following Theorem \ref{th:decay_v13ii1}, we have the following estimation.

\begin{corollary}
\label{cr:decay_v13ii1_2}
Suppose $x_2-H\geq\sqrt{r}$ and $\phi\in C^\infty[0,\delta)$ that vanishes around $\delta$. There is a $C>0$ depends on $\phi$ such that
\begin{equation}
\label{eq:decay_v13ii1_2}
|v(x)|,\,\left|\frac{\partial v}{\partial r}-\i k v\right|\leq Cr^{-2}
\end{equation}
holds uniformly for   $x\in U_{H+h}$.
\end{corollary}

\begin{proof}
Since $\xi\geq\sqrt{r}$ and $\sqrt{t^2+2kt}\geq \sqrt{2kt}$, we have that
\[
|v(x)|\leq C\int_0^\delta t e^{-\sqrt{2krt}}\d t=\frac{3}{k^2r^2}-\frac{e^{-\sqrt{2kr\delta}}}{k^2 r^2}\left[(3+kr\delta)\sqrt{2kr\delta}+3kr\delta+3\right]\leq Cr^{-2}
\]
where we use the formula
\[
\int t e^{-\sqrt{\gamma t}}\d t=-\frac{2(\lambda t\sqrt{\lambda t}+3\lambda t+6\sqrt{\lambda t}+6)}{\lambda^2}e^{-\sqrt{\lambda t }}+C.
\]
The same result also holds for $\frac{\partial v}{\partial r}-\i k v$. The proof is finished.
\end{proof}

The last topic in this section is the case that $\xi<\sqrt{r}$. This will be concluded as the Corollary of Theorem \ref{th:decay_v13ii2}.

\begin{corollary}
\label{cr:decay_v13ii2_2}
Suppose $x_2-H<\sqrt{r}$ and $\phi\in C^\infty[0,\delta)$ that vanishes around $\delta$. There is a $C>0$ depends on $\phi$ such that
\begin{equation}
\label{eq:decay_v13ii2_2}
|v(x)|,\,\left|\frac{\partial v}{\partial r}-\i k v\right|\leq Cr^{-3/2}
\end{equation}
holds uniformly for  $x\in U_{H+h}$.
\end{corollary}

\begin{proof}
Similar as in Theorem \ref{th:decay_v13ii2}, we only need to consider $\frac{\partial v}{\partial r}-\i k v$. From direct computation,
\begin{align*}
\frac{\partial v}{\partial r}-\i k v&=e^{\i k x_1}\int_0^\delta t\psi(t)e^{\i t x_1-c\sqrt{t^2+2kt}}\d t+e^{\i k x_1}\int_0^\delta t^{3/2}\psi(t)e^{\i t x_1-c\sqrt{t^2+2kt}}\d t\\
&:=e^{\i k x_1}\int_0^\delta \sqrt{t}\zeta(t)e^{\i t x_1-c\sqrt{t^2+2kt}}\d t+e^{\i k x_1}\int_0^\delta t\psi(t)e^{\i t x_1-c\sqrt{t^2+2kt}}\d t
\end{align*}
where $\zeta(t)=t\psi(t)\in C^\infty[0,\delta)$ and vanishes at $\delta$. These two integrals have exactly the same form as in the proof of Theorem \ref{th:decay_v13ii2}. Thus the results come immediately.
\end{proof}

\section{Proof for Part (B) of the Main Theorem}
\label{sec:proof_main}

Now we conclude the proof for Part (B) of the main theorem, i.e., Theorem \ref{th:SRC}. The results are concluded from Section \ref{sec:rc_0}, \ref{sec:rc_1} and \ref{sec:rc_2}. According to Theorem \ref{th:reg_sca}, the structures of the solutions are different depending on the wavenumber $k$, thus we will prove the two cases separately.

First we prove the result for {\bf Case I, when $2k\not\in\N$}. In this case, $\S=\{-\kappa,\kappa\}$ where $k=\kappa+j$ for a unique $\kappa\in(-1/2,1/2)\setminus\{0\}$ and $j\in\Z$,  thus $|\S|=2$. For simplicity, denote 
\[
\S:=\{\alpha_1,\alpha_2\},
\]
and $J_\ell^1:=J_\ell(\alpha_1)$ and $J_\ell^1:=J_\ell(\alpha_1)$ for $\ell=-,0,+$.

\begin{proof}[Proof for Case I]
We recall that from the inverse Floquet-Bloch transform and the regularity of $w$ given by \eqref{eq:sing_c1}, $u$ is given by \eqref{eq:u_sing_c1}
\[
u(x)=\int_\Lambda w_0(\alpha,x)\d\alpha+\int_{\alpha_1-\delta}^{\alpha_1+\delta}\sqrt{\alpha-\alpha_1} w_1(\alpha,x)\d\alpha+\int_{\alpha_2-\delta}^{\alpha_2+\delta}\sqrt{\alpha-\alpha_2} w_2(\alpha,x)\d\alpha:=u_0+u_1+u_2,
\]
where $u_0$ has the form of \eqref{eq:def_u0}, $u_1$ and $u_2$ are defined in the way of \eqref{eq:def_u1}. Then for $x_2\geq H+h$,
\begin{align*}
u(x)&=\int_{\Gamma^H}\frac{\partial\Phi(x,y)}{\partial y_2}u_0(y)\d s(y)+\int_{\Gamma^H}\frac{\partial\Phi(x,y)}{\partial y_2}u_1(y)\d s(y)+\int_{\Gamma^H}\frac{\partial\Phi(x,y)}{\partial y_2}u_2(y)\d s(y)\\
&:=v_0+v_1+v_2.
\end{align*}
Thus we only need to study the radiation conditions for $v_0$, $v_1$ and $v_2$. 

The radiation condtion for $v_0$ is already very clear, i.e., when $r\rightarrow \infty$,
\[
|v|\leq Cr^{-1/2},\quad \left|\frac{\partial v}{\partial r}-\i k v\right|\leq Cr^{-3/2}
\]
holds uniformly for $x\in U_{H+h}$.

 Thus we only need to consider that for $v_1$ (since $v_2$ is similar). With \eqref{eq:decom_u},
\begin{align*}
v_1(x)&=\sum_{j\in\Z}\int_{\alpha_1-\delta}^{\alpha_1+\delta}\sqrt{\alpha-\alpha_1}\,g_j(\alpha)e^{\i(\alpha+j)x_1+\i\sqrt{k^2-(\alpha+j)^2}(x_2-H)}\d\alpha\\
&:=\sum_{j\in\Z}v_1^{(j)}(x)\\
&=\sum_{j\in\in J_-^1}+\sum_{j\in J_0^1}+\sum_{j\in J_+^1}\int_{\alpha_1-\delta}^{\alpha_1+\delta}\sqrt{\alpha-\alpha_1}\,g_j(\alpha)e^{\i(\alpha+j)x_1+\i\sqrt{k^2-(\alpha+j)^2}(x_2-H)}\d\alpha,
\end{align*}
where $g_j$ is defined by
\[
g_j(\alpha)=\frac{1}{2\pi}\int_{\alpha_1-\delta}^{\alpha_1+\delta}\left[w(\alpha,x)\big|_{\Gamma^H_0}\right]e^{-\i(\alpha+j)x_1}\d x_1,
\]
thus it is easily checked that $\|g_j\|_{C^m[\alpha_1-\delta,\alpha_1+\delta]}$ is uniformly bounded for $j\in\Z$.

When $j\in J_-^1$, with Theorem \ref{th:v_11_decay}, \ref{th:rc_v111} and \ref{th:rc_v112}, there is a constant $C>0$ such that
\[
\left|v_1^{(j)}\right|\leq Cr^{-1/2},\quad \left|\frac{\partial v_1^{(j)}}{\partial r}-\i k v_1^{(j)}\right|\leq Cr^{-3/2}
\]
holds uniformly for $x\in U_{H+h}$.

When $j\in J_0^1$, with Theorem \ref{th:decay_v_13i} and \ref{th:rc_v13i}, there is a constant $C>0$ such that
\[
\left|v_1^{(j)}\right|\leq Cr^{-1/2},\quad \left|\frac{\partial v_1^{(j)}}{\partial r}-\i k v_1^{(j)}\right|\leq Cr^{-3/2}
\]
holds uniformly for $x\in U_{H+h}$.

When $j\in J_+^1$, with Theorem \ref{th:decay_rc_v121} and \ref{th:decay_rc_v122},  there are two constants $C,\epsilon_0>0$ such that
\[
\left|v_1^{(j)}\right|, \left|\frac{\partial v_1^{(j)}}{\partial r}-\i k v_1^{(j)}\right|\leq Ce^{-\epsilon_0|j| h}r^{-3/2}
\]
holds uniformly for $x\in U_{H+h}$.

Now we combine the above results for the three cases. Now we need to estimate 
\[
|v_1(x)|\leq\sum_{j\in\Z}\left|v_1^{(j)}\right|\leq C\sum_{j\in J_-^1}r^{-1/2}+C\sum_{j\in J_0^1}r^{-1/2}+C\sum_{j\in J_+^1}e^{-\epsilon_0|j| h}r^{-3/2}\leq Cr^{-1/2}
\]
since $|J^1_-|,\,|J^1_0|<\infty$ and the series
\[
\sum_{j\in J_+^1}e^{-\epsilon_0|j| h}<\infty.
\]
Similarly, we can also estimate that
\[
\left|\frac{\partial v_1}{\partial r}-\i k v_1\right|\leq Cr^{-3/2}.
\]
The estimation also holds for $v_2$. Then the proof is finished.
\end{proof}

Now we move on to prove the result for {\bf Case II, when $2k\in\N$}. In this case, $\S=\{0\}$ when $k\in\N$, and $\S=\{1/2\}$ when $k-1/2\in\N$. Thus $|\S|=1$. For simplicity, denote  $\S=\{\alpha_0\}$, and $J_\ell=J_\ell(\alpha_0)$ for $\ell=-,0,+$.

\begin{proof}[Proof for Case II]
We recall that from the inverse Floquet-Bloch transform and the regularity of $w$ given by \eqref{eq:sing_c2}, $u$ is given by \eqref{eq:u_sing_c2}
\[
u(x)=\int_\Lambda w_0(\alpha,x)\d\alpha+\int_{\alpha_1-\delta}^{\alpha_1+\delta}\sqrt{\alpha-\alpha_0} w_1(\alpha,x)\d\alpha+\int_{\alpha_2-\delta}^{\alpha_2+\delta}|\alpha-\alpha_0|w_2(\alpha,x)\d\alpha:=u_0+u_1+u_2,
\]
where $u_0$ has the form of \eqref{eq:def_u0}, $u_1$ and $u_2$ are defined in the way of \eqref{eq:def_u1}. Then for $x_2\geq H+h$,
\begin{align*}
u(x)&=\int_{\Gamma^H}\frac{\partial\Phi(x,y)}{\partial y_2}u_0(y)\d s(y)+\int_{\Gamma^H}\frac{\partial\Phi(x,y)}{\partial y_2}u_1(y)\d s(y)+\int_{\Gamma^H}\frac{\partial\Phi(x,y)}{\partial y_2}u_2(y)\d s(y)\\
&:=v_0+v_1+v_2.
\end{align*}
Thus we only need to study the radiation conditions for $v_0$, $v_1$ and $v_2$. 

From the proof for Case I, we can easily find a constant $C>0$ such that
\[
|v_0|,\,|v_1|\leq Cr^{-1/2},\quad\left|\frac{\partial v_0}{\partial r}-\i k v_0\right|\,\left|\frac{\partial v_1}{\partial r}-\i k v_1\right|\leq Cr^{-3/2}.
\]
Thus we only need to focus on $v_2$.

Again with \eqref{eq:decom_u},
\begin{align*}
v_2(x)&=\sum_{j\in\Z}\int_{\alpha_0-\delta}^{\alpha_0+\delta}|\alpha-\alpha_0|\,g_j(\alpha)e^{\i(\alpha+j)x_1+\i\sqrt{k^2-(\alpha+j)^2}(x_2-H)}\d\alpha\\
&:=\sum_{j\in\Z}v_2^{(j)}(x)\\
&=\sum_{j\in\in J_-^1}+\sum_{j\in J_0^1}+\sum_{j\in J_+^1}\int_{\alpha_0-\delta}^{\alpha_0+\delta}|\alpha-\alpha_0|\,g_j(\alpha)e^{\i(\alpha+j)x_1+\i\sqrt{k^2-(\alpha+j)^2}(x_2-H)}\d\alpha,
\end{align*}
where $\|g_j\|_{C^m(\alpha_0-\delta,\alpha_0+\delta)}$ are uniformly bounded with respect to $j$.

When $j\in J_-$, with Theorem \ref{th:v1_2_decay} and \ref{th:v1_2_rc}, there is a constant $C>0$ such that
\[
\left|v_2^{(j)}\right|\leq Cr^{-1/2},\quad \left|\frac{\partial v_2^{(j)}}{\partial r}-\i k v_2^{(j)}\right|\leq Cr^{-3/2}
\]
holds uniformly for $x\in U_{H+h}$.

When $j\in J_0$, with Corollary \ref{cr:decay_v_13i_2}, \ref{cr:decay_v13ii1_2} and \ref{cr:decay_v13ii2_2}, there is a constant $C>0$ such that
\[
\left|v_2^{(j)}\right|\leq Cr^{-1/2},\quad \left|\frac{\partial v_2^{(j)}}{\partial r}-\i k v_2^{(j)}\right|\leq Cr^{-3/2}
\]
holds uniformly for $x\in U_{H+h}$.

When $j\in J_+^1$, with Theorem \ref{th:decay_rc_v221} and \ref{th:decay_rc_v222},  there are two constants $C,\epsilon_0>0$ such that
\[
\left|v_2^{(j)}\right|, \left|\frac{\partial v_2^{(j)}}{\partial r}-\i k v_2^{(j)}\right|\leq Ce^{-\epsilon_0|j| h}r^{-2}
\]
holds uniformly for $x\in U_{H+h}$.

Similar to the proof for Case I, we can easily arrive at the final result for Case II. The proof is now finished.

\end{proof}

\section{Radiation condition for scattering problems with locally perturbed periodic surfaces}
\label{sec:loc_sur}

In this section, we extend the results to the case when the periodic surface is perturbed. Let $p:\,\R\rightarrow\R$ be a function which is compactly supported, then there is an $L>0$ such that $p(x_1)=0$ when $|x_1|\geq L$. Let
\[
\Gamma_p:=\{(x_1,\zeta(x_1)+p(x_1)):\,x_1\in\R\}\subset\R^2
\]
be the locally perturbed surface of $\Gamma$. Similarly we still define $\Omega_p$ be the domain above $\Gamma_p$. In this case, $\Omega_p$ is no longer periodic. In this section, we assume that both $\zeta$ and $p$ are bounded and $C^1$-continuous due to technical reasons.

We still consider the problem 
\begin{equation}
\label{eq:sca_p}
\Delta u_p+k^2 u_p=0\text{ in }\Omega_p;\quad u_p=f=-G(x,y)\text{ on }\Gamma_p
\end{equation}
with the UPRC \eqref{eq:uprc}, i.e.,
\[
u_p(x)=2\int_{\Gamma^H}\frac{\partial\Phi(x,y)}{\partial y_2}u_p(y)\d s(y),\quad x_2\geq H.
\]
 From Theorem \ref{th:wep}, the problem is still uniquely solvable in the weighted Sobolev space $\Omega_p^H:=\Omega_p\cap\R\times[0,H_0)$. 

Since $\Omega_p$ is not periodic, the Floquet-Bloch transform can not be applied directly to $u_p$. Following \cite{Lechl2016,Lechl2017}, we can define a diffeomorphis $\Phi_p:\,\Omega\rightarrow\Omega_p$, for example in the form of 
\[
\Phi_p(x)=\begin{cases}
\left(x_1,x_2+\frac{(x_2-H_0)^3}{(\zeta(x_1)-H_0)^3}p(x)\right),\quad x_2<H_0;\\
x,\quad x_2\geq H_0;
\end{cases}
\]
where $\max_{t\in\R}\{\zeta(t),\zeta(t)+p(t)\}<H_0<H$. 
Let $u_T:=u_p\circ\Phi_p$, then $u_T$ is a function defined in the periodic domain $\Omega$ and 
\[
u_T=u_p,\quad |x_1|\geq L,\text{ or }x_2\geq H_0.
\]
Thus the radiation condition for $u_T$ is exactly the same as $u_T$. 
Moreover, $u_T$ satisfies the problem
\begin{equation}
\label{eq:sca_T}
\nabla\cdot A_p\nabla u_T+k^2 c_p u_T=0\text{ in }\Omega;\quad u_T=f(\Phi_p(x))\quad\text{ on }\Gamma
\end{equation}
with the same UPRC. Here $A_p$ and $c_p$ are defined by $\Phi_p$:
\begin{align*}
A_p(x)&=\left|\det\nabla\Phi_p(x)\right|\left[(\nabla\Phi_p(x))^{-1}(\nabla\Phi_p(x))^{-T}\right]\in L^\infty(\Omega_{H_0},\R^2\times\R^2);\\
c_p(x)&=\left|\det\nabla\Phi_p(x)\right|\in L^\infty(\Omega_{H_0}).
\end{align*}
From the definition of $\Phi_p$, $A_p-I_2$ and $c_p-1$ are both compactly supported in the domain $\Omega\cap[-L,L]\times[0,H_0]$. 

Now we are prepared to apply the Floquet-Bloch transform to the transformed problem \eqref{eq:sca_T}. Define $w(\alpha,x):=\J u_T$, following Section 3.2 in \cite{Zhang2017e}, we have exactly the same regularity as the purely perturbed case, i.e., either \eqref{eq:sing_c1} or \eqref{eq:sing_c2} holds for the transformed field $w$. It implies that all the proofs for the purely periodic case can be adopted directly for the case with local perturbations. Thus Theorem \ref{th:SRC} also holds for $u_T$ (equivalently, $u_p$).\\

We can also consider the scattering problems with incident plane waves. Let $u^i$ be a plane wave, then consider the problem
\begin{equation}
\label{eq:sca_pl}
\Delta u_p+k^2 u_p=0\text{ in }\Omega_p;\quad u_p=f:=-u^i\text{ on }\Gamma_p,
\end{equation}
with the UPRC. To analyze this problem, we also need a reference problem defined in a purely periodic domain:
\begin{equation}
\label{eq:sca_plr}
\Delta u_r+k^2 u_r=0\text{ in }\Omega;\quad u_r=-u^i\text{ on }\Gamma,
\end{equation}
with the UPRC. For \eqref{eq:sca_pl}, we again apply the domain transformation $\Phi_p$ and let $u_T:=u_p\circ\Phi_p$, then $u_T$ is defined in $\Omega$ and satisfies
\begin{equation}
\label{eq:sca_plT}
\nabla\cdot A_p\nabla u_T+k^2 c_p u_T=0\text{ in }\Omega;\quad u_T=f(\Phi_p(x))\text{ on }\Gamma.
\end{equation}
Let $u_d:=u_p-u_i$, then it satisfies
\begin{equation}
\label{eq:sca_diff}
\Delta u_d+k^2 u_d=g\text{ in }\Omega;\quad u_d=f-f\circ\Phi_p\text{ on }\Gamma,
\end{equation}
where
\[
g=\nabla\cdot(I_2-A_p)\nabla u_T+k^2(1-c_p)u_T
\]
is compactly supported, and $u_d\big|_{\Gamma}$ is also compactly supported. Following the arguments in previous section, it is easily checked that $u_d$ satisfies the radiation condition defined in Theorem \ref{th:SRC}. Thus the solution $u_p$ is decomposed as
\[
u_p=u_r+u_d,
\]
where $u_r$ satisfies the Rayleigh expansion \eqref{eq:rayleigh} and $u_d$ satisfies the radiation condition defined in Definition \ref{def:rc}. The result is now summarized as the following theorem.

\begin{theorem}
Let $\Gamma_p$ be a local perturbation of the periodic surface $\Gamma$. Given the incident field $u^i$, which is either the point source, i.e., $u^i=G(x,y)$, or the plane wave $u^i=e^{\i k (x_1\cos\theta-x_2\sin\theta)}$ with $\theta\in(0,\pi)$. Then
\begin{enumerate}[label=(\roman*)]
\item when $u^i$ is a point source, the scattered field $u_p$ satisfies the radiation condition defined by Definition \ref{def:rc};
\item when $u^i$ is a plane wave, the scattered  field $u_p$ is decomposed as
\[
u_p=u_r+u_d,
\]
where $u_r$ satisfies the Rayleigh expansion \eqref{eq:rayleigh} with $\alpha= k\cos\theta$ and $u_d$ satisfies the radiation condition defined by Definition \ref{def:rc}.
\end{enumerate}
\end{theorem}

\appendix

\section{Appendix A: Useful formulas and estimations}

\subsection{Fresnel integral and its application}
\label{sec:fresnel}
Recall that the Fresnel integrals are defined by (see \cite{Abram1972}):
\begin{equation}
\label{eq:fresnel}
S(t)=\int_0^t \sin(z^2)\d z,\quad C(t)=\int_0^t \cos(z^2)\d z
\end{equation}
with asymptotic behaviours
\begin{align*}
& S(t)=\sqrt{\frac{\pi}{8}}{\rm sgn}(t)-\left[1+O\left(t^{-4}\right)\right]\left(\frac{\cos(t^2)}{2t}+\frac{\sin(t^2)}{4t^3}\right),\\
& C(t)=\sqrt{\frac{\pi}{8}}{\rm sgn}(t)-\left[1+O\left(t^{-4}\right)\right]\left(\frac{\sin(t^2)}{2t}-\frac{\cos(t^2)}{4t^3}\right),
\end{align*}
where sgn is the sign function which takes the  value of $1$ for positive $t$, $-1$ for negative $t$ and $0$ when $t=0$. Thus we claim that the integral
\[
\int_0^t e^{\i x^2}\d x=\left(C(t)+\i S(t)\right) 
\]
is uniformly bounded for $t\geq 0$ since $S(t)$ and $C(t)$ are uniformly bounded.

We can also extend the Fresnel integral into more general cases (see \cite{Matha2012}): 
\[
\int_0^\infty x^m e^{\i x^n}\d x=\frac{1}{n}\Gamma\left(\frac{m+1}{n}\right)e^{\i\pi\frac{m+1}{2n}}
\]
where $\Gamma$ is the $\Gamma$-function. In particular,
\begin{equation}
\label{eq:gen_fresnel}
\int_0^\infty\sqrt{x}e^{\i x^2}\d x=\frac{1}{2}\Gamma\left(\frac{3}{4}\right)\exp\left(\frac{3\i\pi}{8}\right).
\end{equation}

%We apply the asymptotic properties of the Fresnel integrals to estimate the following integrals:
%\[
%\int_0^t e^{\i s^2 x_1}\d s,\quad \int_0^t e^{-\i s^2 x_1}\d s
%\]
%when $|x_1|\rightarrow\infty$.

%Suppose $x_1>0$, let $z=\sqrt{x_1}s$ then
%\begin{align*}
%\int_0^t e^{\i s^2 x_1}\d s&=\frac{1}{\sqrt{x_1}}\int_0^{t\sqrt{x_1}} e^{\i z^2}\d z\\&=\frac{1}{\sqrt{x_1}}\left(\int_0^{t\sqrt{x_1}}\cos(z^2)\d z+\i\int_0^{t\sqrt{x_1}}\sin(z^2)\d z\right)\\
%&=\frac{1}{\sqrt{x_1}}\left[C(t\sqrt{x_1})+\i S(t\sqrt{x_1})\right]\\&=O\left(\frac{1}{\sqrt{x_1}}\right).
%\end{align*}
%With similar arguments, the estimate holds for $x_1<0$ and 
%\[
%\int_0^t e^{\i s^2 x_1}\d s,\quad \int_0^t e^{-\i s^2 x_1}\d s=O\left(\frac{1}%{\sqrt{|x_1|}}\right).
%\]

\subsection{Properties of the Bessel functions}

In this section, we recall some important properties of the Bessel functions. First, the derivative:
\begin{equation}
\label{eq:der_bessel}
\left(\frac{1}{z}\frac{\d}{\d z}\right)^m[z^n J_n(z)]=z^{n-m}J_{n-m}(z).
\end{equation}
In particular, when $m=1$, 
\[
\frac{\d}{\d z}[z^n J_n(z)]=z^n J_{n-1}(z).
\]
Let $z:=rx$ where $r>0$, then
\[
\frac{\d}{\d z}=\frac{1}{r}\frac{\d}{\d x}. 
\]
This implies that
\begin{equation}
\label{eq:der_bessel_1}
\frac{\d}{\d x}[x^n J_n(rx)]= r x^n J_{n-1}(rx).
\end{equation}

The property also holds for Hankel functions of the first kind. The following equations will also be useful:
\begin{eqnarray}
\label{eq:hankel_diff1}
&&\frac{\d}{\d t}H_0^{(1)}(t)=-H_1^{(1)}(t);\\
\label{eq:hankel_diff2}
&&\frac{\d}{\d t}H_1^{(1)}(t)=H_0^{(1)}(t)-\frac{1}{t}H_1^{(1)}(t).
\end{eqnarray}

We also need the asymptotic behaviour
\begin{equation}
\label{eq:asym}
H_n^{(1)}(t)=\sqrt{\frac{2}{\pi t}}\exp\left(\i\left(t-\frac{n\pi}{2}-\frac{\pi}{4}\right)\right)\left[1+O\left(\frac{1}{t}\right)\right],\quad t\rightarrow\infty.
\end{equation}

\section{Appendix B: Analysis of some important integrals}
\label{sec:fourier}

To prepare for the estimation of the radiation condition in Theorem \ref{th:SRC} with the help of the Floquet-Bloch transform, we  need to estimate the decay of some integrals depends on parameters, which are concluded in Table \ref{tb:integral}. 

\begin{table}[H]
\centering
\label{tb:integral}
\begin{tabular}{cccc}
\hline
class & name & definition & decay rate\\
\hline
\hline
A & $I^A_1(r)$ & $\int_0^\delta\frac{1}{\sqrt{x}}\phi(x)e^{\i r x}\d x$ & $O\left(r^{-1/2}\right)$\\
& $I^A_2(r)$ & $\int_0^\delta x^{-1/2+m}\phi(x)e^{\i r x}\d x$ & $O\left(r^{-1/2-m}\right)$\\
& $I^A_3(r)$ & $\int_0^\delta x^m\phi(x)e^{\i r x}\d x$ & $O\left(r^{-1-m}\right)$\\
\hline
B  & $I^B_1(r)$ & $\int_0^\delta \sqrt{x}\phi(x)e^{\i r x^2}\d x$ & $O\left(r^{-3/4}\right)$\\
& $I^B_2(r)$ & $\int_{-a}^\delta(x+a)^2\phi(x)e^{\i r x^2}\d x$ & $O\left(r^{-1/2}\right)$\\
\hline
C & $I^C_1(r)$ & $\int_{\gamma_1}^{\gamma_2}\frac{1}{\sqrt{x-a}}\phi(x)e^{\i r x^2}\d x$ & $a^{-1}O\left(r^{-1/2}\right)$\\
& $I^C_2(r)$ & $\int_{\gamma_1}^{\gamma_2}{\sqrt{x-a}}\phi(x)e^{\i r x^2}\d x$ & $O\left(r^{-1/2}\right)$\\
\hline
D & $I^D_1(r)$ & $\int_{\gamma_1}^{\gamma_2}|x-a|\phi(x)e^{\i r x^2}\d x$ & $O\left(r^{-1/2}\right)$\\
& $I^D_2(r)$ & $\int_{\gamma_1}^{\gamma_2}{\rm sgn}(x-a) \phi(x)e^{\i  r x^2}\d x$ & $O\left(r^{-1/2}\right)$\\
\hline
\end{tabular}
\end{table}
We will consider the decay rate of the above integrals when $r\rightarrow\infty$. Here $\delta>0$ is a fixed small value. For simplicity, we all require that $\phi\in C^\infty_0(\R)$ and vanishes either near $\delta$ (for example in  $I^A_1(r),\, I^A_2(r),\,I^B_1(r),\,I^B_2(r)$) or the edge points $\gamma_1,\gamma_2$ (for example in $I^C_1(r),\,I^C_2(r),\,I^D_1(r),\,I^D_2(r)$. Note that the results can be naturally extended to  more general functions, i.e., the condition $\phi\in C^\infty_0(\R)$ is always too strong (for example see Remark \ref{rm:smooth}). We will not discuss the very explicit explanations since it will be too long for the readers. For the discussions of the integrals of Classes A, B, C and D, we refer to Subsections \ref{sec:intg}, \ref{sec:intg2}, \ref{sec:intg3} and \ref{sec:intg4}, respectively.

\subsection{Analyse of the integrals of Class A}
\label{sec:intg}

We begin with {\bf the first integral of Class A}:
\begin{equation}
\label{eq:inta1}
I_1^A(r):=\int_0^\delta \frac{1}{\sqrt{x}}\phi(x)e^{\i r x}\d x,\quad r>0.
\end{equation}
where $\phi\in C^\infty_0(\R)$ which vanishes near $x=\delta$. The decay rate is concluded in the following theorem.

\begin{theorem}\label{th:decay_inta1}
For any $\phi\in C^\infty_0(\R)$ which vanishes in a small neighbourhood of $x=\delta$, there is a constant $C>0$ which only depends on $\phi$ such that
\begin{equation}
\label{eq:decay_inta1}
|I^A_1(r)|\leq Cr^{-1/2},\quad \forall\, r>0.
\end{equation}
\end{theorem}

\begin{proof}
From direct computation,
\[
I_1^A(r)=\phi(0)\int_0^\delta\frac{1}{\sqrt{x}}e^{\i r x}\d x+\int_0^\delta {\sqrt{x}}\psi(x)e^{\i r x}\d x:=\phi(0)(I)+(II),
\]
where
\[
\psi(x)=\frac{\phi(x)-\phi(0)}{x}\quad\text{ with }\psi(0)=\phi'(0).
\]
Then $\psi\in C^\infty(\R)$ with $\psi'(0)=\phi''(0)/2$. 

For (I), let $s:=\sqrt{rx}$, thus $2s\d s=r\d x$, which results in
\[
(I)=\frac{2}{\sqrt{r}}\int_0^{\sqrt{r\delta}}e^{\i s^2}\d s=\frac{2}{\sqrt{r}}\left[C\left(\sqrt{r\delta}\right)+\i S\left(\sqrt{r\delta}\right)\right],
\]
where $S(x)$ and $C(t)$ are Fresnel integrals (see equation \ref{eq:fresnel}) and uniformly bounded. Thus
\[
\left|(I)\right|\leq C\frac{1}{\sqrt{r}}|\phi(0)|,\quad r\rightarrow\infty.
\]
For (II), with integration by parts,
\begin{align*}
(II)=\frac{1}{\i r}\int_0^\delta \sqrt{x}\psi(x)\d e^{\i r x}=\frac{1}{\i r }\sqrt{\delta}\psi(\delta)e^{\i r \delta}+\frac{1}{r}\int_0^\delta \frac{1}{\sqrt{x}}\tilde{\psi}(x) e^{\i r x}\d x
\end{align*}
where $$\tilde{\psi}(x)=\i\left(\frac{\psi(x)}{2}+x\psi'(x)\right)\in C^\infty(\R).$$ Since $\frac{1}{\sqrt{x}}$ is integrable in $[0,\delta]$
\[
\left|(II)\right|\leq \frac{C}{r}.
\]
Together with (I), we get that
\[
|I_1^A(r)|\leq C\frac{1}{\sqrt{r}}.
\]
\end{proof}
\begin{remark}
\label{rm:smooth}
From the proof of Theorem \ref{th:decay_inta1}, the positive constant $C$ only depends on the norm $\|\phi\|_{C^2[0,\delta]}$ thus the result is extended naturally for all $\phi\in C^2_0[0,\delta]$. However, for the rest of estimations in this section, we will not give very explicit discussions to save the space.  
\end{remark}

\vspace{0.2cm}

With the analysis of the first integral of Class A, we are prepared to analyze {\bf the second integral of Class A}
\begin{equation}
\label{eq:inta2}
I_2^A(r):=\int_0^\delta x^{-1/2+m}\phi(x)e^{\i r x}\d x,\quad m\in\N.
\end{equation}
The estimation will be concluded as the corollary of Theorem \ref{th:decay_inta1}.

\begin{corollary}
\label{cr:decay_inta2}
For any $\phi\in C^\infty_0(\R)$ which vanishes in a small neighbourhood of $x=\delta$ and any non-negative integer $m$, there is a constant $C>0$ which only depends on $\phi$ and $m$ such that
\begin{equation}
\label{eq:decay_inta2}
\left|I^A_2(r)\right|\leq Cr^{-1/2-m},\quad \forall\, r>0.
\end{equation}
\end{corollary}

\begin{proof}
We prove by induction. For $m=0$, the result has already been proved in Theorem \ref{th:decay_inta1}. Then assume that for a non-negative integer $m$, the estimation \eqref{eq:decay_inta2} holds. Then for $m+1$, from integration by parts, since $\phi(\delta)=0$,
\begin{align*}
\int_0^\delta {x}^{-1/2+m+1}\phi(x) e^{\i r x} \d x&=\frac{1}{\i r}\int_0^\delta {x}^{1/2+m}\psi(x)\d e^{\i r x}\\&=-\frac{1}{\i r}\int_0^\delta\left(x^{1/2+m}\phi(x)\right)'e^{\i r x}\d x:=-\frac{1}{\i r}\int_0^\delta x^{-1/2+m}\psi(x)e^{\i r x}\d x
\end{align*}
where 
\[
\psi(x):=(-1/2+m+1)\phi(x)+x\phi'(x)\in C^\infty_0(\R)
\]
which vanishes near $x=\delta$. With \eqref{eq:decay_inta2}, there is a constant $C>0$ which only depends on $\psi$ (equivalently $\phi$) such that
\[
\left|\int_0^\delta {x}^{1/2+m}\psi(x)e^{\i r x}\d x\right|\leq Cr^{-1}r^{-1/2-m}=Cr^{-1/2-m-1}.
\]
The proof is then finished.
\end{proof}

\vspace{0.2cm}

We are also able to estimate the asymptotic behaviours of {\bf the third integral of Class A:}
\begin{equation}
\label{eq:inta3}
I^A_3(r):=\int_0^\delta x^m\phi(x)e^{\i r x}\d x,\quad r>0,
\end{equation}
where $m=0,1,2,\dots$. 

\begin{theorem}
\label{th:decay_inta3}
Let $\phi\in C_0^\infty(\R)$ and vanishes near $\delta$, $m$ is a non-negative integer. Then there is a constant that depends only on $\phi$ and $m$ such that
\begin{equation}
\label{eq:decay_inta3}
\left|I^A_3(r)\right|\leq Cr^{-1-m}.
\end{equation}
\end{theorem}

\begin{proof}
We prove by induction. First when $m=0$, from integration by parts,
\begin{align*}
\int_0^\delta \phi(x)e^{\i  rx}\d x&=\frac{1}{\i r}\int_0^\delta \phi(x)\d e^{\i r x}\\
&=-\frac{1}{\i r}\phi(0)-\frac{1}{\i r}\int_0^\delta\phi'(x)e^{\i r x}\d x.
\end{align*} 
Then we can immediately get that there is a $C>0$ such that
\[
\left|\int_0^\delta \phi(x)e^{\i  rx}\d x\right|\leq Cr^{-1}.
\]
Now assume that the estimation \eqref{eq:decay_inta3} holds for any $m\geq 1$, then for $m+1$,
\begin{align*}
\int_0^\delta x^{m+1}\phi(x)e^{\i  rx}\d x&=\frac{1}{\i r}\int_0^\delta x^{m+1}\phi(x)\d e^{\i r x}\\
&=-\frac{1}{\i r}\int_0^\delta\left[x^{m+1}\phi(x)\right]e^{\i r x}\d x:=-\frac{1}{\i r}\int_0^\delta x^m\psi(x) e^{\i r x}\d x,
\end{align*}
where
\[
\psi(x)=(m+1)\phi(x)+x\phi'(x)\in C_0^\infty(\R),
\]
which also vanishes around $\delta$. With \eqref{eq:decay_inta3},
\[
\left|\int_0^\delta x^{m+1}\phi(x)e^{\i  rx}\d x\right|\leq\frac{1}{r}\left|\int_0^\delta x^m\psi(x) e^{\i r x}\d x\right|\leq Cr^{-1-m}\frac{1}{r}=Cr^{-2-m}.
\]
Thus \eqref{eq:decay_inta3} also holds for $m+1$. The proof is finished.
\end{proof}

\subsection{Analyse of the integrals of Class B}
\label{sec:intg2}

{\bf The first integral of Classe B} is defined as:
\begin{equation}
\label{eq:intb1}
I_1^B(r):=\int_0^\delta \sqrt{x}\phi(x)e^{\i r x^2}\d x,\quad r>0,
\end{equation}
where the function $\phi$ has exactly the same property as in Section \ref{sec:intg}.

\begin{theorem}
\label{th:intb1}
When $\phi\in C_0^\infty(\R)$ and vanishes near $\delta$, then there is a constant $C>0$ only depends on $\phi$ such that
\begin{equation}
\label{eq:decay_intb1}
\left|I_1^B(r)\right|\leq Cr^{-3/4},\quad \forall\, r>0.
\end{equation}
\end{theorem}

\begin{proof}
Similar to the proof of Theorem \ref{th:decay_inta1}, we can also split the integral into two parts: 
\[
I_1^B(r)=\phi(0)\int_0^\delta\sqrt{x}e^{\i r x^2}\d x+\int_0^\delta x^{3/2}\psi(x)e^{\i r x^2}\d x:=\phi(0)(I)+(II),
\]
where
\[
\psi(x)=\frac{\phi(x)-\phi(0)}{x}\in C^\infty(\R).
\]
For (I), let $s:=\sqrt{r}x$, from \eqref{eq:gen_fresnel}
\[
|(I)|=r^{-3/4}\left|\int_0^{\sqrt{r}\delta}\sqrt{s}e^{\i s^2}\d s\right|\leq Cr^{-3/4}.
\]
For (II), with integration by parts,
\begin{align*}
|(II)|=\left|\frac{1}{2\i r}\int_0^\delta \sqrt{x}\psi(x)\d e^{\i r x^2}\right|\leq\frac{1}{2 r}\sqrt{\delta}|\psi(\delta)|+\frac{1}{2 r}\left|\int_0^\delta \frac{1}{\sqrt{x}}\tilde{\psi}(x)e^{\i r x^2}\d x\right|\leq Cr^{-1},
\end{align*}
where $C>0$ depends on $\psi$ and $\tilde{\psi}$, which is defined by
\[
\tilde{\psi}(x)=\frac{1}{2}\psi(x)+x\psi'(x)\in C^\infty(\R).
\]
So finally we can show that
\[
|I^B_1(r)|\leq Cr^{-3/4},\quad\forall\,r>0.
\]
\end{proof}

\vspace{0.2cm}

Now we consider {\bf the second integral of Class B}:
\begin{equation}
\label{eq:intb2}
I_2^B(r)=\int_{-a}^\delta (x+a)^2\phi(x)e^{\i r x^2}\d x.
\end{equation}

\begin{theorem}
\label{th:decay_intb2}
Suppose $\phi\in C^\infty_0(\R)$ vanishes near $\delta$, and fix  a number $a>0$. Then there is a constant $C>0$ which only depends on $\phi$ such that
\begin{equation}
\label{eq:decay_intb2}
\left|I^B_2(r)\right|\leq Cr^{-1/2}.
\end{equation}
\end{theorem}

\begin{proof}Since $0\in\left[-a,r^{-1/2}\right]$, we can not directly carry out integral by parts since 
\[
e^{\i r x^2}\d x=\frac{1}{2\i r x}\d e^{\i r x^2}
\]
has a singularity at $0$. To avoid the singularity at $0$, we split the integral \eqref{eq:intb2} into three parts:
\[
I^B_2(r):=\int_{-a}^{-r^{-1/2}}+\int_{-r^{-1/2}}^{r^{-1/2}}+\int_{r^{-1/2}}^\delta (x+a)^2\phi(x)e^{\i r x^2}\d x:=(I)+(II)+(III).
\]
Note that when $a<r^{-1/2}$, then (I) vanishes and the integral interval of (II) is modified as $[-a,r^{-1/2}]$. For simplicity, let $\epsilon_0=\min\{a,r^{-1/2}\}$ then $\epsilon_0\leq r^{-1/2}$. For (II), from direct computation,
\[
|(II)|\leq\int_{-\epsilon_0}^{r^{-1/2}}(x+a)^2|\phi(x)|\d x\leq C\left(\epsilon_0+r^{-1/2}\right)\leq Cr^{-1/2}.
\]
Thus (II) decays at the rate of $r^{-1/2}$ when either $a<r^{-1/2}$ or $a\geq r^{-1/2}$.

Then move on to $(III)$. Since $0\notin[r^{-1/2},\delta]$, with integration by parts,
\begin{align*}
\left|(III)\right|&=\frac{1}{2r}\left|\int_{r^{-1/2}}^\delta \frac{(x+a)^2}{x}\phi(x)\d e^{\i r x^2}\right|\\
&\leq\frac{1}{2r}\left|\frac{(x+a)^2}{x}\phi(x)\Big|^\delta_{r^{-1/2}}\right|+ \frac{1}{r}\left|\int_{r^{-1/2}}^\delta \frac{x+a}{x^2}\psi(x) e^{\i r x^2}\d x\right|\\
&\leq Cr^{-1/2}+Cr^{-1}\int_{r^{-1/2}}^\delta\frac{1}{x^2}\d x\leq Cr^{-1/2}.
\end{align*}
Note that here
\[
\psi(x)=(x-a)\phi(x)+x(x+a)\phi'(x)\in C^\infty(\R).
\]
The analysis for (I) is exactly the same as (III) thus is omitted here.
Based on the analysis of (I), (II) and (III), finally we have
\[
|I^B_2(r)|\leq Cr^{-1/2}.
\]
\end{proof}

\subsection{Analyse of the integrals of Class C}
\label{sec:intg3}

Now we move on to a pair of integrals depends on $r$ and $a$, which is a parameter close to $0$. Without loss of generality, we only prove in this subsection the case that $0\leq a<<1$. In this paper, these two integrals are named {\bf the first and second integrals of Class C}:
\begin{align}
\label{eq:intc1}
&I_1^C(r)=\int_{\gamma_1}^{\gamma_2} \frac{1}{\sqrt{x-a}}\phi(x)e^{\i r x^2}\d x;\\
\label{eq:intc2}
&I_2^C(r)=\int_{\gamma_1}^{\gamma_2} {\sqrt{x-a}}\phi(x)e^{\i r x^2}\ d x.
\end{align}
Here we require that $\gamma_1<0<a<\gamma_2$, and similar as before, $\phi\in C_0^\infty(\gamma_1,\gamma_2)$.

\begin{theorem}
\label{th:intc1}
Suppose $\phi\in C_0^\infty(\gamma_1,\gamma_2)$, where $\gamma_1<a<\gamma_2$ with $|a|<1$. Then there is a constant which only depends on $\phi$ such that
\begin{equation}
\label{eq:decay_intc1}
\left|I_1^C(r)\right|\leq C|a|^{-1} r^{-1/2},
\end{equation}
which is equivalent to 
\begin{equation}
\label{eq:decay_intc1_v2}
\left|a\,I_1^C(r)\right|\leq C r^{-1/2}.
\end{equation}
\end{theorem}

\begin{proof}Without loss of generality, we assume that $0\leq a <1$. Note that similarly to the proof of Theorem \ref{th:decay_intb2}, the point $0$ becomes a singularity when we apply integration by parts to the term $e^{\i r x^2}\d x$. From the definition \eqref{eq:intc1}, $a$ is another singularity. To avoid these singularities,  We  split the integral $I_1^C(r)$ into five or three parts, according to different situations. When $a\geq 2r^{-1/2}$, then $a-1/r> 2r^{-1/2}-r^{-1/2}=r^{-1/2}$. Then 
\begin{align}
I_1^C(r)&=\int_{\gamma_1}^{-r^{-1/2}}+\int_{-r^{-1/2}}^{r^{-1/2}}+\int_{r^{-1/2}}^{a-1/r}+\int_{a-1/r}^{a+1/r}+\int_{a+1/r}^{\gamma_2}\frac{1}{\sqrt{x-a}}\phi(x)e^{\i r x^2}\d x\nonumber\\\label{eq:c11}&:=(I_1)+(II_1)+(I_2)+(II_2)+(I_3).
\end{align}
Note that the integrals $(II_1)$ and $(II_2)$ involve singularities thus are treated particularly. 
When $a<2r^{-1/2}$, then $a+1/r<2r^{-1/2}+r^{-1/2}<4r^{-1/2}$. Let
\begin{equation}
\label{eq:c12}
I_1^C(r)=\int_{\gamma_1}^{-r^{-1/2}}+\int_{-r^{-1/2}}^{4r^{-1/2}}+\int_{4r^{-1/2}}^{\gamma_2}\frac{1}{\sqrt{x-a}}\phi(x)e^{\i r x^2}\d x:=(I_1)+(II_3)+(I_4).
\end{equation}
Similarly, $(II_3)$ involves the singularities. 

In the following proof, we will consider the two cases, i.e., $a\geq 2r^{-1/2}$ and $0\leq a<2r^{-1/2}$, seperately. We start from the first case.\\

\noindent
(i) Case I: $a\geq 2r^{-1/2}$, where $I^C_1(r)$ is separated into 5 parts \eqref{eq:c11}. We need to estimate two terms $(II_1)$  and $(II_2)$ which involve singularities, and the rest of the three terms which don't.

First observe that since $a\geq 2r^{-1/2}$, from direct computation,
\[
\left|(II_1)\right|\leq C\int_{-r^{-1/2}}^{r^{-1/2}}\frac{1}{\sqrt{a-x}}\d x\leq\frac{2Cr^{-1/2}}{\sqrt{a-r^{-1/2}}}=Ca^{-1/2}r^{-1/2},
\]
and
\[
\left|(II_2)\right|\leq C\int_{a-1/r}^{a+1/r}\frac{1}{\sqrt{|x-a|}}\d x\leq Cr^{-1/2}.
\]

Then we move on to the integrals $(I_1)$, $(I_2)$ and $(I_3)$. First define the following intervals:
\[
K_1=[\gamma_1,-r^{-1/2}],\,K_2=[r^{-1/2},a-r^{-1}],\,K_3=[a+r^{-1},\gamma_2].
\]
 Since the integrand in these three integrals are well-defined and we are safe to apply integration by parts, 
\begin{align*}
(I_{\ell})&=\frac{1}{2\i r }\int_{K_\ell}\frac{1}{x\sqrt{x-a}}\phi(x)\d e^{\i r x^2}\\
&=\frac{1}{2\i r}\frac{1}{x\sqrt{x-a}}\phi(x)e^{\i r x}\Big|^+_-+\frac{1}{2\i r}\int_{K_\ell}\left(\frac{\phi(x)}{x^2\sqrt{x-a}}+\frac{\psi(x)}{x(x-a)^{3/2}}\right)e^{\i r x^2}\ d x\\&:=\frac{1}{2\i r}\frac{1}{x\sqrt{x-a}}\phi(x)e^{\i r x}\Big|^+_-+\frac{1}{2\i r}(i_\ell)
\end{align*}
where
\[
\psi(x)=\frac{\phi( x)}{2}-(x-a)\phi'(x)\in C^\infty_0(\gamma_1,\gamma_2)
\]
and
\[
f(x)\Big|^+_-=f(\sigma_+^{(\ell)})-f(\sigma_-^{\ell})\text{ where }K_\ell=[\sigma_-^{(\ell)},\sigma_+^{(\ell)}]\text{ for }\ell=1,2,3.
\]
Since 
\[
\left|(I_\ell)\right|\leq \frac{1}{2r}\left|\left.\frac{1}{x\sqrt{x-a}}\phi(x)\right|_{x=\sigma^{(\ell)}_+}\right|+\frac{1}{2r}\left|\left.\frac{1}{x\sqrt{x-a}}\phi(x)\right|_{x=\sigma^{(\ell)}_-}\right|+\frac{1}{2r}|(i_\ell)|,
\]
we need to estimate the values $\left|\frac{1}{x\sqrt{x-a}}\phi(x)\right|$ at the boundaries of $K_\ell$, $\ell=1,2,3$, as well as the integral $|(i_\ell)|$. Note that since $\phi(\gamma_1)=\phi(\gamma_2)=0$, we only need to consider the points $x=\pm r^{-1/2},\,a\pm r^{-1}$.  Since $a\geq 2r^{-1/2}$, $a\pm r^{-1/2}\geq a-r^{-1/2}\geq a/2$ thus
\[
\left|\left.\frac{1}{x\sqrt{x-a}}\right|_{x= \pm r^{-1/2}}\right|
\leq \frac{\sqrt{r}}{\sqrt{a-r^{-1/2}}}\leq C\sqrt{\frac{r}{a}};
\]
and since $a\pm r^{-1}\geq a-r^{-1}\geq a-a^2/4>a/2$,
\[
\left|\left.\frac{1}{x\sqrt{x-a}}\right|_{x=a\pm r^{-1}}\right|\leq \frac{\sqrt{r}}{a-1/r}\leq C\frac{\sqrt{r}}{a}.
\]
At last we need to estimate the integral $(i_\ell)$. We need to discuss for each interval seperately.
For the integral on $K_1=[\gamma_1,-r^{-1/2}]$, 
\begin{align*}
\left|(i_1)\right|\leq& C\int_{r^{-1/2}}^{-\gamma_1}\frac{1}{x^2\sqrt{x+a}}\d x+C\int_{r^{-1/2}}^{-\gamma_1}\frac{1}{x(x+a)^{3/2}}\d x\\
\leq &C\frac{1}{\sqrt{a}}\int_{r^{-1/2}}^{-\gamma_1}\frac{1}{x^2}\d x+C\sqrt{r}\int_{r^{-1/2}}^{-\gamma_1}\frac{1}{\sqrt{x+a}^3}\d x\\
= &- C\frac{1}{\sqrt{a}}\frac{1}{x}\Big|^{-\gamma_1}_{r^{-1/2}}-C\sqrt{r}\frac{1}{\sqrt{x+a}}\Big|^{-\gamma_1}_{r^{-1/2}}\leq C\sqrt{\frac{r}{a}}.
\end{align*}
For $K_2=[r^{-1/2},a-r^{-1}]$, we need to again slit each integral into two parts:
\begin{align*}
\left|(i_2)\right|\leq&C\int_{r^{-1/2}}^{a-1/r}\frac{1}{x^2\sqrt{a-x}}\d x+C\int_{r^{-1/2}}^{a-1/r}\frac{1}{x(a-x)^{3/2}}\d x\\
=&C\left(\int_{r^{-1/2}}^{a/2}+\int_{a/2}^{a-1/r}\right)\frac{1}{x^2\sqrt{a-x}}\d x+C\left(\int_{r^{-1/2}}^{a/2}+\int_{a/2}^{a-1/r}\right)\frac{1}{x(a-x)^{3/2}}\d x.
\end{align*}
We estimate the above four integrals in the following. Recall that $a\geq 2r^{-1/2}$,
\[
\left|\int_{r^{-1/2}}^{a/2}\frac{1}{x^2\sqrt{a-x}}\d x\right|
\leq C \frac{1}{\sqrt{a/2}}\int_{r^{-1/2}}^{a/2}\frac{1}{x^2}\d x=-C\frac{1}{\sqrt{a/2}}\frac{1}{x}\Big|^{a/2}_{r^{-1/2}}\leq C\sqrt{\frac{r}{a}};
\]
\[
\left|\int_{r^{-1/2}}^{a/2}\frac{1}{x\sqrt{a-x}^3}\d x\right|\leq C\sqrt{r}\int_{r^{-1/2}}^{a/2}\frac{1}{\sqrt{a-x}^3}\d x= C{\sqrt{r}}\frac{1}{\sqrt{a-x}}\Big|^{a/2}_{r^{-1/2}}\leq C\sqrt{\frac{r}{a}};
\]
\[
\left|\int_{a/2}^{a-1/r}\frac{1}{x^2\sqrt{a-x}}\d x\right|\leq C\sqrt{r}\int_{a/2}^{a-1/r}\frac{1}{x^2}\d x=-C\sqrt{r}\frac{1}{x}\Big|^{a-1/r}_{a/2}\leq C\frac{\sqrt{r}}{a};
\]
\[
\left|\int_{a/2}^{a-1/r}\frac{1}{x\sqrt{a-x}^3}\d x\right|\leq C\frac{2}{a}\int_{a/2}^{a-1/r}\frac{1}{\sqrt{a-x}^3}\d x=C\frac{1}{a}\frac{1}{\sqrt{a-x}}\Big|^{a-1/r}_{a/2}\leq C\frac{\sqrt{r}}{a}.
\]
For $K_3=[a+1/r,\gamma_2]$,
\begin{align*}
\left|(i_3)\right|\leq& C\int_{a+1/r}^{\gamma_2}\frac{1}{x^2\sqrt{x-a}}\d x+C\int_{a+1/r}^{\gamma_2}\frac{1}{x(x-a)^{3/2}}\d x\\\leq &C\sqrt{r}\int_{a+1/r}^{\gamma_2}\frac{1}{x^2}\d x+C\frac{1}{a}\int_{a+1/r}^{\gamma_2}\frac{1}{\sqrt{x-a}^3}\d x\\
=&-C\sqrt{r}\frac{1}{x}\Big|^{\gamma_2}_{a+1/r}+C\frac{2}{a}\frac{1}{\sqrt{x-a}}\Big|^{\gamma_2}_{a+1/r}\leq C\frac{\sqrt{r}}{a}.
\end{align*}
Thus altoghter, there is a constant $C>0$ such that
\[
|(i_2)|\leq C\frac{\sqrt{r}}{a}
\]
when $a\geq 2r^{-1/2}$. Thus combine with previous estimations we have
\[
|(I_\ell)|\leq Ca^{-1}r^{-1/2}\quad\Rightarrow\quad \left|I^C_1(r)\right|\leq Ca^{-1}r^{-1/2}.
\]
The first case is then proved.\\

\noindent
(ii) Case II: $0\leq a<2r^{-1/2}$, when $I^C_1(r)$ is given by \eqref{eq:c12}. Similarly, we need to consider the term $(II_3)$ which involves the singularities, and $(I_1)$ and $(I_4)$ which are regular. Notice that for $(I_1)$, it has been discussed in Case I and the estimation still holds for $a<2r^{-1/2}$, thus we don't need to discuss here. Thus we only need to estimate the terms $(II_3)$ and $(I_4)$.

For $(II_3)$, from direct computation,
\begin{align*}
\left|(II_3)\right|&\leq \int_{-r^{-1/2}}^{4r^{-1/2}}\frac{1}{\sqrt{|x-a|}}\d x=\int_{-r^{-1/2}}^a \frac{1}{\sqrt{a-x}}\d x+\int_a^{4r^{-1/2}}\frac{1}{\sqrt{x-a}}\d x\\
&=2\sqrt{a+r^{-1/2}}+2\sqrt{4r^{-1/2}-a}\leq Cr^{-1/4}.
\end{align*}
Note that since $a<2r^{-1/2}$, we estimate that  $r^{1/4}<\sqrt{{2}}a^{-1/2}$ then
\[
|(II_3)|\leq Cr^{-1/2}r^{1/4}\leq Ca^{-1/2}r^{-1/2}.
\]

To estimate $(I_4)$, we define the following interval similary
\[
K_4=[4r^{-1/2},\gamma_2].
\]
With integration by parts, since $\phi(\gamma_2)=0$,
\begin{align*}
(I_4)&=\frac{1}{2\i r }\int_{K_4}\frac{1}{x\sqrt{x-a}}\phi(x)\d e^{\i r x^2}\\
&=\frac{1}{2\i r}\frac{1}{x\sqrt{x-a}}\phi(x)e^{\i r x}\Big|^+_-+\frac{1}{2\i r}\int_{K_4}\left(\frac{\phi(x)}{x^2\sqrt{x-a}}+\frac{\psi(x)}{x(x-a)^{3/2}}\right)e^{\i r x^2}\ d x\\&:=\frac{1}{2\i r}\frac{1}{x\sqrt{x-a}}\phi(x)e^{\i r x}\Big|_{x=4r^{-1/2}}+\frac{1}{2\i r}(i_4)
\end{align*}
From direct computation, since $a<2r^{-1/2}$, $4r^{-1/2}-a>a$ then
\[
\left|\left.\frac{1}{x\sqrt{x-a}}\right|_{x=4r^{-1/2}}\right|=\frac{\sqrt{r}}{4\sqrt{4r^{-1/2}-a}}\leq\frac{1}{4}\sqrt{\frac{r}{a}}.
\]

Now we move on to the integral $(i_4)$. Similar to Case I,
\begin{align*}
\left|(i_4)\right|\leq& C\int_{4r^{-1/2}}^{\gamma_2}\frac{1}{x^2\sqrt{x-a}}\d x+C\int_{4r^{-1/2}}^{\gamma_2}\frac{1}{x(x-a)^{3/2}}\d x\\\leq &C\frac{1}{\sqrt{a}}\int_{4r^{-1/2}}^{\gamma_2}\frac{1}{x^2}\d x+C\sqrt{r}\int_{4r^{-1/2}}^{\gamma_2}\frac{1}{\sqrt{x-a}^3}\d x\\
=&-C\frac{1}{\sqrt{a}}\frac{1}{x}\Big|^{\gamma_2}_{4r^{-1/2}}+C\sqrt{r}\frac{1}{\sqrt{x-a}}\Big|^{\gamma_2}_{4r^{-1/2}}\leq C\sqrt{\frac{r}{a}}.
\end{align*}
Similar to Case I, we can also estimate that
\[
\left|I^C_1(r)\right|\leq C\frac{1}{a\sqrt{r}}
\]
holds uniformly for $r$ and $0<a<1$. Then we can get the estimation \eqref{eq:decay_intc1_v2} immediately.
\end{proof}

As the corollary from Theorem \ref{th:intc1}, we estimate {\bf the second integral of Class C}, i.e., $I^C_2(r)$.

\begin{corollary}
\label{cr:decay_intc2}
Let $\phi\in C_0^\infty(\gamma_1,\gamma_2)$. Then for any $|a|<1$, there is a constant $C>0$ which only depends on $\phi$ such that
\begin{equation}
\label{eq:decay_intc2}
\left|I_2^C(r)\right|\leq Cr^{-1/2}.
\end{equation}
\end{corollary}

\begin{proof}

From direct computation,
\begin{align*}
I_2^C(r)&=\int_{\gamma_1}^{\gamma_2} {\sqrt{x-a}}\phi(x)e^{\i r x^2}\ d x\\&=\int_{\gamma_1}^{\gamma_2} \frac{x}{\sqrt{x-a}}\phi(x)e^{\i r x^2}\ d x-a\int_{\gamma_1}^{\gamma_2} \frac{1}{\sqrt{x-a}}\phi(x)e^{\i r x^2}\ d x:=(I)-a I_1^C(r).
\end{align*}
From \eqref{eq:decay_intc1_v2}, $|a I_1^C(r)|\leq Cr^{-1/2}$, we only need to consider the integral (I). First split the integral into three parts due to the singularity $a$:
\[
(I)=\int_{\gamma_1}^{a-1/r}+\int_{a-1/r}^{a+1/r}+\int_{a+1/r}^{\gamma_2}\frac{x}{\sqrt{x-a}}\phi(x)e^{\i r x^2}\d x:=(i_1)+(ii)+(i_2).
\]
From direct computation,
\begin{align*}
|(ii)|\leq C\int_{a-1/r}^{a+1/r}\frac{1}{\sqrt{|x-a|}}\d x
=C\int_a^{a+1/r}\frac{1}{\sqrt{x-a}}\d x+C\int_{a-1/r}^a\frac{1}{\sqrt{a-x}}\d x=Cr^{-1/2}.
\end{align*}
For $(i_1)$, with integration by parts,
\begin{align*}
(i_1)&=\frac{1}{2\i r}\int_{\gamma_1}^{a-1/r}\frac{1}{\sqrt{x-a}}\phi(x)\d e^{\i r x^2}\\
&=\frac{1}{2\i r}\frac{1}{\sqrt{x-a}}\phi(x)\Big|^{a-1/r}_{\gamma_1}-\frac{1}{2\i r}\int_{\gamma_1}^{a-1/r}\frac{1}{(x-a)^{3/2}}\psi(x) e^{\i r x^2}\d x,
\end{align*} 
where
\[
\psi(x)=\frac{1}{2}\phi(x)-(x-a)\phi'(x)\in C^\infty_0(\R).
\]
Since $\phi(\gamma_1)=0$,
\[
\left|\frac{1}{2\i r}\frac{1}{\sqrt{x-a}}\phi(x)\Big|^{a-1/r}_{\gamma_1}\right|\leq Cr^{-1/2},
\]
and
\[
\left|\int_{\gamma_1}^{a-1/r}\frac{1}{(x-a)^{3/2}}\psi(x) e^{\i r x^2}\d x\right|\leq\int_{\gamma_1}^{a-1/r}\frac{1}{(x-a)^{3/2}}\d x\leq C\sqrt{r}.
\]
Thus we can easily obtain that
\[
|(i_1)|\leq Cr^{-1/2}.
\]
With similar technique, we can also get
\[
|(i_2)|\leq Cr^{-1/2}.
\]
Together with the estimation of (ii), 
\[
|(I)|\leq Cr^{-1/2}.
\]
The proof is then finished.
\end{proof}

\subsection{Analyse of the integrals of Class D}
\label{sec:intg4}

{\bf The first integral of Class D} is defined as
\begin{equation}
\label{eq:intd1}
I^D_1(r):=\int_{\gamma_1}^{\gamma_2}|x-a|\phi(x)e^{\i r x^2}\d x
\end{equation}
where $a\in[L_1,L_2]$ in a fixed bounded set. We first study the asymptotic behaviour of this integral when $r\rightarrow\infty$.

\begin{theorem}
\label{th:intd1}
When $\phi\in C_0^\infty(\gamma_1,\gamma_2)$ and $a\in[L_1,L_2]$, there is a constant $C>0$ only depends on $\phi$ such that
\begin{equation}
\label{eq:decay_intd1}
\left|I^D_1(r)\right|\leq Cr^{-1/2}.
\end{equation}
\end{theorem}

\begin{proof}We only need to consider the case that $a\in(\gamma_1,\gamma_2)$. When $a\not\in (\gamma_1,\gamma_2)$, it is just a simplified version of the case that $a\in(\gamma_1,\gamma_2)$ thus the details are omitted. 

Assume $a\in(\gamma_1,\gamma_2)$. 
Since the function $x\mapsto|x-a|$ is piecewise analytic,
\begin{align*}
I^D_1(r)&=\int_{\gamma_1}^a (a-t)\phi(t)e^{-2\i k r t^2}\d t+\int_a^{\gamma_2}(t-a)\phi(t)e^{-2\i kr t^2}\d t\\
&=a\int_{\gamma_1}^a\phi(t)e^{-2\i r t^2}\d t-a\int_a^{\gamma_2}\phi(t)e^{-2\i k rt^2}\d t-\int_{\gamma_1}^a t\phi(t)e^{-2\i k r t^2}\d t+\int_a^{\gamma_2}t\phi(t)e^{-2\i k rt^2}\d t\\
&:=a(I)-a(II)-(III)+(IV).
\end{align*}
For (I) (similar also for (II)), since $0$ becomes a singularity when we apply integration by parts, we split the integral as follows
\begin{align*}
(I)=\phi(0)\int_{\gamma_1}^a e^{-2\i r t^2}\d t+\int_{\gamma_1}^a t\psi(t)e^{-2\i k r t^2}\d t:=\phi(0)(I_1)+(I_2)
\end{align*}
where
\[
\psi(t)=\frac{\phi(t)-\phi(0)}{t}\in C^\infty(\gamma_1,a]\quad\text{ with }\psi(0)=\phi'(0).
\]
For $(I_1)$, let $s=\sqrt{2r}t$ with the Fresnel integral \eqref{eq:fresnel},
\[
|(I_1)|=\left|\frac{1}{\sqrt{2r}}\int_{\sqrt{2r}\gamma_1}^{\sqrt{2r}a}e^{-\i s^2}\d s\right|\leq Cr^{-1/2}.
\]
For $(I_2)$, with integration by parts,
\begin{align*}
|(I_2)|&=\left|-\frac{1}{4\i k r}\int_{\gamma_1}^a \psi(t)\d e^{-2\i k rt^2}\right|\\
&=\left|\left.-\frac{1}{4\i k r}\psi(t)e^{-2\i k rt^2}\right|^a_{\gamma_1}+\frac{1}{4\i k r}\int_{\gamma_1}^a\psi'(t)e^{-2\i k rt^2}\d t\right|\leq Cr^{-1}.
\end{align*}
Thus we can get the  estimations for both $(I)$ and $(II)$:
\[
|(I)|,\,|(II)|\leq Cr^{-1/2}.
\]

Similar as the arguments for $(I_2)$, we can also get the estimation for $(III)$ and $(IV)$:
\[
|(III)|,\,|(IV)|\leq Cr^{-1}.
\]
With all results for $(I)$, $(II)$, $(III)$ and $(IV)$, we can finaly finish the proof.
\end{proof}

Now we move on to {\bf the second integral of Class D}, which is defined by
\begin{equation}
\label{eq:intd2}
I^D_2(r):=\int_{\gamma_1}^{\gamma_2}{\rm sgn}(x-a)\phi(x)e^{\i r t^2}\d t.
\end{equation}
Note that here the sign function ${\rm sgn}$ is defined as 
\[
{\rm sgn}(x)=D|x|=\begin{cases}
1,\quad x>0;\\
0,\quad x=0;\\
-1,\quad x<0.
\end{cases}
\]
The decay of $I^D_2(r)$ is concluded as the corollary of Theorem \ref{th:intc1}.

\begin{corollary}
\label{cr:decay_intd2}
When $\phi\in C_0^\infty(\gamma_1,\gamma_2)$ and $a\in[L_1,L_2]$, there is a constant $C>0$ only depends on $\phi$ such that
\begin{equation}
\label{eq:decay_intd2}
\left|I^D_2(r)\right|\leq Cr^{-1/2}.
\end{equation}
\end{corollary}

\begin{proof}
Directly from the definition of the sign function,
\[
I^D_2(r)=-\int_{\gamma_1}^a\phi(x)e^{\i r t^2}\d t+\int_a^{\gamma_2}\phi(x)e^{\i r t^2}\d t
\]
where both functions has exactly the same forms of the terms (I) and (II) in the proof of Theorem \ref{th:intc1}. Thus the proof follows exactly as the proof of the previous theorem.
\end{proof}

\bibliographystyle{plain}
\bibliography{ip-biblio}
\end{document}